\documentclass[10pt,oneside,a4paper]{amsart}

\usepackage[english]{babel}
\usepackage{faktor} 
\usepackage{graphicx} 
\usepackage{amsmath,amsfonts,amsthm} 
\usepackage{mathtools} 
\usepackage{amssymb} 
\usepackage[inline]{enumitem} 
\usepackage{tabularx} 
\usepackage{tikz-cd} 
\usepackage{tikz} 
\usetikzlibrary{arrows} 
\usetikzlibrary{babel} 
\usepackage{adjustbox} 
\usepackage{hyperref}

\usepackage[colorinlistoftodos]{todonotes}
\usepackage{leftidx}
\usepackage{wrapfig}
\usepackage[margin=1in]{geometry}
\usepackage[utf8]{inputenc}
\usepackage{tikzit}

\tikzstyle{black dot}=[fill={rgb,255: red,255; green,191; blue,191}, draw=black, shape=circle]
\tikzstyle{circle vertex}=[fill=white, draw=black, shape=circle]
\tikzstyle{small dot}=[scale=0.75]
\tikzstyle{large circle}=[fill=white, draw=black, shape=circle, minimum size=4.75cm]
\tikzstyle{blue small dot}=[fill=white, draw=blue, shape=circle]
\tikzstyle{full blue}=[fill=blue, draw=blue, shape=circle]

\tikzstyle{dashed grey}=[-, dashed, fill={rgb,255: red,191; green,191; blue,191}]
\tikzstyle{arrow}=[->]
\tikzstyle{none white fill}=[-, fill=white]
\tikzstyle{fill grey}=[-, fill={rgb,255: red,191; green,191; blue,191}]
\tikzstyle{dashed line}=[-, dashed]
\tikzstyle{right_arrow}=[<-]
\tikzstyle{blue arrow}=[draw=blue, <->]
\tikzstyle{double arrow}=[<->]
\tikzstyle{mapsto}=[{|->}]
\tikzstyle{blue edge}=[-, draw=blue]
\tikzstyle{blue dashed arrow}=[draw=blue, ->, dashed]
\tikzstyle{red line}=[-, draw=red]
\tikzstyle{red dashed arrow}=[dashed, ->, draw=red]
\tikzstyle{blue arrow}=[draw=blue, ->]
\tikzstyle{green dashed}=[-, dashed, draw=green]
\tikzstyle{dashed arrow}=[->, dashed]
\tikzstyle{fill grey no line}=[-, fill={rgb,255: red,191; green,191; blue,191}, draw=none]
\tikzstyle{new edge style 0}=[<->]
\tikzstyle{grey line}=[-, draw={rgb,255: red,191; green,191; blue,191}]
\tikzstyle{fill green}=[-, fill=green, draw=none]
\tikzstyle{fill red}=[-, draw=none, fill=red]

\newtheorem{theorem}{Theorem}[section]
\newtheorem{lemma}[theorem]{Lemma}
\newtheorem{prop}[theorem]{Proposition}
\newtheorem{cor}[theorem]{Corollary}
\theoremstyle{definition}
	\newtheorem{mydef}[theorem]{Definition}
    \newtheorem{constr}[theorem]{Construction}

\theoremstyle{remark}
	\newtheorem{opm}[theorem]{Remark}
    
    \newtheorem{vb}[theorem]{Example}

\newcommand{\thistheoremname}{}
\newtheorem*{genericthm*}{\thistheoremname}
\newenvironment{namedthm*}[1]
  {\renewcommand{\thistheoremname}{#1}%
   \begin{genericthm*}}
  {\end{genericthm*}}

\theoremstyle{definition}

\newcommand{\thisexname}{}
\newtheorem*{genericex*}{\thisexname}
\newenvironment{namedex*}[1]
  {\renewcommand{\thisexname}{#1}%
   \begin{genericex*}}
  {\end{genericex*}}

\newcommand{\inv}{^{-1}}

\newcommand{\sub}{\subseteq}

\newcommand{\N}{\mathbb{N}}
\newcommand{\Z}{\mathbb{Z}}

\newcommand{\Oo}{\mathcal{O}}

\newcommand{\C}{\mathcal{C}}
\newcommand{\U}{\mathcal{U}}

\newcommand{\eps}{\epsilon}

\newcommand{\A}{\mathcal{A}}
\newcommand{\si}{\sigma}
\newcommand{\te}{\theta}
\newcommand{\al}{\alpha}

\newcommand{\be}{\beta}
\newcommand{\ga}{\gamma}

\makeatletter
\newcommand{\colim@}[2]{%
  \vtop{\m@th\ialign{##\cr
    \hfil$#1\operator@font colim$\hfil\cr
    \noalign{\nointerlineskip\kern1.5\ex@}#2\cr
    \noalign{\nointerlineskip\kern-\ex@}\cr}}%
}
\newcommand{\dcolim}{%
  \mathop{\mathpalette\colim@{\rightarrowfill@\textstyle}}\nmlimits@
}
\makeatother

\DeclareMathOperator{\Hom}{Hom}
\DeclareMathOperator{\Ob}{Ob}
\DeclareMathOperator{\Fun}{Fun}
\DeclareMathOperator{\Image}{Im}
\DeclareMathOperator{\sgn}{sgn}
\DeclareMathOperator{\Clr}{Clr}

\DeclareMathOperator{\Tree}{Tree}

\DeclareMathOperator{\Multitr}{Multi\Delta}
\DeclareMathOperator{\Multitrp}{Multi\Delta_{+}}
\DeclareMathOperator{\FS}{F_{2}S}
\DeclareMathOperator{\NSOp}{NSOp}
\DeclareMathOperator{\mNSOp}{mNSOp}
\DeclareMathOperator{\mmNSOp}{(m)NSOp}
\DeclareMathOperator{\Patch}{Patch}

\DeclareMathOperator{\Quilt}{Quilt}
\DeclareMathOperator{\mQuilt}{mQuilt}

\DeclareMathOperator{\Quiltb}{Quilt_{b}}

\DeclareMathOperator{\Bracee}{Brace}

\newcommand{\End}{\textbf{End}}

\newcommand{\CGS}{\textbf{C}_{GS}}
\newcommand{\CGSnr}{\overline{\textbf{C}}_{GS}}
\newcommand{\tre}{\vartriangleleft}
\newcommand{\treq}{\unlhd}

\newcommand{\Ss}{\mathbb{S}}

\newcommand{\fromto}[2]{#1_{1},\ldots,#1_{#2}}
\newcommand{\nth}[1]{#1_{1},\ldots,#1_{n}}

\newcommand{\rh}{\rangle}
\newcommand{\lh}{\langle}
\newcommand{\st}{^{*}}
\newcommand{\bu}{\bullet}

\newcommand{\hotimes}{\otimes_{H}}
\newcommand{\Linf}{L_{\infty}}

\makeatletter
\newcommand{\proofpart}[2]{%
  \par
  \addvspace{\medskipamount}%
  \noindent\emph{Step #1: #2 }\par\nobreak
  \addvspace{\smallskipamount}%
  \@afterheading
}
\makeatother

\newcommand{\Te}{\Theta}

\newcommand{\LL}{\mathcal{L}}

\newcommand{\hs}{^{\#}}

\title{Operadic structure on the Gerstenhaber-Schack complex for prestacks}	

\author{Hoang Dinh Van}
\address[Hoang Dinh Van]{Faculty of Applied Sciences, University of Technology and Education, 1 Vo Van Ngan, Thu Duc, Ho Chi Minh city, Vietnam.}
\email{hoangdv@hcmute.edu.vn}

\author{Lander Hermans}
\address[Lander Hermans]{Universiteit Antwerpen, Departement Wiskunde, Middelheimcampus,
	Middelheimlaan 1,
	2020 Antwerp, Belgium}
\email{lander.hermans@uantwerpen.be}

\author{Wendy Lowen} 
\address[Wendy Lowen]{Universiteit Antwerpen, Departement Wiskunde-Informatica, Middelheimcampus,
Middelheimlaan 1,
2020 Antwerp, Belgium (main affiliation)}
\address{Laboratory of Algebraic Geometry, National Research University Higher School of Economics, Russian Federation}
\email{wendy.lowen@uantwerpen.be}

\thanks{This project has received funding from the European Research Council (ERC) under the European Union's Horizon 2020 research and innovation programme (grant agreement No. 817762). The second named auther holds a PhD fellowship of the Research Foundation - Flanders (FWO). The third named author was partially supported by the HSE University Basic Research Program.
}

\begin{document}

\begin{abstract}

We introduce an operad $\Patch$ which acts on the Gerstenhaber-Schack complex of a prestack as defined by Dinh Van and Lowen, and which in particular allows us to endow this complex with an underlying $L_{\infty}$-structure. We make use of the operad $\Quilt$ which was used by Hawkins in order to solve the presheaf case. Due to the additional difficulty posed by the presence of twists, we have to use $\Quilt$ in a fundamentally different way (even for presheaves) in order to allow for an extension to prestacks. 
\end{abstract}

\maketitle

\section{Introduction}

The deformation theory of algebras due to Gerstenhaber furnishes the guiding example for algebraic deformation theory. For an algebra $A$, the Hochschild complex $\mathbf{C}(A)$ is a dg Lie algebra governing the deformation theory of $A$ through the Maurer-Cartan formalism. This dg Lie structure is the shadow of a richer operadic structure, which can be expressed by saying that $\mathbf{C}(A)$ is a homotopy G-algebra \cite{gerstenhabervoronov}. This structure, which captures both the brace operations and the cup product, is a special case of a $B_{\infty}$-structure \cite{getzlerjones}. Importantly, this purely algebraic structure constitutes a stepping stone in the proof of the Deligne conjecture, proving $\mathbf{C}(A)$ to be an algebra over the chain little disk operad \cite{mccluresmith} \cite{kontsevichsoibelmandeligne}.

The deformation theory of algebras was later extended to presheaves of algebras by Gerstenhaber and Schack, who in particular introduced a bicomplex computing the natural bimodule Ext groups \cite{gerstenhaberschack}, \cite{gerstenhaberschack1}. However, this GS-complex $\mathbf{C}(\A)$ of a presheaf $\A$ does not control deformations of $\A$ as a presheaf, but rather as a \emph{twisted} presheaf, see for instance \cite{lowenalg}, \cite{DVL}. From this point of view, it is more natural to develop deformation theory at once on the level of twisted presheaves or, more generally prestacks, that is, pseudofunctors taking values in the 2-category of linear categories (over some fixed commutative ground ring). In \cite{DVL}, Dinh Van and Lowen established a Gerstenhaber-Schack complex for prestacks, involving a differential which features an infinite sequence of higher components in addition to the classical simplicial and Hochshild differentials. Further, for a prestack $\A$, they construct a homotopy equivalence $\CGS(\A) \cong \mathbf{CC}(\A!)$ between the Gerstenhaber-Schack complex $\CGS(\A)$ and the Hochschild complex $\mathbf{CC}(\A!)$ of the Grothendieck construction $\A!$ of $\A$. Through homotopy transfer, this endows the GS-complex with an $L_{\infty}$-structure. This result improves upon the existence of a quasi-isomorphism, which is a consequence of the Cohomology Comparison Theorem due to Gerstenhaber and Schack for presheaves \cite{gerstenhaberschack1} and to Lowen and Van den Bergh for prestacks \cite{lowenvandenberghCCT}.

Although the GS-complex does not possess a $B_{\infty}$-structure, its elements - linear maps involving different levels of the prestack - can be composed in an operadic fashion. As such, it makes sense to investigate this higher structure in its own right, and use it directly in order to establish an underlying $L_{\infty}$-structure.
For particular types of presheaves, explicit $L_{\infty}$-structures on the GS-complex have been established by Fr{\'e}gier, Markl and Yau in \cite{FMY09} and by Barmeier and Fr{\'e}gier in \cite{barmeierfregier}.

Let $\Bracee$ be the brace operad and $\FS$ the homotopy G-operad.
In the case of a presheaf $(\A, m, f)$, in \cite{hawkins}, Hawkins introduces an operad $\Quilt \subseteq \FS \otimes_H \Bracee$ which he later extends to an operad $\mQuilt$ acting on the GS-complex. These operads are naturally endowed with $L_{\infty}$-operations as desired.
The action of $\Quilt$ on the GS-complex considered by Hawkins only involves the restriction functors $f$ of the presheaf, the multiplication $m$ being incorporated later on in $\mQuilt$. Unfortunately, the way in which functoriality of $f$ is built into these actions, does not allow for an extension to twisted presheaves or prestacks. 

The goal of this paper is to solve the problem of establishing a natural operadic structure with underlying $L_{\infty}$-structure on $\CGS(\A)$ in the case of a general prestack $(\A, m, f, c)$ with twists $c$.
As part of our solution, we use $\Quilt$ in a fundamentally different way in relation to the GS-complex, but still allowing us to make use of the naturally associated $L_{\infty}$-structure from \cite{hawkins}.
 In section \S \ref{parparGS}, we capture the higher structure of $\CGS(\A)$ by introducing the new operad $\Patch \subseteq \mNSOp \otimes_H \NSOp$ over which $\CGS(\A)$ is shown to be an algebra (see Theorem \ref{thmPGS}). Here, $\mmNSOp$ is the operad of nonsymmetric operads (with multiplication).
 
In \cite{gerstenhabervoronov}, Gerstenhaber and Voronov obtain a brace algebra structure on an operad and a homotopy G-algebra structure on an operad with multiplication. Based upon the expression of these results in terms of the underlying operads $\NSOp$ and $\mNSOp$ in \S \ref{parparGV}, we construct a morphism $\Quilt \longrightarrow \Patch_s$ (see Proposition \ref{propQP}) as a restriction of
$$\FS \otimes_H \Bracee \longrightarrow \mNSOp_{st} \otimes_H \NSOp_s,$$
where the operads with subscript denote the (uncolored) graded operads associated to the unsubscripted colored operads. This gives rise to the composition
$$R: \Quilt \longrightarrow \Patch_s \longrightarrow \End(s\CGS(\A))$$
which incorporates the multiplication $m$ and the restrictions $f$ of $\A$.

In \S \ref{parc}, we extend the action $R$ to
$$R_c: \Quiltb[[c]] \longrightarrow \End(s\CGS(\A))$$
in order to incorporate the twists (see Theorem \ref{thmRc}). Here, $\Quiltb[[c]]$ is obtained from an operad of formal power series. Further, we establish $L_{\infty}$-operations on $\Quiltb[[c]]$ extending those on $\Quilt$ from \cite{hawkins} (see Theorem \ref{maintheorem}) by adding an infinite series of higher components containing twists. Under the action of $\Quilt$ this neatly corresponds to and extends the differential on $\CGS(\A)$ obtained in \cite{DVL}. In the final section \ref{pardef} we briefly discuss the relation of this $L_{\infty}$-structure with the deformation theory of the prestack $\A$.

The present work naturally grew out of \cite{DVL}, and at the time when \cite{hawkins} appeared large parts of an operadic approach to the GS complex of a prestack had already been developed independently by us. Given the efficient way in which Hawkins' description of $\Quilt$ gives rise to an $L_{\infty}$-structure, we decided it was worthwhile to build on this approach to the presheaf case, albeit in a way which ``flips and refines'' the action of $\Quilt$ in order to make it useful for general prestacks. As a consequence, when we follow through Hawkins' approach, in comparison we manage to incorporate not only the restrictions $f$, but also the multiplications $m$ in an initial action of $\Quilt$ on the GS complex. In analogy with the way in which Hawkins extends his action from $\Quilt$ to $\mathrm{mQuilt}$ in order to incorporate the multiplications $m$, we establish an extension from $\Quilt$ to $\Quiltb[[c]]$ in order to incorporate the twists $c$.

The current paper is part of a larger project in which it is our goal to understand the homotopy equivalence $\CGS(\A) \cong \mathbf{CC}(\A!)$ from \cite{DVL} operadically, showing in particular that the $L_{\infty}$-structure from \cite{DVL} and the one established in the present paper actually coincide.

\vspace{0,3cm}
\noindent \emph{Acknowledgement.} 
The authors are grateful to an anonymous referee for their meticulous reading of an earlier version of the manuscript, and in particular for several corrections and valuable suggestions which greatly helped improve the paper. The third named author thanks Severin Barmeier and Eli Hawkins for bringing the preprints \cite{barmeierfregier}, respectively \cite{hawkins} to her attention.

\section{Gerstenhaber-Voronov operadically}\label{parparGV}

In the seminal paper  \cite{gerstenhabervoronov}, Gerstenhaber and Voronov define a brace-algebra structure on the totalisation of a non-symmetric operad $\Oo$. Moreover, in presence of a multiplication, they define a homotopy G-algebra structure on $\Oo$ incorporating both the cup product and the Gerstenhaber-bracket. 

In this section we describe the morphisms of operads underlying these results. To this end, in \S \ref{parNSOp}, we recall the colored operad $\NSOp$ encoding non-symmetric operads, and we describe the natural extension $\mNSOp$ which adds a multiplication. Let $\NSOp_s$ and $\mNSOp_{st}$ be their totalised graded (uncolored) operads with suspended, respectively standard degree ( see \S \ref{parbraceNSOp} and \S \ref{parFSmNSOp}). Let $\Bracee$ be the brace operad (see \S \ref{parbrace}) and $\FS$ the Gerstenhaber-Voronov operad encoding homotopy G-algebras (see \S \ref{parFS}). The main goal of this section is the definition of morphisms of dg-operads
$$\phi:\Bracee \longrightarrow \NSOp_s$$
and
$$\bar{\phi}:\FS \longrightarrow \mNSOp_{st}$$
(see Theorems \ref{thmbraceNSOp} and \ref{thmFSmNSOp} respectively). In these definitions, we have to pay particular attention to the choice of signs. For this, we will make use of morphisms of operads $\mmNSOp \longrightarrow \Multitr$ landing in the multicategory associated to the simplex category $\Delta$ (see Proposition \ref{propmapmulti}).

For both uncolored as colored operads, we use the term morphism of operads. In case confusion may arise, we add a subscript to differentiate the uncolored operads from their colored counterparts.

\subsection{The operad $\Bracee$}\label{parbrace}

Throughout, we work over a fixed commutative ground ring $k$. 

The operad $\Bracee$ encoding brace algebras is defined using trees, that is, planar rooted trees. Following the presentation from \cite[\S 2.2]{hawkins}, for a tree $T$ we denote the set of vertices by $V_T$, the set of edges by $E_T$, the ``vertical'' partial order on $V_T$ generated by $E_T$ by $\leq_T$, and the ``horizontal'' partial order on $V_T$ by $\treq_{T}$. 
For $(u,v) \in E_{T}$ we call $u$ the \textit{parent} of $v$ and $v$ a \textit{child} of $u$.

For $n\in \N$, put $[n] := \{ 0,\ldots, n\}$ and $\langle n \rangle := \{ 1,\ldots ,n \}$.

Let $\Tree(n)$ denote the set of trees with vertex set $\langle n \rangle$ and let $\Bracee(n)$ be the free $k$-module on $\Tree(n)$ endowed with the $\Ss_{n}$-action given by permuting the vertices, i.e., $T^{\si}$ is the tree defined by replacing vertex $i$ in $T$ by $\si \inv (i)$.
The operadic composition on $\Bracee$ is based upon substitution of trees, as follows. For trees $T \in \Tree(m)$, $T' \in \Tree(n)$ and $1 \leq i \leq m$, we denote by $Ext(T,T',i) \subseteq \Tree(m + n  -1)$ the set of trees extending $T$ by $T'$ at $i$ (that is, $U \in Ext(T,T',i)$ has $T'$ as a subtree which upon removal reduces to the vertex $i$ of $T$). We then define
$$T \circ_{i} T' := \sum_{U \in Ext(T,T',i)} U.$$
Underlying every such extension lie two maps $\lh n \rh \overset{\al}{\hookrightarrow} \lh n+m-1\rh \overset{\be}{\twoheadrightarrow} \lh m \rh$ acting on the vertices, where $\al$ embeds $n$ vertices consecutively and $\be$ contracts the image of $\al$ to the vertex $i$. We call the pair $(\al,\be)$ the \textit{extension of $m$ by $n$ at $i$}.
We refer to \cite[\S 2.2]{hawkins} for more details.

\subsection{The operad $\FS$}\label{parFS}

The operad $\FS$ encodes homotopy G-algebras \cite{gerstenhabervoronov}. Again, we largely follow the exposition from \cite[\S 2.3]{hawkins}. 
Given a set $A$, a \emph{word} over $A$ is an element of the free monoid on $A$. For a word $W = a_1a_2 \dots a_k$, correspondiong to the function $W: \langle k \rangle \longrightarrow A: i \longmapsto a_i$, the $i$-th \emph{letter} of $W$ is the couple $(i, a_i)$. We will often identify a word with its graph $W = \{ (i, a_i) \,\, |\,\, i \in \langle k \rangle\} \subseteq \langle k \rangle \times A$, writing $(i, a_i) \in W$.

 For $a \in A$, a letter $(i,a) \in W$ is called an \emph{occurrence} of $a$ in $W$. The letter $(i,a)$ is a \emph{caesura} if there is a later occurrence of $a$ in $W$, that is, a letter $(j,a)$ with $i < j$. We say that $a \in A$ is \emph{interposed} in $W$ if $W = \dots ba \dots b \dots$. The \emph{length} of $W: \langle k \rangle \longrightarrow A$ is $|W| = k$.

By definition, $\FS(n)$ is the free $k$-module generated by the words $W$ over $\langle n \rangle$ such that:
\begin{enumerate}
\item $W: \langle k \rangle \longrightarrow \langle n \rangle$ is surjective,
\item $W \neq \ldots {u}{u} \ldots$ (nondegeneracy), and
\item for any $u\neq v \in \langle n \rangle$, $W \neq \ldots {u} \ldots {v} \ldots {u} \ldots {v} \ldots$ (no interlacing).
\end{enumerate}
The set $\FS(n)$ is graded by setting $deg(W) := |W| - n$ and naturally carries a $\Ss_{n}$-action by permuting letters, i.e. $W^{\si} = \si\inv W$. 

For a word $W\in \FS(n)$ and $u \in \langle n \rangle$, let $(i_u, u)$ be the first occurrence of $u$ in $W$. Then we obtain a total order $u \downarrow v \iff i_u \leq i_v$ on $\lh n \rh$.

The operadic composition on $\FS$ is based upon merging of words, as follows. For words $W\in \FS(m), W' \in \FS(n)$ and $1 \leq i \leq m$, we denote by $Ext(W,W',i) \subseteq \FS(m+n-1)$ the set of extensions of $W$ by $W'$ at $i$ (that is, $X \in Ext(W,W',i)$ if up to relabelling and deleting repetitions, $W'$ is a subword of $X$ and upon collapsing the letters from $W$ to $i$, relabelling and deleting repetitions, we recover $W$). 

In order to define the composition, we need the sign of an extension. 
\medskip

\noindent \emph{Sign of Extension.}
Let $W\in \FS(m)$ and let $int(W)$ be the set of interposed elements of $\langle m \rangle$ ordered by their first occurrence in $W$. For $X \in Ext(W,W',i)$ the relabelling gives rise to two maps $\al:int(W')\longrightarrow int(X)$ and $\ga: int(W) \longrightarrow int(X)$ where $\ga := \be\inv$ except if $i$ is interposed in $W$, then $\gamma(i) := \al(a)$ for $(1,a)$ the first letter of $W'$.
 
As $|int(W)| = deg(W)$, an extension $X$ defines a unique $(deg(W),deg(W'))-$shuffle $\chi$ and we define 
$$\sgn_{W,W',i}(X) := (-1)^{\chi}$$

Moreover, it is possible to talk about the boundary of a word, inducing a differential.

\medskip

\noindent \emph{Boundary.}
Given a word $W \in \FS(n)$ and a letter $(i,a)$ of $W$ for which $a$ is repeated in $W$, then define $\partial_{i}W\in \FS(n)$ as the word obtained by deleting the letter $(i,a)$ from $W$ (and relabelling). If $a$ is not repeated, then set $\partial_{i}W= 0$.

\medskip

\noindent \emph{Sign of Deletion.}
Given a word $W \in \FS(n)$ of length $k$, then we define $\sgn_{W}:\langle k \rangle \longrightarrow \{-1,1\}$ by setting $\sgn_{W}(i)= (-1)^{k}$ if $(i,a_i)$ is the $k$-th caesura of $W$, and otherwise $\sgn_{W}(i)= (-1)^{k+1}$ if it is the last occurrence, but the previous occurrence is the $k$-th caesura of $W$.

\medskip

The $\Ss$-module $\FS$ defines a dg-operad with operadic composition given by 
$$ W \circ_{i} W' := \sum_{X \in Ext(W,W',i)} \sgn_{W,W',i}(X) X$$
and boundary given by
$$\partial W := \sum_{i \in \langle |W| \rangle} \sgn_{W}(i)\partial_{i}W$$

The following lemma, which we include for the convenience of the reader, shows how $\FS$ encodes the algebraic operations of a homotopy G-algebra.

\medskip

\noindent \emph{Notations.}
To avoid too large expressions, we leave out certain bracketings by setting as default the bracketing

$$a \circ_{i} b \circ_{j} c := (a \circ_{i} b)\circ_{j} c$$
Moreover, we compress the following
$$ a (\circ_{i_{t}}a_{t})_{t} := a \circ_{i_{1}} a_{1} \circ_{i_{2}} \hdots \circ_{i_{n}}a_{n}  $$

\medskip 

\begin{lemma}\label{GeneratorsFS}
Let $M_{2}:= 12, M_{1,0} = 1$ and $M_{1,k} := 121\ldots1(k+1)
1$ for $k\geq 1$, then $\FS$ is generated by these elements and the following holds
\begin{enumerate}
\item $\partial(M_{2}) = 0$
\item $\partial(M_{1,k}) = -(M_{2}^{(12)}\circ_{2}M_{1,k-1}) + \sum_{i=2}^{k} (-1)^{i} M_{1,k-1} \circ_{i} M_{2} + (-1)^{k+1} M_{2} \circ_{1} M_{1,k-1}$
\end{enumerate}
\end{lemma}
\begin{proof}
It is a straightforward computation to determine that $M_{2}$ and $M_{1,k}$ satisfy these relations.

Let $W\in \FS(n)$, we then show that it lies in the suboperad generated by $M_{2},(M_{1,k})_{k\geq1}$ using only the above relations. We prove this by induction on $n$. If $n=1$, then $W={1}$. So assume $n>1$ and apply a permutation such that the first letter of $W$ is $1$, then $W$ is of the form
$$W = {1}W_{1}{1} \ldots {1} W_{k} {1} W_{k+1}$$
where $W_{i}$ is the image of a non-empty word $W_{i}'\in \FS(n_{i})$ under the map $\ga_{i} : \lh n_{i} \rh \longrightarrow \lh n \rh$, except $W_{k+1}$ which is possibly empty. Due to no interlacing we also know that the images $\Image(\ga_{i})$ are pair-wise disjoint. Hence, we can apply a permutation to assume that $\max \Image(\ga_{i}) < \min \Image(\ga_{i'})$ holds for every $i < i'$. In this case, we have that 
$$W = M_{1,k} \circ_{k+1} W'_{k} \circ_{k} \ldots  \circ_{2} W_{1}'$$
if $W_{k+1} = \emptyset$, and
$$W = (M_{2} \circ_{2} W'_{k+1}) \circ_{1} ( M_{1,k} \circ_{k+1} W'_{k} \circ_{k} \ldots \circ_{2} W'_{1})$$
otherwise. By induction, this shows that $W$ is generated by $M_{2}$ and $(M_{1,k})_{l\geq0}$.
\end{proof}

\subsection{The operads $\NSOp$ and $\mNSOp$}\label{parNSOp}

It is well-known that non-symmetric operads can be encoded using a colored operad $\NSOp$ which can be defined using indexed trees, that is, for $\nth{q} \in \N$ and $q'= 1+ \sum_{i=1}^{n} (q_{i}-1)$,  $\NSOp(\nth{q};q')$ is the set of pairs $(T,I)$ where $T \in \Tree(n)$ and $I:E_{T}\longrightarrow \N$ a function such that
\begin{itemize}
\item for $(u,v)\in E_{T}$, $1\leq I(u,v) \leq q_{u}$
\item $(t,u),(t,v)\in E_{T}$ and $u \tre_{T}v \Longrightarrow I(t,u) < I(t,v)$
\end{itemize} 
We will often write $I$ to denote the indexed tree $(T,I)$. Moreover, $\NSOp$ is generated by those trees with a single edge, that is,
$$ E_i:= \quad \begin{tikzpicture}
	\begin{pgfonlayer}{nodelayer}
		\node [style=none] (0) at (0, -1) {$1$};
		\node [style=none] (1) at (0, 1) {$2$};
		\node [style=none] (2) at (0, 0.525) {};
		\node [style=none] (3) at (0, -0.55) {};
		\node [style=small dot] (4) at (0.4, 0) {$i$};
	\end{pgfonlayer}
	\begin{pgfonlayer}{edgelayer}
		\draw (2.center) to (3.center);
	\end{pgfonlayer}
\end{tikzpicture}
 \quad \in \NSOp(q_1,q_2;q_1+q_2-1)$$
for every $q_1,q_2$ and $1 \leq i \leq q_1$,
with the following pair of relations
\begin{enumerate}[label=(\Roman*)]
\item \quad $\begin{tikzpicture}
	\begin{pgfonlayer}{nodelayer}
		\node [style=none] (0) at (0, -1) {$1$};
		\node [style=none] (1) at (0, 1) {$2$};
		\node [style=none] (2) at (0, 0.525) {};
		\node [style=none] (3) at (0, -0.55) {};
		\node [style=small dot] (4) at (0.4, 0) {$i$};
	\end{pgfonlayer}
	\begin{pgfonlayer}{edgelayer}
		\draw (2.center) to (3.center);
	\end{pgfonlayer}
\end{tikzpicture}
\quad  \circ_2 \quad \begin{tikzpicture}
	\begin{pgfonlayer}{nodelayer}
		\node [style=none] (0) at (0, -1) {$1$};
		\node [style=none] (1) at (0, 1) {$2$};
		\node [style=none] (2) at (0, 0.525) {};
		\node [style=none] (3) at (0, -0.55) {};
		\node [style=small dot] (4) at (0.4, 0) {$j$};
	\end{pgfonlayer}
	\begin{pgfonlayer}{edgelayer}
		\draw (2.center) to (3.center);
	\end{pgfonlayer}
\end{tikzpicture}
 \quad = \quad \input{3Tree_1.tikz} \quad =  \quad \begin{tikzpicture}
	\begin{pgfonlayer}{nodelayer}
		\node [style=none] (0) at (0, -1) {$1$};
		\node [style=none] (1) at (0, 1) {$2$};
		\node [style=none] (2) at (0, 0.525) {};
		\node [style=none] (3) at (0, -0.55) {};
		\node [style=small dot] (4) at (1.3, 0) {$i-1+j$};
	\end{pgfonlayer}
	\begin{pgfonlayer}{edgelayer}
		\draw (2.center) to (3.center);
	\end{pgfonlayer}
\end{tikzpicture}
 \quad \circ_1 \quad \begin{tikzpicture}
	\begin{pgfonlayer}{nodelayer}
		\node [style=none] (0) at (0, -1) {$1$};
		\node [style=none] (1) at (0, 1) {$2$};
		\node [style=none] (2) at (0, 0.525) {};
		\node [style=none] (3) at (0, -0.55) {};
		\node [style=small dot] (4) at (0.4, 0) {$i$};
	\end{pgfonlayer}
	\begin{pgfonlayer}{edgelayer}
		\draw (2.center) to (3.center);
	\end{pgfonlayer}
\end{tikzpicture}
 \quad $ for $1 \leq j \leq q_2$ and $1 \leq i \leq q_1$,\medskip \label{NSOpRel_I}
\item \quad $\begin{tikzpicture}
	\begin{pgfonlayer}{nodelayer}
		\node [style=none] (0) at (0, -1) {$1$};
		\node [style=none] (1) at (0, 1) {$2$};
		\node [style=none] (2) at (0, 0.525) {};
		\node [style=none] (3) at (0, -0.55) {};
		\node [style=small dot] (4) at (0.4, 0) {$i$};
	\end{pgfonlayer}
	\begin{pgfonlayer}{edgelayer}
		\draw (2.center) to (3.center);
	\end{pgfonlayer}
\end{tikzpicture}
\quad  \circ_1 \quad \begin{tikzpicture}
	\begin{pgfonlayer}{nodelayer}
		\node [style=none] (0) at (0, -1) {$1$};
		\node [style=none] (1) at (0, 1) {$2$};
		\node [style=none] (2) at (0, 0.525) {};
		\node [style=none] (3) at (0, -0.55) {};
		\node [style=small dot] (4) at (0.4, 0) {$k$};
	\end{pgfonlayer}
	\begin{pgfonlayer}{edgelayer}
		\draw (2.center) to (3.center);
	\end{pgfonlayer}
\end{tikzpicture}
 \quad = \quad \input{3Tree_2.tikz} \quad = \left( \quad \begin{tikzpicture}
	\begin{pgfonlayer}{nodelayer}
		\node [style=none] (0) at (0, -1) {$1$};
		\node [style=none] (1) at (0, 1) {$2$};
		\node [style=none] (2) at (0, 0.525) {};
		\node [style=none] (3) at (0, -0.55) {};
		\node [style=small dot] (4) at (1.425, 0) {$k-1+q_2$};
	\end{pgfonlayer}
	\begin{pgfonlayer}{edgelayer}
		\draw (2.center) to (3.center);
	\end{pgfonlayer}
\end{tikzpicture}
 \quad \circ_1 \quad \begin{tikzpicture}
	\begin{pgfonlayer}{nodelayer}
		\node [style=none] (0) at (0, -1) {$1$};
		\node [style=none] (1) at (0, 1) {$2$};
		\node [style=none] (2) at (0, 0.525) {};
		\node [style=none] (3) at (0, -0.55) {};
		\node [style=small dot] (4) at (0.4, 0) {$i$};
	\end{pgfonlayer}
	\begin{pgfonlayer}{edgelayer}
		\draw (2.center) to (3.center);
	\end{pgfonlayer}
\end{tikzpicture}
  \quad \right)^{(23)}$ for $1\leq i<k \leq q_1$.\label{NSOpRel_II}
\end{enumerate}
Note that these are the well-known associativity relations for non-symmetric operads.
\begin{mydef}
Let $\mNSOp$ be the $\N$-colored operad generated by $\NSOp$ and an element $m\in \mNSOp(;2)$ satisfying the relation 
$$  \begin{tikzpicture}
	\begin{pgfonlayer}{nodelayer}
		\node [style=none] (0) at (0, -1) {$1$};
		\node [style=none] (1) at (0, 1) {$2$};
		\node [style=none] (2) at (0, 0.525) {};
		\node [style=none] (3) at (0, -0.55) {};
		\node [style=small dot] (4) at (0.4, 0) {$1$};
	\end{pgfonlayer}
	\begin{pgfonlayer}{edgelayer}
		\draw (2.center) to (3.center);
	\end{pgfonlayer}
\end{tikzpicture}
 \circ_{1} m \circ_{1} m = \begin{tikzpicture}
	\begin{pgfonlayer}{nodelayer}
		\node [style=none] (0) at (0, -1) {$1$};
		\node [style=none] (1) at (0, 1) {$2$};
		\node [style=none] (2) at (0, 0.525) {};
		\node [style=none] (3) at (0, -0.55) {};
		\node [style=small dot] (4) at (0.4, 0) {$2$};
	\end{pgfonlayer}
	\begin{pgfonlayer}{edgelayer}
		\draw (2.center) to (3.center);
	\end{pgfonlayer}
\end{tikzpicture}
 \circ_{1} m \circ_{1} m$$
\end{mydef}
\begin{opm}
We often write $\begin{tikzpicture}
	\begin{pgfonlayer}{nodelayer}
		\node [style=none] (0) at (0, -1) {$m$};
		\node [style=none] (1) at (0, 1) {$m$};
		\node [style=none] (2) at (0, 0.525) {};
		\node [style=none] (3) at (0, -0.55) {};
		\node [style=small dot] (4) at (0.4, 0) {$1$};
	\end{pgfonlayer}
	\begin{pgfonlayer}{edgelayer}
		\draw (2.center) to (3.center);
	\end{pgfonlayer}
\end{tikzpicture}
$ where we have already filled in the plugged in $m$'s.
\end{opm}

More explicitly, every representative of an element $X\in \mNSOp(\nth{q};q)$ is of the form 
$I \circ_{i_{1}} m \circ_{i_{2}} \ldots \circ_{i_{k}} m$
 for $I \in \NSOp$ and appropriate $i_{1},\ldots,i_{k} \in \N$. Due to equivariance, we can always consider a representative of $X$ of the form
 $$I \circ_{n+1} m \circ_{n+1} \ldots \circ_{n+1} m$$
 for $I \in \NSOp(\nth{q},2,\ldots,2;q)$.

\begin{lemma}\label{partialorders}
Let $X = [ I \circ_{n+1} m \circ_{n+1} \ldots \circ_{n+1} m]  \in \mNSOp(\nth{q};q)$, the partial orders $<_{I}$ and $\tre_{I}$ on $\lh n \rh$ are independent of the representative of $X$. We denote them by $<_{X}$ and $\tre_{X}$.
\end{lemma} 
\begin{proof}
We proceed by induction on $k$ the number of $m$'s in $X$. For $k =0$ or $k=1$, there is nothing to show, so assume $k >1$. It is clear that if the lemma holds for $X$, then the relations that hold for $<$ and $\tre$ for trees, also hold for $<_{X}$ and $\tre_{X}$. In particular, if the lemma holds for $X$ and $X'$ and $(\al,\be)$ is the extension of $n$ by $m$ at $i$, then for $a,b \notin \Image(\al)$ we have 
$$ a <_{X\circ_{i} X'} b \iff \be a <_{X} \be b \text{ and } a \tre_{X \circ_{i} X'} b \iff \be a \tre_{X} \be b$$

Now, let $X_{0,i} := X_{0} \circ_{n+1} \; \begin{tikzpicture}
	\begin{pgfonlayer}{nodelayer}
		\node [style=none] (0) at (0, -1) {$1$};
		\node [style=none] (1) at (0, 1) {$2$};
		\node [style=none] (2) at (0, 0.525) {};
		\node [style=none] (3) at (0, -0.55) {};
		\node [style=small dot] (4) at (0.4, 0) {$i$};
	\end{pgfonlayer}
	\begin{pgfonlayer}{edgelayer}
		\draw (2.center) to (3.center);
	\end{pgfonlayer}
\end{tikzpicture}
 $ such that $X = X_{0,1} \circ_{n+1} m \circ_{n+1} m =  X_{0,2} \circ_{n+1} m \circ_{n+1} m$, then we have by induction that the lemma holds for $X_{0}$. Moreover, we have for $a,b\in \lh n \rh$ that
$$ a <_{X_{0,1}} b \iff a <_{X_{0}} b \iff a <_{X_{0,2}} b \text{ and } a \tre_{X_{0,1}} b \iff a \tre_{X_{0}} b \iff a \tre_{X_{0,2}} b$$
which proves the lemma for $X$.
\end{proof}

\subsection{The morphisms $\mmNSOp \longrightarrow \Multitr$}\label{mNSOPMulti}\label{parmapmulti}

Let $\C$ be a small category. We denote by $\mathrm{Multi} \C$ the $\mathrm{Ob}(\C)$-colored operad for which $\mathrm{Multi} \C(c_1, \dots, c_n; c)$ is freely generated as a $k$-module by $n$-tuples $(\nth{\zeta})$ of $\C$-morphisms with $\zeta_i: c_i \longrightarrow c$, $\Ss_{n}$ acts by permutating labels, and composition is defined in the obvious way.

Let $\Delta$ be the simplex category. Next, we construct a morphism of operads
$$ \NSOp \longrightarrow \Multitr$$
by associating to every indexed tree $I$ in $\NSOp(\nth{q};q)$ a $n$-tuple $\zeta_I$ in $\Multitr(\nth{q};q)$ which assigns to each vertex $a$, considered as an $q_a$-corolla, a numbering denoting where its inputs are amongst the inputs of the indexed tree as a whole. 

It suffices to define the morphism on the generators $E_i\in \NSOp$ and show that it respects the relations.

\begin{constr}\label{NSOpDelta}
Let $E_i\in \NSOp(q_1,q_2;q_1 +q_2 -1)$ for $ 1 \leq i \leq q_1$, then we define 
$$\zeta_{E_i,1}(t) := \begin{cases} t & t< i \\ t+q_2 -1 & t\geq i \end{cases} \quad \text{ and } \quad \zeta_{E_i,2}(t) := t + i-1$$
Then, $\zeta_{E_i}\in \Multitr(q_1,q_2;q_1+q_2-1)$, that is, it is a tuple of non-decreasing maps. Moreover, if $q_2 > 0$, then these are strictly increasing.
\end{constr}

We will employ it as in the following example.

\begin{vb}
Let $\A$ be a $k$-linear category, then its Hochschild complex is defined as 
$$\textbf{C}^n(\A) = \prod_{A_0,\ldots,A_n \in \A} \Hom_k\left(\A(A_1,A_0) \otimes \ldots \otimes \A(A_n, A_{n-1}), \A(A_n,A_0)\right)$$
For a Hochschild cochain $\phi \in \textbf{C}^n(\A)$ and a $n$-simplex $A_0 \overset{a_1}{\leftarrow} A_1 \overset{a_2}{\leftarrow} \ldots \overset{a_{n-1}}{\leftarrow}A_{n-1} \overset{a_n}{\leftarrow} A_n$ in $\A$, we have that $\phi^{A_0,\ldots,A_n}(\nth{a}) \in \A(A_n,A_0)$.

Let $\phi_1 \in \textbf{C}^{q_1}(\A)$ and $\phi_2 \in \textbf{C}^{q_2}(\A)$, then each $E_i \in \NSOp(q_1,q_2;q_1+q_2-1)$ determines a cochain $\phi_1 \circ_i \phi_2 \in \textbf{C}^{q_1+q_2-1}(\A)$ as follows
$$(\phi_1 \circ_i \phi_2)^{A_0,\ldots,A_{q_1+q_2-1}} = \phi_1^{A_{\zeta_{E_i,1}(0)},\ldots,A_{\zeta_{E_i,1}(q_1)}} \circ_i \phi_2^{A_{\zeta_{E_i,2}(0)}, \ldots, A_{\zeta_{E_i,2}(q_2)}}$$
which we can visualize using $n$-corollas
$$\scalebox{0.9}{$\input{NSOpMultiD.tikz}$}$$
\end{vb}

\begin{lemma}\label{multitr}
Construction \ref{NSOpDelta} extends to a morphism of operads 
$$\NSOp \longrightarrow \Multitr: (T,I) \longmapsto \zeta_I$$
\end{lemma}
\begin{proof}
It suffices to verify the relations \ref{NSOpRel_I} and \ref{NSOpRel_II} of $\NSOp$. These are two simple computations and thus we only verify the first relation \ref{NSOpRel_I} as an example. Let $\zeta := \zeta_{E_i} \circ_2 \zeta_{E_j}$ denote the left-hand side, then we compute
\begin{alignat*}{2}
\zeta_1(t) &= \zeta_{E_i,1}(t) &&= \begin{cases} t & t<i \\ t+ q_2 +q_3 - 2 & t\geq i \end{cases} \\
\zeta_2(t) &= \zeta_{E_i,2} \circ \zeta_{E_j,1}(t) &&= \begin{cases} t+ i-1 & t<j \\ t+ i-1 + q_3 -1 & t\geq j \end{cases}  \\
\zeta_3(t) &= \zeta_{E_i,2} \circ \zeta_{E_j,2}(t) &&= t + i -1 + j-1
\end{alignat*}
On the other hand, we compute the right-hand side $\zeta' := \zeta_{E_{i-1+j}} \circ_1 \zeta_{E_i}$ and obtain
\begin{alignat*}{2}
\zeta'_1(t) &= \zeta_{E_{i-1+j},1} \circ \zeta_{E_i,1}(t)&&= \begin{cases} t & t<i \\ t+ q_2-1 +q_3 - 1 & t\geq i \end{cases} \\
\zeta'_2(t) &= \zeta_{E_{i-1+j},1} \circ \zeta_{E_i,2}(t) &&= \begin{cases} t+ i-1 & t<j \\ t+ i-1 + q_3 -1 & t\geq j \end{cases}  \\
\zeta'_3(t) &= \zeta_{E_{i-1+j},1}(t) &&= t + i -1 + j-1
\end{alignat*}
\end{proof}
\begin{opm}
In appendix \ref{appgenfree} we have added a generator-free description of this morphism and an alternative proof of lemma \ref{multitr}, which we consider insightful and valuable, especially for concrete computations of signs in later sections.
\end{opm}

\begin{vb}
Consider the indexed tree $$I = \quad \input{multiD_example.tikz} \quad \in \NSOp(3,0,4,1,0; 4)$$
then we compute $\zeta_{I}$ and obtain
$$\begin{tikzpicture}[baseline = (base), scale = 0.35]
\node (base) at (1,-0.5) {};
\node (T) at (1.3,1) {$\zeta_{I,1}$};
\node (A) at (0,0) {$0$};
\node (B) at (0,-1) {$1$};
\node (C) at (0,-2) {$2$};
\node (D) at (0,-3) {$3$};

\node (A') at (3,0) {$0$};
\node (B') at (3,-1) {$1$};
\node (C') at (3,-2) {$2$};
\node (D') at (3,-3) {$3$};
\node (E') at (3,-4) {$4$};

\draw[->] (A) edge (A') (B) edge (A') (C) edge (D')(D) edge (E');
\end{tikzpicture}\quad \begin{tikzpicture}[baseline = (base), scale = 0.35]
\node (base) at (1,-0.5) {};
\node (T) at (1.3,1) {$\zeta_{I,2}$};
\node (A) at (0,0) {$0$};

\node (A') at (3,0) {$0$};
\node (B') at (3,-1) {$1$};
\node (C') at (3,-2) {$2$};
\node (D') at (3,-3) {$3$};
\node (E') at (3,-4) {$4$};

\draw[->] (A) edge (A') ;
\end{tikzpicture} \quad \begin{tikzpicture}[baseline = (base), scale = 0.35]
\node (base) at (1,-0.5) {};
\node (T) at (1.3,1) {$\zeta_{I,3}$};
\node (A) at (0,0) {$0$};
\node (B) at (0,-1) {$1$};
\node (C) at (0,-2) {$2$};
\node (D) at (0,-3) {$3$};
\node (E) at (0,-4) {$4$};

\node (A') at (3,0) {$0$};
\node (B') at (3,-1) {$1$};
\node (C') at (3,-2) {$2$};
\node (D') at (3,-3) {$3$};
\node (E') at (3,-4) {$4$};

\draw[->] (A) edge (A') (B) edge (B') (C) edge (C')(D) edge (C') (E) edge (D');
\end{tikzpicture}\quad  \begin{tikzpicture}[baseline = (base), scale = 0.35]
\node (base) at (1,-0.5) {};
\node (T) at (1.3,1) {$\zeta_{I,4}$};
\node (A) at (0,0) {$0$};
\node (B) at (0,-1) {$1$};

\node (A') at (3,0) {$0$};
\node (B') at (3,-1) {$1$};
\node (C') at (3,-2) {$2$};
\node (D') at (3,-3) {$3$};
\node (E') at (3,-4) {$4$};

\draw[->] (A) edge (A') (B) edge (B');
\end{tikzpicture}\quad \begin{tikzpicture}[baseline = (base), scale = 0.35]
\node (base) at (1,-0.5) {};
\node (T) at (1.3,1) {$\zeta_{I,5}$};
\node (A) at (0,0) {$0$};

\node (A') at (3,0) {$0$};
\node (B') at (3,-1) {$1$};
\node (C') at (3,-2) {$2$};
\node (D') at (3,-3) {$3$};
\node (E') at (3,-4) {$4$};

\draw[->] (A) edge (C') ;
\end{tikzpicture} $$ 
\end{vb}

It is also possible to associate to an element of $\mNSOp$ an element of $\Multitr$.

\begin{lemma}
Let $X = [ I \circ_{n+1} m \circ_{n+1} \ldots \circ_{n+1} m]  \in \mNSOp(\nth{q};q)$, then $\zeta_{I,t}:[q_{t}]\longrightarrow [q]$ for $t\in \lh n \rh$ is independent of the representative $I$ of $X$.

In this case, we write $\zeta_{X}$.
\end{lemma}
\begin{proof}
We prove the lemma by induction  on $k$ the number of occurrences of $m$. The cases $k=0$ and $k=1$ are trivial, so assume $k>1$. Let $X_{0,i} := X_{0} \circ_{n+1} \; \begin{tikzpicture}
	\begin{pgfonlayer}{nodelayer}
		\node [style=none] (0) at (0, -1) {$1$};
		\node [style=none] (1) at (0, 1) {$2$};
		\node [style=none] (2) at (0, 0.525) {};
		\node [style=none] (3) at (0, -0.55) {};
		\node [style=small dot] (4) at (0.4, 0) {$i$};
	\end{pgfonlayer}
	\begin{pgfonlayer}{edgelayer}
		\draw (2.center) to (3.center);
	\end{pgfonlayer}
\end{tikzpicture}
$ such that $X = X_{0,1} \circ_{n+1}m \circ_{n+1} m = X_{0,2} \circ_{n+1}m \circ_{n+1} m$, then by induction and lemma \ref{multitr} we have for $t\in \lh n \rh$ that 
$$  \zeta_{X_{0,1},t} = \zeta_{X_{0},t} = \zeta_{X_{0,2},t}$$
which proves the lemma.
\end{proof}

\begin{prop}\label{propzeta}\label{propmapmulti}
We have morphisms of operads
$$\NSOp \longrightarrow \Multitr : I \longmapsto \zeta_{I}$$
and its extension
$$\mNSOp \longrightarrow \Multitr: X \longmapsto \zeta_{X}.$$
\end{prop}
Moreover, this last morphism is surjective, but not an isomorphism. This is due to the existence of vertices with zero inputs which collapse information. We consider a simple example.
\begin{vb}
Consider the indexed $2$-corolla $I:=\; \input{Corolla_1.tikz}	$ and its permuted form $I^{(2,3)} =\; \input{Corolla_2.tikz}$ as elements of $\NSOp(2,0,0;0)$, then they have the same image in $\Multitr$. Note that this example holds for both $\NSOp$ and $\mNSOp$. 
\end{vb} Hence, we can consider $\mNSOp$ as a finer operad than $\Multitr$ and thus encoding more information.

\subsection{The morphism $\Bracee \longrightarrow \NSOp_s$}\label{parbraceNSOp}

In order to define the morphism $\phi: \Bracee \longrightarrow \NSOp_s$ properly we need to compile the colored operad $\NSOp$ into a graded non-colored operad 
$$\NSOp_s(n) \sub \prod_{\nth{p}} \NSOp(\nth{p};p).$$
where an element $x\in\NSOp(\nth{p};p)$ is graded as $|x|= \sum_{i=1}^{r} (p_{i}-1) - (p-1) = 0$ (this is the \emph{suspended} grading, whence the subscript) and $\NSOp_s(n)$ is the subspace generated by sequences of elements with constant grading. The composition on $\NSOp_s$ is derived from the composition of $\NSOp$ where it is set to $0$ when the colors do not match. Note in particular that the $\Ss_{n}$-action on $\NSOp_s(n)$ is affected by this grading: permuting two vertices $i$ and $j$ introduces the signs $(-1)^{(p_{i}-1)(p_{j}-1)}$. 

 \begin{mydef}
 Let $(T,I) \in \NSOp(\nth{p};p)$, then $(T,I)$ is a coloring of $T$ and we write $\Clr(T,\nth{p})$ as the set of all such colorings of $T$.
 \end{mydef}
 
 In order to define the sign $\sgn_{T}(I)$ for $T\in \Bracee(n)$, we use the morphism of operads
 $$\NSOp \longrightarrow \Multitr : (T,I) \longmapsto \zeta_{I}$$
 and base this definition on the sign $\sgn_{Q}(\zeta,I)$ from \cite[Def. 4.20]{hawkins}.
 
\begin{constr}
 We work with the following alphabet 
$$ 1_{i},\ldots,(p_{i}-1)_{i}$$
for $i=1,\ldots,n$ and define the word 
$$J^{s}(\nth{p}) = 1_{1}\ldots(p_{1}-1)_{1}\ldots1_{n}\ldots(p_{n}-1)_{n}$$
We define a second word $J^{s}_{T}(I)$ having in the $\zeta_{I,k}(i)$-th position $i_{k}$ for $ 1 \leq i \leq q_{k}-1$. Note that we start from position $1$ for $J^{s}_{T}(I)$ (instead of $0$).
\end{constr}

\begin{mydef}
For $I\in \NSOp(\nth{p};p)$ where we replace those $p_{i} = 0$ by $2$, we define $\sgn_{T}(I)$ as the sign of the shuffle transforming $J^{s}(\nth{p})$ to $J_{T}^{s}(I)$.
\end{mydef}

\begin{theorem}\label{thmbraceNSOp}
We have a morphism of operads 
$$\phi:\Bracee \longrightarrow \NSOp_s : T \longmapsto \left( \sum_{I \in \Clr(T,\nth{p})} \sgn_{T}(I)(T,I) \right)_{\nth{p}}. $$ 
\end{theorem}
\begin{proof}
Per definition of $\sgn_{T}(I)$ we see that $\phi$ is equivariant. Hence, we only need to verify that $\sgn_{T\circ_{1}T'}(I\circ_{1}I') = \sgn_{T}(I) \sgn_{T'}(I')$ for $T\in \Bracee(n), T'\in \Bracee(m)$ and $I\in \Clr(T,\nth{p})$ and $I'\in \Clr(T',\fromto{p'}{m})$. This equation holds as we can decompose the shuffle $\chi'': J^{s}(p'_{1},\ldots,p'_{m},p_{2},\ldots,p_{n}) \leadsto J_{T\circ_{i}T'}(I\circ_{i}I')$ into two shuffles 
$$ J^{s}(p'_{1},\ldots,p'_{m},p_{2},\ldots,p_{n}) \overset{\chi'}{\leadsto} J_{T'}^{s}(I')J^{s}(p_{2},\ldots,p_{m}) \overset{\chi}{\leadsto} J_{T\circ_{1}T'}^{s}(I\circ_{1}I')$$
where $\chi$ and $\chi'$ are the corresponding shuffles determining $\sgn_{T}(I)$ and $\sgn_{T'}(I')$.
\end{proof}
 
\subsection{The morphism $\FS \longrightarrow \mNSOp_{st}$}\label{parFSmNSOp}
 
 In order to define the morphism $ \bar{\phi}: \FS \longrightarrow \mNSOp_{st}$ properly, we again need to compile the colored operad $\mNSOp$ to obtain a graded non-colored operad 
$$\mNSOp_{st}(n) \sub \prod_{\nth{q},q} \mNSOp(\nth{q};q)$$
where an element $x\in\mNSOp(\nth{q};q)$ is graded as $deg(x) = \sum_{i=1}^{r} q_{i} - q$ (\emph{standard} grading) and $\mNSOp_{st}(n)$ is generated by the sequences of constant grading. The composition on $\mNSOp_{st}$ is derived from the composition of $\mNSOp$ where it is set to $0$ when the colors do not match. Note in particular that the $\Ss_{n}$-action on $\mNSOp_{st}(n)$ is affected by this grading: permuting two vertices $i$ and $j$ introduces the sign $(-1)^{q_{i}q_{j}}$. 

\subsubsection{Colorings}

\begin{mydef}\label{coloringWord}
Let $X:=[I \circ_{n+1} m \circ_{n+1} \ldots \circ_{n+1} m ] \in \mNSOp(\nth{q};q)$ for $I$ having $n+k$ vertices, then $X$ is a coloring of $W \in \FS(n)$ if 
\begin{itemize}
\item each vertex $n+1,\ldots,n+k$ has exactly two children in $I$
\item for $u,v \in \lh n \rh$ holds
\begin{enumerate}
\item $u <_{X} v \iff W = \ldots  {u} \ldots  {v} \ldots  {u}\ldots $
\item $u \tre_{X} v \iff$ every occurrence of $u$ in $W$ is left of every occurrence of $v$ in $W$.
\end{enumerate}
\end{itemize}
We write $\Clr(W,\nth{q})$ for the set of all such colorings for $W$.
\end{mydef}
\begin{opm}
An element $X\in \Clr(W,\nth{q})$ has $n - deg(W) -1$ many $m$'s plugged in. Hence, $X \in \mNSOp(\nth{q}; \sum_{i=1}^{n}(q_{i}-1) + n - deg(W))$.
\end{opm}

We give some examples.

\begin{vb}\label{exampleColoring}
The following three elements of $\mNSOp$ 
$$ \scalebox{1}{$\input{Coloring_1.tikz}$} \in \mNSOp(q_1,q_2;q_1+q_2), \qquad  \scalebox{1}{$\input{Coloring_2.tikz}$}\in \mNSOp(q_1,q_2,q_3; q_1+q_2+q_3-1) \quad \text{ for } 1 \leq i \leq q_1$$
 and 
$$\scalebox{1}{$\input{Coloring_3.tikz}$} \in \mNSOp(q_1,...,q_5; \sum_{i=1}^5 q_i -2) \quad \text{ for } 1 \leq j < k \leq q_1$$
 color respectively the words
$$ 12 \in \FS(2), \qquad 1231 \in \FS(3) \qquad \text{ and } \qquad 1213451 \in \FS(5).$$
Note however that not all elements of $\mNSOp$ color a word of $\FS$: the following set of elements
$$ \scalebox{1}{$\begin{tikzpicture}
	\begin{pgfonlayer}{nodelayer}
		\node [style=none] (0) at (0, -1) {$m$};
		\node [style=none] (1) at (0, 1) {$1$};
		\node [style=none] (2) at (0, 0.525) {};
		\node [style=none] (3) at (0, -0.55) {};
		\node [style=small dot] (4) at (0.4, 0) {$r$};
	\end{pgfonlayer}
	\begin{pgfonlayer}{edgelayer}
		\draw (2.center) to (3.center);
	\end{pgfonlayer}
\end{tikzpicture}
$} \in \mNSOp(q_1;q_1+1)$$
for $r\in \{1,2\}$, colors no word in $\FS$ because the vertex plugged by $m$ does not have two children.
\end{vb}

\begin{lemma}
Definition \ref{coloringWord} is well-defined, that is, it is independent of the chosen representative $I$ of $X$.
\end{lemma}
\begin{proof}
Due to lemma \ref{partialorders}, both $<_{X}$ and $\tre_{X}$ are well-defined. We show that the condition stipulating that all vertices of $I$ that are plugged by $m$'s have exactly two children, is independent of the representative of $X$. Thus, suppose $I = I_0 \circ_k \; \begin{tikzpicture}
	\begin{pgfonlayer}{nodelayer}
		\node [style=none] (0) at (0, -1) {$1$};
		\node [style=none] (1) at (0, 1) {$2$};
		\node [style=none] (2) at (0, 0.525) {};
		\node [style=none] (3) at (0, -0.55) {};
		\node [style=small dot] (4) at (0.4, 0) {$1$};
	\end{pgfonlayer}
	\begin{pgfonlayer}{edgelayer}
		\draw (2.center) to (3.center);
	\end{pgfonlayer}
\end{tikzpicture}
$ for some $I_0 \in \NSOp$, such that both vertices $k$ and $k+1$ of $I$ are plugged by $m$'s in $X$. Due to the relations in $\mNSOp$, $X$ can equivalently be represented using $I' := I_0 \circ \; \begin{tikzpicture}
	\begin{pgfonlayer}{nodelayer}
		\node [style=none] (0) at (0, -1) {$1$};
		\node [style=none] (1) at (0, 1) {$2$};
		\node [style=none] (2) at (0, 0.525) {};
		\node [style=none] (3) at (0, -0.55) {};
		\node [style=small dot] (4) at (0.4, 0) {$2$};
	\end{pgfonlayer}
	\begin{pgfonlayer}{edgelayer}
		\draw (2.center) to (3.center);
	\end{pgfonlayer}
\end{tikzpicture}
$. In this case, we have that vertices $k$ and $k+1$ each have exactly $2$ children in $I$ iff vertex $k$ has exactly $3$ children in $I_0$ iff vertices $k$ and $k+1$ each have exactly $2$ children in $I'$.
\end{proof}

We construct a word for every element of $\mNSOp$ satisfying the above criteria.

\begin{constr}\label{underlyingWord}
Let $X:=[I \circ_{n+1} m \circ_{n+1} \ldots \circ_{n+1} m ] \in \mNSOp(\nth{q};q)$ such that each vertex $a > n$ in $I$ has exactly two children, then we construct a word $W_{X} \in \FS(n)$ such that $X \in \Clr(W_{X},\nth{q},q)$.
\begin{itemize}
\item To every tree $T$ we can associate a word $W_{T} \in \FS(n+k)$ (see \cite[\S 2.3]{hawkins}).
\item Suppose for $X_{0}\in \mNSOp$ such that $X_{0}\circ_{n+1}m =X$ we have an associated word $W_{X_{0}}\in \FS(n+1)$, then let $W_{X}$ be the word given by deleting all occurrences of $n+1$. Then $W_{X}\in \FS(n)$ because $n+1$ had two children, so no degeneracy can occur.
\end{itemize}
\end{constr}

We consider an example of this procedure.

\begin{vb}
We consider the element
$$\input{Coloring_2.tikz} \in \mNSOp(q_1,q_2,q_3;q_1+q_2+q_3-1)$$
for some $1\leq i \leq q_1$ from example \ref{exampleColoring} and show how construction \ref{underlyingWord} assigns a word. First, we associate to the indexed tree 
$$\input{wordAssociate.tikz} \in \NSOp(q_1,q_2,q_3,2;q_1+q_2+q_3-1)$$
the word $1424341$ and then delete all occurrences of $4$ as it is plugged by an instance of $m$. As a result, we obtain the word $1231$.
\end{vb}

\begin{lemma}
For $X \in \mNSOp(\nth{q};q)$ we have $X \in \Clr(W_{X},\nth{q})$ and if $X \in \Clr(W,\nth{q})$, then $W = W_{X}$. 
\end{lemma}
\begin{proof}
This clearly holds for $X = I \in \NSOp(\nth{q};q)$. Assume the lemma holds for $X_{0} \in \mNSOp$ and $X = X_{0} \circ_{n+1} m$, then $W_{X_{0}} = W_{0} (n+1) W_{1}  (n+1) W_{2}  (n+1) W_{3}$ for $W_{0}$ and $W_{3}$ possibly empty. In this case, $W_{X} = W_{0}W_{1} W_{2} W_{3}$ and it is easy to see that $X \in \Clr(W_{X},\nth{q})$.

Now reversely, if $X \in \Clr(W,\nth{q})$ and $a \tre b$ are the two children of $n+1$ in $X_{0}$, then $W = W_{0}W_{1}W_{2}W_{3}$ where $W_{1}=  {a} \ldots  {a}, W_{2} =  b \ldots b$ and $W_{0}$ and $W_{3}$ are possibly empty. In that case, $X_{0} \in \Clr(W_{0} (n+1) W_{1}  (n+1) W_{2}  (n+1) W_{3},\nth{q},2)$ and thus by induction $W_{X_{0}} = W_{0} (n+1) W_{1}  (n+1) W_{2}  (n+1) W_{3}$. Hence, $W_{X} = W_{0}W_{1}W_{2}W_{3} = W$.
\end{proof}

\begin{lemma}\label{uniqueword}
Let $X \in \Clr(V,\nth{q},q)$ and $Y \in \Clr(W,\fromto{q'}{m},q_{i})$, then there exists a unique $U \in Ext(V,W,i)$ such that $X\circ_{i} Y\in \Clr(U,q_{1},\ldots,\fromto{q'}{m},\ldots,q_{n},q)$  
\end{lemma}
\begin{proof}
By construction \ref{underlyingWord} we obtain a word $U\in \FS(n+m-1)$ such that $Z := X \circ_{i} Y \in \Clr(U,\ldots)$. We show that $U \in Ext(V,W,i)$: let $U_{\al}$ be the word obtained from deleting from $U$ occurrences of vertices not in the image of $\al$ and eliminating consecutive repetitions ($uu\mapsto u)$. It is easy to check that $W = \ldots  {u} \ldots  {v} \ldots  {u} \ldots $ iff $U_{\al} = \ldots  {\al(u)} \ldots  {\al(v)} \ldots  {\al(u)} \ldots$, and that all occurrences of $u$ are left to those of $v$ in $W$ iff the same holds for $\al(u)$ and $\al(v)$ in $U_{\al}$.

Let $U_{\be}$ be the word obtained from $U$ by relabelling by $\be$ and eliminating consecutive repetitions. To verify that $U_{\be} = V$ is straight forward, except in the following case: $U_{\be} = \ldots {i} \ldots {\be(u)} \ldots {i} \ldots$ where $\be(u) \neq i$. In this case, there exist $v,v'$ such that $U = \ldots {\al(v)} \ldots {u} \ldots {\al(v')} \ldots$. Now, we argue that there exists a vertex $v''$ such that $\al(v'') <_Z u$. If neither $\al(v)$ nor $\al(v')$ do, then $U$ implies that $\al(v) \tre_Z u \tre_Z \al(v')$. As $\al(v)$ and $\al(v')$ are part of the same subtree $\al(Y)$, i.e. the image of $Y$ under $\al$ in $Z$, there must be
some vertex $a$ in the tree underlying $\al(Y)$ (possibly plugged by $m$) such that $a$ lies underneath $u$. As $Y$ is a coloring of a word, the conditions imply that $a$ is not plugged by an instance of $m$ (otherwise it would not have two children in $Y$). As a result, there is some vertex $v''$ in $Y$ such that $\al(v'') =a <_Z u$. Thus, $i <_X \be(u)$ which verifies that $V = \ldots i \ldots \be(u) \ldots i \ldots$. This clearly also holds reversibly.  
\end{proof}

\begin{lemma}\label{uniqueColorings}
Let $U \in Ext(V,W,i)$ and $Z\in \Clr(U,q_{1},\ldots,\fromto{q'}{m},\ldots,q_{n},q)$, then there exist unique colorings $X\in \Clr(V,\nth{q},q)$ and $Y\in \Clr(W,\fromto{q'}{m},q_{i})$ such that $Z = X \circ_{i} Y$.
\end{lemma}
\begin{proof}
Let $Z = [ I'' \circ_{n+m} m \circ_{n+m} \ldots \circ_{n+m}m]$ with $l+k$ added $m$'s. The word $W$ can be uniquely written as 
$$W = W_{1}\ldots W_{t}$$
where two subwords $W_{j}$ and $W_{j'}$ do not share any occurrence of the same number, and $W_{j}$ is of the form ${a_{j}}\ldots {a_{j}}$. As $Z$ is a coloring of $U$, we have that $\al(a_{1})\tre_{Z} \ldots \tre_{Z} \al(a_{t})$ and no vertex of $\Image(\al)$ lies under any $a_{j}$. In this case, there exists some vertex $a\in \{n+m,\ldots,n+m+k+l-1\}$ such that $a\leq_{I''}\al(a_{j})$ which is $\leq_{I''}$-maximal for these conditions (otherwise when applying $\be$ to $U$ we will not obtain $V$).

 Let $I'$ be the minimal subtree of $I''$ on the root $a$ containing $\Image(\al)$. By contracting this subtree to a point we obtain a tree $I$ such that, after permutation of some vertices, we obtain $I \circ_{i} I' = I''$. Consider also the permutation such that $Z = [I'' \circ_{m+i} m \circ_{m+i} \ldots \circ_{m+i} m \circ_{n+m} m \circ_{n+m} \ldots \circ_{n+m} m]$.
 
It now suffices to show that $X := [I \circ_{n+1} m\circ_{n+1} \ldots \circ_{n+1} m] \in \Clr(V,\ldots)$ and $Y:= [I' \circ_{m+1} m \circ_{m+1} \ldots \circ_{m+1} m ] \in \Clr(W,\ldots)$, which is a straight forward computation using the facts $X \circ_{i} Y =Z$, $Z\in \Clr(U,\ldots)$ and $U \in Ext(V,W,i)$. 
\end{proof}

\subsubsection{Signs}

In order to define a sign $\sgn_{W}(X)$ for $W\in \FS(n)$ and $X\in \Clr(W,\nth{q})$ we use the morphism of operads
$$ \NSOp \longrightarrow \Multitr : (T,I) \longmapsto \zeta_{I}$$
which extends to 
$$\mNSOp \longrightarrow \Multitr: X \longmapsto \zeta_{X}$$
We base this definition of the sign on the sign $\sgn_{Q}(\zeta,I)$ defined in \cite[\S 4.7]{hawkins}.

\begin{lemma}
Let $X \in \Clr(W,\nth{q};q)$ for $W\in \FS(n)$ and $q_{i} > 0$, then $\zeta_{X}$ is a coloring of $W$ in the sense of \cite[Def. 4.13]{hawkins}, that is,
\begin{itemize}
\item $\zeta_{X} \in \Multitrp$,
\item $\bigcup_{i=1}^{n} \Image(\zeta_{X,i}) = [q]$
\item for each $a\in \lh n \rh$ there exists a function $\pi_{a}:[q_{a}] \longrightarrow W$ such that
\begin{enumerate}
\item the image of $\pi_{a}$ is the set of occurrences of $a$ in $W$,
\item for an $(i,a)\in W$, $\zeta_{X,a}(\pi_{a}\inv(i,a))$ is an interval
\item if $W= \ldots {a}{b} \ldots$, then 
$$\max \; \zeta_{X,a}(\pi_{a}\inv(i,a)) = \min \; \zeta_{X,b}(\pi_{b}\inv(i+1,b))$$
\end{enumerate}
\end{itemize} 
\end{lemma}
\begin{proof}
We first show it holds for $X= (T,I) \in \NSOp$ by induction on the number of vertices: let $$ I = \quad \input{Decomposition_I_coloring.tikz}$$
be its decomposition into its root $u$ with maximal subtrees $I^i$. In this case, we have $W = u W^1 u \ldots u W^k u$ where the subwords $W_i$ represent the subtrees $I^i$. Let $\ga_{i}:\lh k_{i} \rh \hookrightarrow \lh n \rh$ be the maps embedding the tree $I^i_0$ onto $I^i$ in $I$ and $I^i_0\in \Clr(W^i_0,\ldots)$, then they extend to a map $\overline{\ga_i}: W^i_0 \hookrightarrow W$ inserting $W^i_0$ as $W^i$ into $W$. By induction, the lemma holds for $I^{i}_0 \in \Clr(W^i_0,\ldots)$. Let $(p_1,u),\ldots,(p_{k+1},u)$ be all the occurrences of $u$ in $W$, then we define
$$\pi_{\ga_{i}(a)} = \overline{\ga_{i}} \circ \pi_{a}^{I^{i}_0} \text{ and } \pi_{u}(t) =  \begin{cases} (p_1,u) & t < i_1 \\ 
							(p_j, u) & i_{j-1} \leq t < i_j \\
						(p_{k+1}, u) & i_k \leq t \end{cases}$$
						then it is easy to verify that these satisfy the above conditions. 
						
Now assume $X = [X_{0} \circ_{n+1} m]$ such that the lemma holds for $X_{0} \in \Clr(W_{X_{0}},\ldots)$, then $W_{X}$ is obtained from $W_{X_{0}}$ by deleting all occurrences of $(n+1)$. As the vertex $n+1$ has exactly two children in $X_{0}$ and $q_{n+1} = 2$, we still have that $\bigcup_{i=1}^{n} \Image(\zeta_{X,i}) =\bigcup_{i=1}^{n+1} \Image(\zeta_{X_{0},i}) =[q]$ and that
$$\max \; \zeta_{X_{0},a}(\pi_{a}\inv{(i,a)}) = \min \; \zeta_{X_{0},b}(\pi_{b}\inv{(i,b)})$$
for $W= \ldots {a}{b} \ldots$. Hence, $\zeta_{X}$ satisfies the lemma. 
\end{proof}

Let us define analogously the sign corresponding to the horizontal part of $\sgn_{Q}(\zeta_{X},I)$.

\begin{constr}
 We work with the following alphabet 
$$ 0_{i},\ldots,(q_{i}-1)_{i}$$
for $i=1,\ldots,n$ and define the word 
$$J(\nth{q}) = 0_{1}\ldots(q_{1}-1)_{1}\ldots0_{n}\ldots(q_{n}-1)_{n}$$
The second word $J_{W}(X)$ is the concatenation of two words $J_{0,W}(X)$ and $J_{1,W}(X)$ defined as follows
\begin{itemize}
\item $J_{0,W}(X)$ consists of all $0_{k}$ for $k$ interposed in $W$,  put in reverse $\downarrow$-order. 
\item $J_{1,W}(X)$ has in the $\zeta_{X,k}(i)$-th position $i_{k}$ for $ 1 \leq i \leq q_{k}-1$ for $k$ interposed, and in the $\zeta_{I,k}(i)$ position $i_{k}$ for $0 \leq i \leq q_{k}-1$ for $k$ not interposed. Note that we start from position $0$. 
\end{itemize}
\end{constr}

\begin{mydef}
For $X\in \mNSOp(\nth{q};q)$ where we replace those $q_{i} = 0$ by $2$, we define $\sgn_{W}(X)$ as the sign of the shuffle transforming $J(\nth{q})$ to $J_{W}(X)$.
\end{mydef}

\begin{vb}
Consider the words
$$W= 13121, \quad W'= 1231$$
and colorings
$$X = \quad \input{sgnColoring.tikz}, \qquad X'= \quad \input{sgnColoring2.tikz}$$
for which we calculate the words $J_W(X)$ and $J_{W'}(X')$ and their corresponding signs. In the first case, we have 
$$J_W(X) = 0_3 0_2 0_1 \ldots (i-1)_1 1_3 \ldots (q_3-1)_3 i_1 \ldots (j-1)_1 1_2 \ldots (q_2-1)_2 j_1 \ldots (q_1-1)_1$$
which corresponds to the sign $\sgn_W(X) = (-1)^{(q_2-1)(q_1-j) + (q_3-1)(q_1-i) + q_3q_2 +q_3-1}$. For the second case, we calculate
$$ J_{W'}(X') = 0_2 0_1 \ldots (k-1)_1 1_2 \ldots (q_2-1)_2 0_3 \ldots (q_3-1)_3 k_1 \ldots (q_1-1)_1$$
which corresponds to the sign $\sgn_{W'}(X') = (-1)^{ (q_2+q_3)(q_1-k) + q_1}$. Note in particular that in $J_{W'}(X')$ the letter $0_3$ is not taken to the front of the word as $3$ is not interposed in $W'$.
\end{vb}

\begin{lemma}\label{signsComposition}
Let $X \in \Clr(W,\nth{q})$ and $X'\in \Clr(W',\fromto{q'}{m})$, and $W''\in Ext(W,W',1)$, then $\sgn_{W}(X) \sgn_{W'}(X') = \sgn_{W''}(X\circ_{1}X') \sgn_{W,W',1}(W'')$.
\end{lemma}
\begin{proof}
We can assume that all $q_{i}$ and $q'_{j}$ are not zero. We can decompose $\sgn_{W}(X)$ in three components 
\begin{itemize}
\item sign of the shuffle $\si$ shuffling $J_{0,W}(X)$ to $0_{v_{1}}\ldots 0_{v_{k}}$ for $v_{1} < \ldots < v_{k}$ the interposed vertices of $X$,
\item sign of the shuffle $\tau$ shuffling $J_{1,W}(X)$ to concatenation of $1_{i} \ldots (q_{i}-1)_{i}$ for $i$ interposed and $0_{i} \ldots (q_{i}-1)_{i}$ for $i$ not interposed. We call this latter sequence $J_{W}^{int}(\nth{q})$.
\item sign of the shuffle $\rho$ shuffling 
$$0_{v_{1}}\ldots 0_{v_{k}}J_{W}^{int}(\nth{q}) \leadsto J(\nth{q})$$
\end{itemize}
We add $'$ and $''$ to denote the correspondings shuffles for $X'$ and $X'':= X \circ_{1} X'$.

First we have that $(-1)^{\si''} = (-1)^{\si+\si'}\sgn_{W,W',1}(W'')$ per definition of $\sgn_{W,W',1}(W'')$. Further we clearly have $(-1)^{\tau''} = (-1)^{\tau+ \tau'}$ by simply applying them one after the other and renaming using $\al$ and $\be\inv$
$$ J_{1,W''}(X) \leadsto J_{1,W'}J_{W}(q_{2},\ldots,q_{n}) \leadsto J_{W'}(q'_{1},\ldots,q_{m}')J_{W}(q_{2},\ldots,q_{n})$$
as $\zeta_{X''} = \zeta_{X} \circ_{1} \zeta_{X'}$.
We also have that $(-1)^{\rho''}= (-1)^{\rho+\rho'}$ because the length of the sequence $J_{W'}^{int}(\fromto{q'}{m})$ is $q_{1}-1$ if $1$ is interposed, and $ q_{1}$ if $1$ is not interposed.
\end{proof}

\subsubsection{The morphism}

\begin{prop}\label{gradedoperadmap}
We have a morphism of graded operads 
$$ \bar{\phi} : \FS \longrightarrow \mNSOp_{st}: W \longmapsto \left( \sum_{X\in \Clr(W,\nth{q})} \sgn_{W}(X) X \right)_{\nth{q}}$$
\end{prop}
\begin{proof}
By definition of $\sgn_{W}(X)$ the above linear maps are equivariant. By lemma \ref{uniqueword}, \ref{uniqueColorings} and \ref{signsComposition} they define a morphism of graded operads.
\end{proof}
 
We make $\mNSOp_{st}$ into a dg-operad with the hochschild differential, then $\phi$ will be a morphism of dg-operads.

\begin{mydef}
Define for every $q\in \N$ the element 
\begin{equation} \label{hochschild}
 D_q =   \; \begin{tikzpicture}
	\begin{pgfonlayer}{nodelayer}
		\node [style=none] (0) at (0, -1) {$m$};
		\node [style=none] (1) at (0, 1) {$1$};
		\node [style=none] (2) at (0, 0.525) {};
		\node [style=none] (3) at (0, -0.55) {};
		\node [style=small dot] (4) at (0.4, 0) {$2$};
	\end{pgfonlayer}
	\begin{pgfonlayer}{edgelayer}
		\draw (2.center) to (3.center);
	\end{pgfonlayer}
\end{tikzpicture}
 \; + \quad \sum_{i=1}^{q} (-1)^{i} \;  \begin{tikzpicture}
	\begin{pgfonlayer}{nodelayer}
		\node [style=none] (0) at (0, -1) {$1$};
		\node [style=none] (1) at (0, 1) {$m$};
		\node [style=none] (2) at (0, 0.525) {};
		\node [style=none] (3) at (0, -0.55) {};
		\node [style=small dot] (4) at (0.4, 0) {$i$};
	\end{pgfonlayer}
	\begin{pgfonlayer}{edgelayer}
		\draw (2.center) to (3.center);
	\end{pgfonlayer}
\end{tikzpicture}
 \; + (-1)^{q+1}  \; \begin{tikzpicture}
	\begin{pgfonlayer}{nodelayer}
		\node [style=none] (0) at (0, -1) {$m$};
		\node [style=none] (1) at (0, 1) {$1$};
		\node [style=none] (2) at (0, 0.525) {};
		\node [style=none] (3) at (0, -0.55) {};
		\node [style=small dot] (4) at (0.4, 0) {$1$};
	\end{pgfonlayer}
	\begin{pgfonlayer}{edgelayer}
		\draw (2.center) to (3.center);
	\end{pgfonlayer}
\end{tikzpicture}
 \; \in \mNSOp(q;q+1)
\end{equation}
which compile into an element of degree $-1$ 
$$ D := (D_q)_{q\geq 0} \in \mNSOp_{st}(1).$$
We consider the associated derivation
$$\partial_{D}(X):= D \circ_{1} X - (-1)^{deg(X)} \sum_{i=1}^{n} X \circ_{i} D$$
for $X \in \mNSOp_{st}(n)$.
\end{mydef}

\begin{prop}\label{differential}
$\partial_{D}$ defines a differential making $\mNSOp$ into a dg-operad, for which holds 
$$\partial_{D}\left( \bar{\phi}(W)\right) =  \bar{\phi}\left(\partial(W)\right)$$
\end{prop}
\begin{proof}
The first part follows directly if $D \circ_{1} D =0$ which is an easy computation (see  \cite[Prop. 2]{gerstenhabervoronov}).

In order to prove the second part we only need to show this for the generators of $W$, i.e. $12$ and $121\ldots 1k1$ for $k\geq 1$.
\begin{itemize}
\item For $q_{1},q_{2}\in \N$, it is easy to compute that 
$$\partial_{D}\left(  \bar{\phi}(12) \right) = \partial_{D}\left(\; \input{Multiplication.tikz} \;\right) =0$$
 in $\mNSOp(q_{1},q_{2};q_{1}+q_{2})$. 
 \item For $\fromto{q}{k}\in \N$, we also have 
$$  \bar{\phi}(121\ldots 1 k1) = \sum_{1 \leq i_{1}< \ldots < i_{k-1} \leq q_{1}} (-1)^{\eps} \; \input{Corolla_Diff.tikz} \quad \text{ where } \quad \eps = \sum_{j=2}^{k} (q_{j}-1)(q_{1}-i_{j} + \sum_{l>j}(q_{l}-1))$$
and 
$$\partial(121\ldots 1k1) = - 2131\ldots 1 + \sum_{i=2}^{k-1}(-1)^{i} 1\ldots 1i(i+1)1\ldots 1 + (-1)^{k} 1\ldots 1k$$
for which it is also a standard computation to see that $\partial_{D}( \bar{\phi}(121\ldots 1k1)) =  \bar{\phi}( \partial(121\ldots 1k1))$ (see \cite[Thm 3]{gerstenhabervoronov}).
\end{itemize}
\end{proof}

\begin{theorem}\label{thmFSmNSOp}
We have a morphism of dg-operads 
$$ \bar{\phi} : \FS \longrightarrow \mNSOp_{st}$$
\end{theorem}
\begin{proof}
This is the direct consequence of propositions \ref{gradedoperadmap} and \ref{differential}.
\end{proof}

\section{The Gerstenhaber-Schack Complex For Prestacks}\label{parparGS}

Let $(\A, m, f, c)$ be a prestack over a small category $\U$ and let $\CGS(\A)$ be the associated Gerstenhaber-Schack complex as defined in \cite{DVL} (see \S \ref{parGS}).
In loc. cit., a homotopy equivalence $\CGS(\A) \cong \mathbf{CC}(\A!)$ is constructed with the Hochschild complex $\mathbf{CC}(\A!)$ of the Grothendieck contruction $\A!$ of $\A$. Through homotopy transfer, this allows to endow the GS-complex with an $L_{\infty}$-structure. However, it is desirable to have a direct description available of this structure, without reference to transfer.

In the case of a \emph{presheaf}, originally considered by Gerstenhaber and Schack, in \cite{hawkins}, Hawkins introduces an operad $\Quilt \subseteq \FS \otimes_H \Bracee$ which he later extends to an operad $\mQuilt$ acting on the GS-complex. These operads are naturally endowed with $L_{\infty}$-operations as desired.
The action of $\Quilt$ on the GS-complex considered by Hawkins only involves the restriction functors $f$ of the presheaf, the multiplication $m$ being incorporated later on in $\mQuilt$. Unfortunately, the way in which functoriality of $f$ is built into these actions, does not allow for an extension to twisted presheaves or prestacks. 

In our solution for the prestack case, we propose to use $\Quilt$ in a fundamentally different way in relation to the GS-complex, but still allowing us to make use of the naturally associated $L_{\infty}$-structure.
 In this section we capture the higher structure of $\CGS(\A)$ by introducing the operad $\Patch \subseteq \mNSOp \otimes_H \NSOp$ (see \S \ref{parP}) over which the bicomplex $\textbf{C}^{\bullet,\bullet}(\A)$, of which $\CGS(\A)$  is the totalisation, is shown to be an algebra (see Theorem \ref{thmPGS}).
Next, we construct a morphism $\Quilt \longrightarrow \Patch_s$ (see Proposition \ref{propQP}) as a restriction of
$$\bar{\phi} \otimes_H \phi: \FS \otimes_H \Bracee \longrightarrow \mNSOp_{st} \otimes_H \NSOp_s.$$
This morphism is such that the resulting composition
$$R: \Quilt \longrightarrow \End(s\CGS(\A))$$
incorporates the multiplications $m$ and the restrictions $f$. Note that in Hawkins' approach to the presheaf case, the initial action of $\Quilt$ on $\End(s\CGS(\A))$ only incorporates the restrictions. As far as the structure of both approaches goes, the auxiliary operad $\Patch$ we use is the counterpart of the operad $\mathrm{ColorQuilt}$ from \cite[Def. 4.6]{hawkins}.

In \S \ref{parc}, we will further extend the action $R$ in order to incorporate the twists.

\subsection{The GS complex}\label{parGS}

In this section, we recall the notions of prestack and its associated Gerstenhaber-Schack complex, thus fixing terminology and notations. We use the same terminology as in \cite{DVL}, \cite{lowenvandenberghCCT}.

A prestack is a pseudofunctor taking values in $k$-linear categories. Let $\U$ be a small category.

\begin{mydef}
A prestack $\A= (\A,m,f,c)$ over $\U$ consists of the following data:
\begin{itemize}
\item for every object $U \in \U$, a $k$-linear category $(\A(U), m^U, 1^U)$ where $m^U$ is the composition of morphisms in $\A(U)$ and $1^U$ encodes the identity morphisms of $\A(U)$.
\item for every morphism $u:V \longrightarrow U$ in $\U$, a $k$-linear functor $f^u = u \st : \A(U) \longrightarrow \A(V)$. For $u= 1_U$ the identity morphism of $U$ in $\U$,  we require that $(1_U)\st = 1_{\A(U)}$.
\item for every couple of morphisms $v: W \longrightarrow V, u: V \longrightarrow U$ in $\U$, a natural isomorphism
$$c^{u,v}: v\st u \st \longrightarrow (uv)\st.$$
For $u =1$ or $v=1$, we require that $c^{u,v} =1$. Moreover, the natural isomorphisms have to satisfy the following coherence condition for every triple $w:T \longrightarrow W$, $v :W \longrightarrow V$ and $u:V \longrightarrow U$:
$$c^{u,vw}(c^{v,w} \circ 	u \st) = c^{uv,w}(w\st \circ c^{u,v}).$$
\end{itemize}
\end{mydef}

Given such a prestack $\A$, we have an associated Gerstenhaber-Schack complex $\CGS(\A)$. In \cite{DVL} this is defined as the totalisation of a bicomplex $\textbf{C}^{\bu,\bu}(\A)$. We first review some notations.

\medskip

\noindent \emph{Notations.}
Let $\si = (U_{0} \overset{u_{1}}{\rightarrow} U_{1} \rightarrow  \ldots   \overset{u_{p}}{\rightarrow}  U_{p} )$ be a $p$-simplex in the category $\U$, then we have two functors $\A(U_{p}) \longrightarrow \A(U_{0})$, namely
$$\si\hs := u_{1}\st \circ \ldots \circ u_{p} \st \quad\text{ and }\quad \si\st := (u_{p}\circ\ldots \circ u_{1})\st$$
For each $1 \leq k \leq p-1$, denote by $L_{k}(\si)$ and $R_{k}(\si)$ the following simplices
\begin{align*}
L_{k}(\si) = (U_{0} \overset{u_{1}}{\rightarrow} U_{1} \rightarrow  \ldots   \overset{u_{k}}{\rightarrow}  U_{k}) \\
R_{k}(\si) = (U_{k} \overset{u_{k+1}}{\rightarrow} U_{k+1} \rightarrow  \ldots   \overset{u_{p}}{\rightarrow}  U_{p})
\end{align*}
and we consider the following natural isomorphisms
\begin{align*}
c^{\si,k} = c^{u_{k}\ldots u_{1},u_{p}\ldots u_{k+1}}:(L_{k}\si)\st(R_{k}(\si))\st \longrightarrow \si \st\\
\eps^{\si,k} = u_{1}\st \ldots u_{k-1}\st c^{u_{k},u_{k+1}} u_{k+2}\st \ldots u_{p}\st: \si\hs \longrightarrow u_{1}\st \ldots (u_{k+1}u_{k})\st\ldots u_{p} \st 
\end{align*}
We write $c^{\si,k,A}= c^{\si,k}(A)$ and $\eps^{\si,k,A}= \eps^{\si,k}(A)$ for $A\in \A(U_{p})$.

We also define a set $P(\si)$ of \textit{formal paths from $\si\hs$ to $\si \st$} inductively. A formal path is finite sequence of couples $(\tau,i)$ consisting of a simplex $\si$ and a natural number $i$. We set
$$P(u_1,u_2) := \{ ((u_1,u_2),1)\}$$ and
$$P(\si) := \{ (r_{1},\ldots,r_{p-2},(\si,i)): 1 \leq i \leq p-1 \text{ and } (r_{1},\ldots,r_{p-2}) \in P(\partial_{i}\si) \}$$
where $\partial_{i}$ denotes the $i$th face-operator of the nerve $N_{p}(\U)$. Given such a formal path $r=(\fromto{r}{p-1})$ we define its sign
$$(-1)^{r} = \prod_{i=1}^{p-1}(-1)^{r_{i}} \text{ where } (-1)^{(\si,i)}= (-1)^{i}.$$
By interpreting the data $(\si,i)$ as the natural isomorphism $\eps^{\si,i}$, every formal path $r \in P(\si)$ induces a sequence of natural isomorphisms $\overline{r} \in N_{p-1}(\Fun(\A(U_p),\A(U_0)))$. Note that $\eps^{(u_1,u_2),1} = c^{u_1,u_2}$ and its associated sign is $-1$.

Let $\Ss_{t,p-1}$ denote the set of $(t,p-1)$-shuffles, then given a formal path $r\in P(\si)$, a shuffle $\be \in \Ss_{t,p-1}$ and a tuple $a= (A_{0} \overset{a_{1}}{\leftarrow} A_{1} \leftarrow  \ldots   \overset{a_{t}}{\leftarrow}  A_{t}) \in N_{t}(\A(U_{p}))$  let $\be(a,r)\in N_{p-1+t}(\A(U_{0}))$ be its shuffle product with respect to evaluation of functors as defined in \cite[Ex. 3.2, Ex. 3.4]{DVL}.

Here, we give a more explicit definition of $\be(a,r)$: first we construct inductively a sequence $(\fromto{b}{t+p-1})$ which \emph{formally} represents a sequence of morphisms in $A(U_0)$. Every $b_i$ is either of the form $(\tau,a_i,A_{i-1})$ or $(r_i,A_j)$ for $\tau$ a simplex, $a_i$ and $A_j$ respectively a morphism and an object occurring in $a$, and $r_i$ an element of the formal path $r$. Define
$$b_{t+p-1} = \begin{cases} (\si, a_t, A_{t-1}) & \text{ if } \be(t) = t+p-1 \\ (r_{p-1},  A_t) & \text{ if } \be( t+ p -1  ) = t+p-1 \end{cases}$$
then for $1<i \leq t+p-1$, we have two cases: if $b_{i+1} = (\tau, a_j, A_{j-1})$ for some $j$, then define  
$$ b_{i} = \begin{cases} (\tau, a_{j-1}, A_{j-2}) & \text{ if } \be(j-1) = i \text{ for } j\leq t  \\
(r_k, A_{j-1}) & \text{ if } \be(t+k) = i \end{cases}$$
If $b_{i+1} = (r_k,A_j)$ for $r_k = (\tau,l)$, then define
$$ b_{i} = \begin{cases} (\partial_l\tau, a_{j}, A_{j-1}) & \text{ if } \be(j) = i \text{ for } j< t   \\
(r_{k-1}, A_j) & \text{ if } \be(t+k-1) = i \end{cases}$$
Finally, we define $\be(a,r)$ as the realization $\overline{b}=(B_{0} \overset{\overline{b_1}}{\leftarrow} B_{1} \leftarrow  \ldots   \overset{\overline{b_t}}{\leftarrow}  B_{p-1+t}) $ of $b$ where $\overline{b_i} = \tau\hs a_j$ if $b_i = (\tau, a_j,A_j)$ and $\overline{b_i} = \overline{r_k}(A_j)$ if $b_i = (r_k, A_j)$.
\medskip

\begin{mydef}
Let $p,q\geq 0$, then define
$$\textbf{C}^{p,q}(\A) = \prod_{\si \in N_{p}(\U)} \prod_{A \in \A(U_{p})^{q+1}} \Hom(\bigotimes_{i=1}^{q} \A(U_{p})(A_{i},A_{i-1}), \A(U_{0})(\si\hs A_{q},\si\st A_{0}))$$
and set
$$\CGS^{n}(\A) = \bigoplus_{p+q=n}\textbf{C}^{p,q}(\A)$$
The differential $d$ on the GS-complex is defined for $\te \in \textbf{C}^{p,q}(\A)$ as 
$$d(\te) = \sum_{j=0}^{q+1} d_{j}(\te)$$
where $d_{j}: \textbf{C}^{p,q}(\A) \longrightarrow \textbf{C}^{p+j,q+1-j}(\A)$ is defined as 
\begin{itemize}
\item 
\begin{align*}
d_{0}(\te)^{\si}(A)(a) &= m^{U_{0}}(\si\st(a_{1}), \te^{\si}(A_1,\ldots,A_{q+1})(a_{2},\ldots,a_{q+1}))  \\ 
&+ \sum_{i=1}^{q}(-1)^{i} \te^{\si}(A_0,\ldots,A_{i-1},A_{i+1},\ldots,A_{q+1})(a_{1},\ldots,m^{U_{p}}(a_{i},a_{i+1}),\ldots,a_{1})\\ 
&+ (-1)^{q+1}m^{U_{0}}(\te^{\si}(A_0, \ldots, A_{q})(a_{1},\ldots,a_{q}), \si\hs(a_{q+1}))
\end{align*} 
\item 
\begin{align*}
d_{1}(\te)^{\si}(A)(a) &= (-1)^{p+q+1} m^{U_{0}}(c^{\si,1,A_{0}} , u_{1}\st(\te^{\partial_{0}\si}(A)(a))) \\
&+ \sum_{i=1}^{p} (-1)^{p+q+1+i} m^{U_{0}}(\te^{\partial_{i}\si}(A)(a), \eps^{\si,i,A_{q}}) \\
&+ (-1)^{q} m^{U_{0}}(c^{\si,p,A_{0}} , \te^{\partial_{p+1}\si}(A)(u_{p+1}\st a_{1},\ldots,u_{p+1}\st a_{q}))
\end{align*}
\item 
\begin{align*}
d_{j}(\te)^{\si}(A)(a) = \sum_{\substack{ {r} \in P(R_{p}(\si))\\ \be \in \Ss_{q-j+1,j-1}}} (-1)^{{r}+\be+q-j+1} m^{U_{p+j}}(c^{\si,p,A_{0}},\te^{L_{p}(\si)}(B)(\be(a,r))
\end{align*}
\end{itemize} 
for $\si = (u_{1},\ldots,u_{p+j})\in N_{p+j}(\U)(U_{0},U_{p+j})$, $a = (\fromto{a}{q-j+1})$ where $a_{i} \in \A(U_{p+j})(A_{i},A_{i-1})$ and such that $B$ is the sequence of objects underlying $\be(a,r)$.
\end{mydef}


We will also be interested in the subcomplex $\CGSnr(\A) \sub \CGS(\A)$ of normalized and reduced cochains which is shown to be quasi-isomorphic to the GS complex (see \cite[Prop. 3.16]{DVL}). A simplex $\si= (\fromto{u}{p})$ is \emph{reduced} if $u_i = 1_{U_i}$ for some $1 \leq i \leq p$. A cochain $\te = \left(\te^{\si}(A)\right)_{\si,A} \in \CGS(\A)$ is \emph{reduced} if $\te^\si(A) = 0$ for every reduced simplex $\si$. A simplex $a=(\fromto{a}{q})$ in $\A(U)$ is \emph{normal} if $a_i= 1^{U}$ for some $1 \leq i \leq q$. A cochain $\te$ is \emph{normalized} if $\te^{\si}(A)(a) = 0$ for every normal simplex $a$ in $\A(U_p)$. We come back to this in section \S \ref{parc}.

Elements of the GS complex have a neat geometric interpretation as rectangles: for $\te \in \textbf{C}^{p,q}(\A)$ and the data $(\si,A,a)$ from above, we can represent $\te^\si(A)(a)$ as the rectangle of data
$$\input{GSelement.tikz}$$
Similarly, we can draw different components of the differential $d$ using rectangles, providing more insight in its rather technical definition. For the hochschild component $d_0$ we have
\begin{align*}
d_0(\te)^\si(A) = \quad \input{GSd0_1.tikz} &\quad + \quad \sum_{i=1}^q (-1)^i \quad \input{GSd0_i.tikz} \\
&\quad + (-1)^{q+1} \quad \input{GSd0_last.tikz}
\end{align*}
The first component $d_1$ can similarly be drawn as
\begin{align*}
d_1(\te)^\si(A) = \quad (-1)^{p+q+1} \input{GSd1_1.tikz} \; &+ \quad \sum_{i=1}^q (-1)^{p+q+i+1} \quad \input{GSd1_i.tikz} \\
&\quad + (-1)^{q} \quad \input{GSd1_last.tikz}
\end{align*}
Finally, we will draw $d_2$ as an example from which it is easy to deduce the higher components $d_j$ for $j>2$. Namely, we have
\begin{align*}
d_2(\te)^\si(A) = \quad \sum_{i=1}^{q} (-1)^{r + \be + q-2+1} \quad \input{GSd2.tikz} 
\end{align*}
for shuffle $\be(q)= i$, $\be(s) = s$ for $s<i$ and $\be(s) = s+1$ for $s\geq i$, and formal path $r=((u_p,u_{p+1}),1)$. Note in particular that we can draw $\be(a,r)$ as follows
$$\input{beta.tikz}$$
where $b_s= a_{\be\inv(s)}$ for $s\neq i$, and $b_i = c^{u_{p+1},u_{p+2}}(A_{i-1})$.

We will use this rectangular interpretation as a guide in the next sections.

\subsection{Endomorphism operad of a prestack}\label{parend}

Although the GS-complex does not have partial compositions $\circ_{i}$, its elements $\te = (\te^{\si}(A))_{(\si,A)}$ consist of parts that lie in the endomorphism operad $\End(\A)$.

\begin{mydef}
Let $\Ob(\U,\A)$ be the set consisting of the triples $(U,A,A')$ for $U \in \U$ and $A,A'\in \A(U)$, then we define the $\Ob(\U,\A)$-colored operad $\End(\A)$ as
$$\End(\A)((U_{1},A_{1},A_{1}'),\ldots,(U_{n},A_{n},A'_{n});(U,A,A')) := \Hom(\bigotimes_{i=1}^{n} \A(U_{i})(A_{i},A'_{i}), \A(U)(A,A')) $$
with partial compositions defined by composition of linear maps.
\end{mydef}

\begin{opm}
Note that $\te^{\si}(A) \in \End(\A)((U_{p},A_{1},A_{0}),\ldots,(U_{p},A_{q},A_{q-1});(U_{0},\si\hs A_{q},\si \st A_{0}))$.
\end{opm}

\subsection{The Operad $\Patch$}\label{parP}

In this section we define an $\N\times \N$-colored operad $\Patch\sub \mNSOp \times \NSOp$. Its elements encode concrete (planar) patchworks of rectangles of size $(p_{i},q_{i})$ to form a rectangle of size $(p,q)$.  

\begin{mydef}
Let $\Patch((q_{1},p_{1}),\ldots,(q_{n},p_{n});(q,p))$ consists of the elements
 $(X,J)\in \mNSOp(\nth{q};q)\times\NSOp(\nth{p};p)$ such that 
\begin{enumerate}
\item $a <_{J} b \Longrightarrow a \tre_{X} b$
\item $a <_{X} b \Longrightarrow b \tre_{J} a$
\end{enumerate}
\end{mydef}
\begin{opm}
Note that in order for $\Patch$ not to be empty, we need to allow a multiplication in one of its coordinates which is not present in the other coordinate. 
\end{opm}
This has a neat geometric interpretation as well: a $(p,q)$-rectangle has $p$ inputs on the right-hand side, $q$ inputs on top and a single output on respectively the bottom and the left-hand side
$$\input{GSrectangle.tikz}$$
We then interpret a patchwork $(X,J)$ as an ordering of these rectangles: the first coordinate $X$ represent the vertical ordering (from top to bottom) and the second coordinate $J$ the horizontal ordering (from right to left). The multiplications $m$ form a single exception: they appear only vertically, thus we draw them as \emph{flat} rectangles, that is, having no horizontal input and output. From this perspective, the conditions impose planarity on the patchwork such that we have
\begin{alignat*}{2}
&\text{below} <_X \text{above} \quad &&\text{above} \tre_J \text{below}
 \\
  &\text{left} \tre_X \text{right}  \quad &&\text{left} <_J \text{right}
\end{alignat*}

Note that when we write down a patchwork using rectangles, possible `open spaces' can appear. Moreover, it is possible that multiple rectangles are vertically the `lowest' elements due to the insertion of multiplication elements $m$. However, horizontally there can only appear a single most left rectangle which is (horizontally) connected to all other rectangles. We give an example.
\begin{vb}\label{ExPatch}
The following pair determines an element in $\Patch((3,5),(3,2),(2,1),(0,2);(6,7))$
$$(X,J) = \left(\quad \input{ExPatch_1.tikz} \quad, \quad \input{ExPatch_2.tikz}\quad\right)$$
which we can draw as the following patchwork of rectangles
$$\input{ExPatch_3.tikz}$$
where the grey areas denote the open spaces.
\end{vb} 

\begin{lemma}
$\Patch$ is a suboperad of $\mNSOp \hotimes \NSOp$.
\end{lemma}
\begin{proof}
Let $(X,J)\in \Patch((q_{1},p_{1}),\ldots,(q_{n},p_{n});(q,p))$ and $(X',J')\in  \Patch((q'_{1},p'_{1}),\ldots,(q'_{m},p'_{m});(q_{i},p_{i}))$ and we set $X'':= X \circ_{i} X'$ and $J'':= J \circ_{i} J'$. Let $(\al,\be)$ be the extension of $n$ by $m$ at $i$, then for $a,b\in \lh m \rh$ we compute 
$$\al(a) <_{X''} \al(b) \iff a <_{X'} b \Longrightarrow b \tre_{J'} a \iff \al(b) \tre_{J''} \al(a)$$
and for $c,d\notin \Image(\al)$ we compute
$$ c <_{X''} d \iff \be c <_{X} \be d \Longrightarrow \be d \tre_{J} \be c \iff d \tre_{J''} \be c$$
For $c\notin \Image(\al)$ and $b\in \lh m \rh$, we have
$$ c <_{X''} \al(b) \Longrightarrow \be c <_{X} i \Longrightarrow i=\be \al(b) \tre_{J}  \be\al(b) \Longrightarrow \al(b) \tre_{J''} c$$
and the same reasoning shows $\al(b) <_{X''} c \Longrightarrow c \tre_{J''} \al(b)$. Completely symmetrically, this also shows that $c<_{J''} d\Longrightarrow c \tre_{X''} d$ for $c,d\in \lh n+m-1 \rh$.
\end{proof}

We again compile the colored operad $\Patch$ to obtain a graded non-colored operad
$$\Patch_s(n) \sub \prod_{\substack{\nth{q},q \\ \nth{p} }} \Patch((q_{1},p_{1}),\ldots,(q_{n},p_{n});(q,p))$$
where an element $x\in \Patch((q_{1},p_{1}),\ldots,(q_{n},p_{n});(q,p))$ is graded as 
$$|x| = \sum_{i=1}^{n}(q_{i} +p_{i} -1) - (q+p-1)$$ and $\Patch_s(n)$ is generated as a $k$-module by the sequences of constant degree. Its composition is derived from $\Patch$ where it is set to $0$ when the colors do not match.
 Note in particular that the $\Ss_{n}$-action on $\Patch(n)$ is affected by this grading: permuting two vertices $i$ and $j$ introduces a sign $(-1)^{(q_{i}+p_{i}-1)(q_{j}+p_{j}-1)}$.
 
 \begin{lemma}\label{lempatchdg}
$\Patch_s$ is a dg-suboperad of $(\mNSOp_{st} \hotimes \NSOp_s,(\partial_{D},Id))$. 
\end{lemma}
\begin{proof}
It suffices to see that the elements $(D_q,1) \in \Patch((q,p);(q+1,p))$ for every $p,q\in \N$.
\end{proof}

\subsection{The morphism $\Patch_s \longrightarrow \End(s\CGS(\A))$}\label{parPGS}

In this section we make the GS-complex $\CGS(\A)$ of a prestack $\A$ into a $\Patch_s$-algebra. We do so by making its underlying bicomplex $\textbf{C}^{\bu,\bu}(\A)$ into a $\Patch$-algebra. 
We first fix some notations.

\begin{mydef}\label{subsimplex}
Let $\si = (U_{0} \overset{u_{1}}{\rightarrow} U_{1} \rightarrow  \ldots   \overset{u_{p}}{\rightarrow}  U_{p} )$ be a $p$-simplex in the category $\U$ and $\zeta : [p'] \longrightarrow [p]$ a non-decreasing map (or equivalently a non-decreasing sequence), then let $\overline{\zeta}$ be the reflection of $\zeta$, that is, 
$$  \overline{\zeta}(t) := p - \zeta(p'-t)$$
 and define
$$\zeta(\si):= N_\bu(\U)(\overline{\zeta})(\si)$$
  a $p'$-subsimplex of $\si$, where $N_\bu(\U)$ denotes the nerve construction on $\U$.
\end{mydef}
\begin{opm}
Note that we apply the reflection as we count the horizontal inputs of a patchwork from top to bottom (see example \ref{ExPatch}) instead of bottom to top (see further, example \ref{ExPatching}).
\end{opm}

Given a patchwork $(X,J) \in \Patch$, we now determine which simplices we need to fill in the `open spaces' in between the rectangles. We first sketch the idea.

Given a simplex $\si$ in $\U$ and a vertex $a$, we want to determine two sorts of simplices: for every vertical input $i= I(a,b)$ for some vertex $b$, we want to determine a simplex $\si(a,b)$ that we place between them. For the other vertical inputs $1 \leq i \leq q_a$, we determine a simplex $\si_a(i)$ to place on top of $a$ at input $i$. To do so, we determine the set of left-most vertices which do not ``surpass'' the $i$th input and that lie higher than vertex $a$. In the drawing below, this set consists of the vertices $e_1,e_2$ and $e_3$. To calculate $\si(a,b)$, we restrict this set to those vertices that still lie below vertex $b$, in this case, the vertices $e_2$ and $e_3$. 
 $$\input{Simplices.tikz}$$
We observe that each element of the GS complex composes in $\U$ the subsimplex corresponding to its horizontal inputs.
Hence, using our auxiliary set, we contract the corresponding subsimplices and obtain $\si(a,b)$ and $\si_a(i)$.

Note that we have not yet treated the multiplications $m$.   In order to do so, we have to add the following complexity. Let $X = [ I \circ_{n+1} m \circ_{n+1}  \ldots \circ_{n+1} m]$ where $I$ is an indexed tree with $n+k$ vertices, then we call a vertex $a$ of $I$ \emph{non-plugged} in $X$ if in $X$ it is not inserted by a multiplication element $m$. We continue with the above chosen representation of $X$ where $a$ non-plugged is equivalent to stating $a\leq n$.

\begin{mydef}
We define a function $ \mathop \downarrow: \lh n+k \rh \longrightarrow [ n ]$ on the vertices of $I$ which associates to every vertex $a$ the closest non-plugged vertex in $X$ under or equal to $a$, or $0$ if no such vertex exists. Concretely, 
$$ \mathop \downarrow a := \max_{<_{I}} \{ y \in \lh n \rh : y \leq_{I} a \} \text{ and } \mathop \downarrow a = 0 \text{ if the set is empty } $$
We also set $a \tre_{J} 0$ for every vertex $a$ and define
$$\zeta_{J,a} := \zeta_{J,\mathop\downarrow a} \text{ and } \zeta_{J,0}:= p$$
where $p$ is the total number of inputs of $J$.
\end{mydef}
\begin{opm}
This is clearly independent of the representative $I$ of $X$. Moreover, $\mathop \downarrow$ is for the given representative $I$ the identity on $\lh n \rh$.
\end{opm}

Next, we determine the auxiliary set.

\begin{mydef}
Consider a vertex $a $ of $I$ and let $b_{1} \tre_{I} \ldots \tre_{I} b_{t}$ be the children of $a$ in $I$ lying in $\lh n \rh$  with $i_{s} := I(a,b_{s})$. We then define
$$ L_{a}(i) = \begin{cases}\{ e \in \lh n \rh :  e \not \tre_{I} \mathop \downarrow a \text{ and } e \tre_{J} \mathop \downarrow a \} & i < i_{1} \\
 \{ e \in \lh n \rh : b_{s} \tre_{I} e \text{ and }  e \tre_{J} \mathop \downarrow a \} & i_{s} \leq i < i_{s+1} \\
 \{ e \in \lh n \rh : b_{t} \tre_{I} e \text{ and }  e \tre_{J} \mathop \downarrow a \} & i_{t} \leq i
  \end{cases} $$
for $i \in [ q_{a} ]$, and
 $$L(a,b_{s}) :=   \{ e \in \lh n \rh : b_{s} \tre_{I} e \text{ and } b_{s} \tre_{J} e \tre_{J} \mathop \downarrow a \}$$
 and let $\min L_{a}(i)$ and $\min L(a,b_s)$ be respectively the set of $<_{J}$-minimal elements of $L_{a}(i)$ or $L(a,b_s)$.
\end{mydef}
\begin{opm}
Remark that $L_{a}(i),L(a,b_s) \sub \lh n \rh$ and thus that it contains only vertices which are not plugged by $m$. By default, we will set the subsimplex underneath the plugged children of $a$ as empty (see definition \ref{SimplexInbetween}). 
\end{opm}
\begin{opm}
Note that the condition $e \not \tre_{I} \mathop \downarrow a$ appearing in the first case becomes superfluous in the others.
\end{opm}

\begin{mydef}
Let $a$ be a vertex of $I$ and $\min L_{a}(i) = \{ e_{1} \tre_{J} \ldots \tre_{J} e_{l} \}$, then we have the sequence of inequalities
 $$0 \leq \zeta_{J,e_{1}}(0) \leq \zeta_{J,e_{1}}(p_{e_{1}}) \leq \ldots \leq \zeta_{J,e_{l}}(p_{e_{l}}) \leq \zeta_{J,a}(0)$$
 and thus the non-decreasing sequence
 $$(0,1, \ldots, \zeta_{J,e_{1}}(0), \zeta_{J,e_{1}}(p_{e_1}),\ldots,\zeta_{J,e_{l}}(0),\zeta_{J,e_{l}}(p_{e_{l}}), \ldots, \zeta_{J,a}(0)) $$
 from which we delete $\zeta_{J,e_i}(p_{e_i})$ if $\zeta_{J,e_i}(p_{e_i}) =\zeta_{J,e_{i+1}}(0)$ or $\zeta_{J,e_l}(p_{e_l}) = \zeta_{J,a}(0)$. 
 This defines a subsimplex $\si_{a}(i)$ of $\si$ by definition \ref{subsimplex}.
\end{mydef}

\begin{mydef}\label{SimplexInbetween}
Let $a,b$ be vertices of $I$ such that $b$ is a child of $a$, and $\min L(a,b) = \{ e_{1} \tre_{J} \ldots \tre_{J} e_{l} \}$, then we have the sequence of inequalities
\begin{itemize}
\item if $i = I(a,b)$ for some vertex $b \in \lh n \rh$ (non-plugged)
 $$  \zeta_{J, b}(p_b) \leq \zeta_{J,e_{1}}(0) \leq \zeta_{J,e_{1}}(p_{e_{1}}) \leq \ldots \leq \zeta_{J,e_{l}}(p_{e_{l}}) \leq  \zeta_{J,a}(0)$$
 \item if $i = I(a,b)$ for some vertex $b > n$ (plugged)
 $$ \zeta_{J, b}(p_b) = \zeta_{J, a}(0)$$
\end{itemize}
and thus the non-decreasing sequence
$$(\zeta_{J,b}(p_b),\ldots, \zeta_{J,e_{1}}(0), \zeta_{J,e_{1}}(p_{e_1}),\ldots,\zeta_{J,e_{l}}(0),\zeta_{J,e_{l}}(p_{e_{l}}), \ldots, \zeta_{J,a}(0))$$
from which we delete $\zeta_{J,e_i}(p_{e_i})$ if $\zeta_{J,e_i}(p_{e_i}) =\zeta_{J,e_{i+1}}(0)$ or $\zeta_{J,e_l}(p_{e_l}) = \zeta_{J,a}(0)$. We also delete $\zeta_{J,b}(p_b)$ if it equals $\zeta_{J,e_1}(0)$.
 This defines a subsimplex $\si(a,b)$ of $\si$ by definition \ref{subsimplex}.
\end{mydef}
We consider an example.
\begin{vb}
Given the simplex $\si=(u_1,\ldots,u_8)$ and the following patchwork of rectangles
$$ \scalebox{1}{$\input{exampleL_2.tikz}$}$$
we analyse the case for rectangle $6$: we have
\begin{alignat*}{1}
&L_6(0) = L_6(1) = \{1,3,4,5,7\} \text{ and } \min L_6(0) = \min L_6(1) = \{ 1,3,5,7\},\\
&L_6(2) = L_6(3) = \{3,4,5\} \text{ and } \min L_6(2) = \min L_6(3) = \{3,5\}\\
&L(6,1) = \{ 3 \} \text{ and } \min L(6,1) = \{ 3\}
\end{alignat*}
and thus
\begin{align*}
\si_6(0) = \si_6(1) &= (u_3u_2, u_5u_4, u_6, u_8u_7),\\
\si_6(2) = \si_6(3) &= (u_3u_2, u_4,u_5,u_6,u_8u_7),\\
\si(6,1) &= (u_3u_2).
\end{align*}
\end{vb}

Now, we can assemble for every element of $\Patch$ a concrete patchwork of elements of $\CGS(\A)$ where the first coordinate determines a vertical patching using the operadic structure and the second component determines the horizontal patching to fill in and align the corresponding simplices. 
\begin{constr}\label{LL}
Given $\left(X,J\right) \in \Patch\left(\left(q_{1},p_{1}\right),\ldots,\left(q_{n},p_{n}\right);\left(q,p\right)\right)$ and $\te_{i} \in \textbf{C}^{p_{i},q_{i}}\left(\A\right)$, then we set $\te_{s} = m \in \textbf{C}^{0,2}\left(\A\right)$ for $s = n+1,\ldots,n+k$. 

Let $\si$ be a $p$-simplex in $\U$ and $A = \left(A_{0},\ldots,A_{q}\right)$ $(q+1)$-tuple of objects in $\A\left(U_{p}\right)$, then we define for every vertex $a$ in $I$
$$\Te_{a} := \te_{a}^{\zeta_{J,a}\left(\si\right)}\left(\si_{a}\left(0\right)\hs A_{\zeta_{I,a}\left(0\right)},\ldots,\si_{a}\left(q_a\right)\hs A_{\zeta_{I,a}\left(q_{a}\right)}\right)$$
and for every $i\in \lh q_{a} \rh$ we make the compositions
\begin{itemize}
\item if $i = I\left(a,b\right)$ for some vertex $b$,
$$ \Te_{a} \circ_{i} \left(\si\left(a,b\right)\hs \circ_1 \Te_{b}\right)$$
\item otherwise,
$$ \Te_{a} \circ_{i} \si_{a}\left(i\right)\hs\left(A_{\zeta_{I,a}\left(i-1\right)},A_{\zeta_{I,a}\left(i\right)}\right)$$
\end{itemize}
All these compositions together define 
$$\LL\left(X,J\right)\left(\nth{\te}\right)^\si\left(A\right) \in \Hom\left(\bigotimes_{i=1}^{q} \A\left(U_{p}\right)\left(A_{i},A_{i-1}\right); \A\left(U_{0}\right)\left(\si\hs A_{q}, \si\st A_{0}\right)\right)$$
\end{constr}

\begin{lemma}
Construction \ref{LL} is independent of the representative $I$ of $X$. \end{lemma}
\begin{proof}
It suffices to verify the relation on the formal multiplication elements $m$ in $\mNSOp$. This follows directly from the associativity of the local composition $m^U$ of the category $\A(U)$ for every $U \in \U$.
\end{proof}

Let us work out an example.

\begin{vb} \label{ExPatching}
Consider the patching $(X,J)$ from example \ref{ExPatch}. Let $\te_{1} \in \textbf{C}^{5,3}(\A), \te_{2} \in \textbf{C}^{2,3}(\A), \te_{3} \in \textbf{C}^{1,2}(\A)$ and $\te_{4} \in \textbf{C}^{2,0}(\A)$, then we compute $\LL(X,J)(\te_{1},\te_{2},\te_{3},\te_{4}) \in \textbf{C}^{7,6}(\A)$. Given the simplex $(u_{1},\ldots,u_{7})\in N_{7}(\U)(U_{0},U_{p})$ and the objects $(A_{0},\ldots,A_{6}) \in \A(U_{p})$, we first compute 
\begin{align*}
\Te_{1} &= \te_{1}^{(u_{1},u_{3}u_{2},u_{4},u_{6}u_{5},u_{7})}(A_{0},A_{1},A_{2},A_{3}) \\
\Te_{2} &= \te_{2}^{(u_{2},u_{3})}(u_{4}\st(u_{6}u_{5})\st u_{7}\st A_{3}, u_{4}\st(u_{6}u_{5})\st u_{7}\st A_{5},u_{4}\st(u_{6}u_{5})\st u_{7}\st A_{6},  u_{4}\st u_{5}\st u_{6}\st u_{7}\st A_{6})\\
\Te_{3} &= \te_{3}^{(u_{6}u_{5})}( u_{7}\st A_{3},u_{7}\st A_{4},u_{7}\st A_{5}) \\
\Te_{4} &= \te_{4}^{(u_{5},u_{6})}(u_{7}\st A_{6})
\end{align*}
Then, given $(a_{1},\ldots,a_{6})$ where $a_{i} \in \A(U_{p})(A_{i},A_{i-1})$, we finally compute
$$m^{U_{0}}(\Te_{1}(a_{1},a_{2},a_{3}), u_{1}\st \Te_{2}( u_{4}\st \Te_{3}( u_{7}\st(a_{4}), u_{7}\st(a_{5}) ), u_{4}\st (u_{6}u_{5})\st u_{7}\st(a_{6}) , u_{4}\st \Te_{4})$$
which we can draw as follows
$$\scalebox{0.8}{$\input{Example_L.tikz}$}$$
\end{vb}

\begin{prop}\label{LLComp}
$$\LL(X,J) \circ_{a} \LL(X',J') = \LL(X \circ_{a} X', J \circ_{a} J')$$
\end{prop}
\begin{proof}
We can assume without loss of generality that $a = n$ as $\LL$ is clearly equivariant.

Let $(X'',J'') := (X,J) \circ_{n} (X',J')$, then we add $'$ or $''$ to denote the notions associated to $(X',J')$ or $(X'',J'')$. Let $I$ and $I'$ be the underlying trees representing $X$ and $X'$ having respectively $n+k$ and $m+k'$ vertices, then let $(\al,\be)$ be the extension of $n+k$ by $m+k'$ at $n$. Let $(\overline{\al},\overline{\be})$ be the extension of $n$ by $m$ at $n$.

We compute
\begin{equation}
\LL(X,J)(\te_{1},\ldots,\te_{n-1},\LL(X',J')(\te_{n},\ldots,\te_{n+m-1}))^{\si}(A) \label{PatchProof1}
\end{equation}

and show that it equals 
\begin{equation}
\LL(X'',J'')(\fromto{\te}{n+m-1})^{\si}(A) \label{PatchProof2}
\end{equation}

for $\si \in N_{p}(\U)$ and $A=(A_{0},\ldots,A_{q})$ objects in $\A(U_{p})$.

It is clear that per construction the blocks involved are composed according to $X'' = X \circ_{n} X'$. Hence it suffices to verify that they correspond to the blocks $\Te_x''$ in $\LL(X'',J'')$ and that the functors used to fill in the open spaces, agree.

First, for $x$ a non plugged vertex of $I''$ in $X''$, it is clear that $\Te''_x$ is either $\Te_{\be(x)}$, or $\Te_{\al \inv(x)}'$	 evaluated at $\si' = \zeta_{J,n}(\si)$. Next, we verify the simplices $\si''_i(x)$. For its $i$th input, we have the following two cases:
\begin{itemize}
\item if $x$ does not lie in the image of $(X',J')$, then $\si_{\be x}(i) = \si''_x(i)$ because if $n \in \min L_{\be x}(i)$ then it is replaced by $\overline{\al}(r')$ for $r'$ the root of $J'$ for which holds $\zeta_{J'',\overline{\al}(r')}=\zeta_{J,n}\zeta_{J',r'}$.
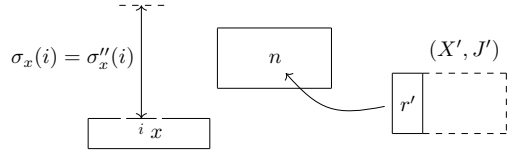
$$\scalebox{0.8}{$\input{Proof_Patch_1.tikz}$}$$
As a result, in both (\ref{PatchProof1}) and (\ref{PatchProof2}) we have the term $\Te_{\be x} \circ_i \si_{\be x}(i)\hs$.
\item if $x$ is part of $(X',J')$, i.e. $x= \al(x')$ for some vertex $x'$, then $\min L''_{\al(x')}(i)$ is the union of $\min L'_{x'}(i) $ and $\min L_{n}(i')$ for some $i'$. Hence, we obtain the concatenation of $\si'_{x'}(i)$ for $\si' = \zeta_{J,n}(\si)$ and $\si_n(i')$. As $\zeta''_{J''} = \zeta_{J}\circ_n \zeta_{J'}$, this corresponds exactly to $\si''_{x}(i)$.
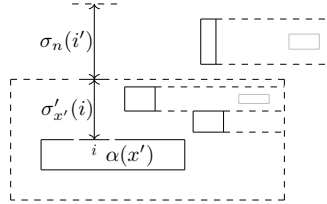
$$\scalebox{0.8}{$\input{Proof_Patch_2.tikz}$}$$
Hence, the corresponding term in both calculations agrees.
\end{itemize}
Next, we calculate $\si''(x,b)$ for $b$ a child of $x$ in $(X'',J'')$ that is not plugged. We again have three cases
\begin{itemize}
\item if both $x$ and $b$ lie either outside or inside the image of $(X',J')$, then clearly $\si''(x,b) = \si(\be x,\be b)$ or $\si'(\al\inv x ,\al \inv b)$ for $\si' = \zeta_{J',n}(\si)$ due to the previous reasoning and thus the terms agree.
\item if $b$ lies in the image of $(X',J')$, i.e. $b =\al(b')$, and $x$ does not, then $b'$ is clearly the root of $X'$. As a result, $\si''(x,b) = \si(\be x, n)$ and thus the terms agree.
\item if $x$ lies in the image of $(X',J')$, i.e. $x = \al(x')$, and $b$ does not, then $\min L''_{\al(x')}(i)$ is the union of $\min L'_{x'}(i) $ and $\min L_{n}(i')$ for some $i'$. 
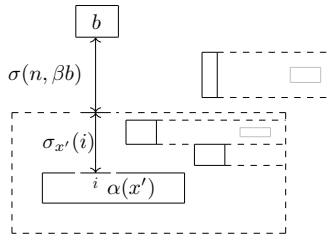
$$\scalebox{0.8}{$\input{Proof_Patch_3.tikz}$}$$
Hence, we obtain in (\ref{PatchProof1}) the concatenation of $\si(n,\be b)$ and  $\si'_{x'}(i)$ for $\si' = \zeta_{J,n}(\si)$, which corresponds exactly to $\si''(x,b)$.
\end{itemize}
In case either $x$ or $b$ is plugged, we possibly have to apply the functorial property of the restrictions, i.e. $u\st \circ m^U = m^V \circ (u \st \otimes u \st)$ for $u:V \rightarrow U$ in $\U$, to pull down $\Te_{\be x} = m^{U_{\zeta_{J,\be x}(0)}}$ or $\Te'_{\al \inv x} = m^{U_{\zeta_{J,n}\zeta_{J',\al \inv x}(0)}}$. Specifically, in the following cases
\begin{itemize}
\item let $\be x$ lie on top of $n$ in $(X,J)$ and $\mathop \downarrow x = \al(y)$ for some vertex $y$ of $(X',J')$. 
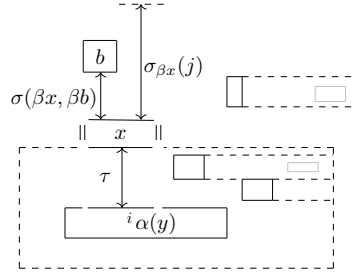
$$\scalebox{0.8}{$\input{Proof_Patch_4.tikz}$}$$
In this case, $\Te_{\be x}= m^{U_{\zeta_{J,n}(0)}}$ occurs in (\ref{PatchProof1}) and $\Te''_x = m^{U_{\zeta_{J,n}\zeta_{J',y}(0)}}$ occurs in (\ref{PatchProof2}). Using functoriality, in (\ref{PatchProof1}) we equivalently have $\tau\hs \circ m^{U_{\zeta_{J,n}(0)}} = m^{U_{\zeta_{J,n}\zeta_{J',y}(0)}} \circ ( \tau \hs \otimes \tau \hs)$ for an appropriate simplex $\tau$. As a result they agree. 

Next, it is clear from the drawing that $\si''_x(j)$ is the concatenation of $\si_{\be x}(j)$ and $\tau$. Moreover, for some vertex $b$, we have $\si''(x,b)$ as the concatenation of $\si(\be x ,\be b)$ and $\tau$, except in the case that $b$ is plugged as well. In the latter case, we can also pull $\Te_{\be b}$ in  (\ref{PatchProof1}) down to $\Te_{\be x}$ and obtain $m^{U_{\zeta_{J,n}\zeta_{J',y}(0)}} = \Te''_x = \Te''_b$ as in (\ref{PatchProof2}).  

\item let $\be x$ lie on top of $n$, but $\mathop \downarrow x  \notin \Image(\al)$.
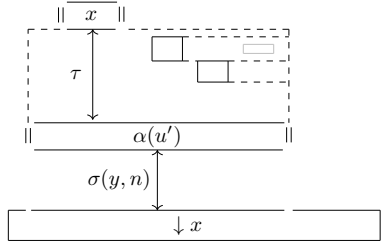
$$\scalebox{0.8}{$\input{Proof_Patch_5.tikz}$}$$
Again, we can pull down $\Te_{\be x}$ in (\ref{PatchProof1}) past both functors $\tau\hs$ and $\si(y,n)\hs$ and obtain $m^{U_{\zeta_{J,\be \mathop \downarrow y}(0)}}= \Te''_x$. The same reasoning as before also holds for the functors $\si''_x(j)$ and $\si''(x,b)$ in (\ref{PatchProof2}) and its counterparts $\si_{\be x}(j)$ and $\si(\be x ,\be b)$ in (\ref{PatchProof1}).
\item The case where $x$ lies in the image of $(X',J')$ such that $\mathop \downarrow  x \notin \Image(\al)$, is analogous to the previous one.
\end{itemize}
This finishes the proof.
\end{proof}

\begin{theorem} \label{thmPGS}
We obtain a morphism of dg-operads
 $$\LL: \Patch_s \longrightarrow \End(s\CGS(\A),d_{0}).$$
\end{theorem}
\begin{proof}
The map $\LL: \Patch \longrightarrow \End(\textbf{C}^{\bu,\bu}(\A))$ is clearly equivariant and thus it is a morphism of operads due to proposition $\ref{LLComp}$. Hence, the induced map $\LL: \Patch_s \longrightarrow \End(s\CGS(\A))$  is a morphism of graded operads. Moreover, it is a morphism of dg-operads as $\LL(D,1) = d_{0}$.
\end{proof}

\subsection{The morphism $\Quilt \longrightarrow \Patch_s$}\label{parquiltpatch}

In \cite{hawkins}, Hawkins defines a suboperad $\Quilt \subseteq \FS\hotimes \Bracee$ for which 
$\Quilt(n)$ is the free $k$-module generated by $(W,T) \in \FS(n)\times \Tree(n)$ such that 
\begin{enumerate}
\item $W= \ldots {u} \ldots {v} \ldots \Longrightarrow u \not>_{T} v;$
\item $W= \ldots {u} \ldots {v}\ldots {u} \ldots \Longrightarrow v \tre_{T} u.$
\end{enumerate}
Here, $deg(W,T) := deg(W)$ and the boundary operator is $\partial(W,T) := (\partial W,T)$.

Insightfully, elements of $\Quilt$ can also be drawn as a stacking of rectangles in the plane, as extensively explained in \cite[\S 3.2]{hawkins}. We will use $\Quilt$ in a fundamentally different way by switching the roles of its first and second component, and thus flipping the rectangles on their side. As such, we also draw the elements of $\Quilt$ on their side. We give an example.
\begin{vb}
We consider an example from \cite[Ex. 3.2]{hawkins} and flip it on its side as follows
$$\left( 14234, \quad \input{quilt_tree.tikz} \quad \right)  = \quad \input{quilt.tikz}$$ 
Note the double line above rectangle $4$: this reflects the fact that $3$ is not interposed, otherwise the corresponding word would be $142434$.
\end{vb}

By definition, we have $\Patch_s \subseteq \mNSOp_s \otimes_H \NSOp_{st}$. In this section, we will construct a morphism of operads $\Quilt \longrightarrow \Patch_s$ as a restriction of
$$\bar{\phi} \otimes_H \phi: \FS \otimes_H \Bracee \longrightarrow \mNSOp_s \otimes_H \NSOp_{st}.$$

\begin{lemma}
Let $Q=(W,T) \in \Quilt$, $X \in \Clr(W,\nth{q})$ and $I \in \Clr(T,\nth{p})$, then $(X,I) \in \Patch$.
\end{lemma}
\begin{proof}
Let $u, v \in \lh n \rh$, if $u<_{I} v$, then $u<_{T}v$ and thus $W \neq \ldots {v} \ldots {u} \ldots $ and thus every occurrence of $u$ in $W$ is left of every occurrence of $v$ in $W$. Hence, $u \tre_{X} v$.

The other way around, if $u <_{X} v$, then $W = \ldots {u} \ldots {v} \ldots {u} \ldots$ and thus $v \tre_{T} u$ which is equivalent to $v \tre_{I} u$. 
\end{proof}

We obtain a morphism of graded operads
$$\bar{\phi}\otimes_H \phi : \Quilt \longrightarrow \Patch_s$$
defined as 
$$(\bar{\phi}\otimes_H \phi)((q_1,p_1),\ldots,(q_n,p_n))(W,T) = \sum_{\substack{ X\in \Clr(W,\nth{q}) \\ I \in \Clr(T,\nth{p})}} \sgn_{W}(X)\sgn_{T}(I) (-1)^{\si} (X,I)$$
where the sign $(-1)^{\si}$ is defined as the Koszul sign obtained from switching 
$$\nth{q},p_{1}-1,\ldots,p_{n}-1 \leadsto q_{1},p_{1}-1,\ldots,q_{n},p_{n}-1$$
This is the consequence of the Hadamard product of two graded operads
$$\Patch((q_{1},p_{1}),\ldots,(q_{n},p_{n});(q,p)) \sub \mNSOp(\nth{q};q) \otimes \NSOp(\nth{p};p)$$
where we have switched the order of the inputs. 

Note in particular that this sign corresponds to the sign defined in \cite[\S 4.7]{hawkins} and that we write $\sgn_Q(X,J) := \sgn_W(X) \sgn_T(J) (-1)^{\si}$.

As a direct consequence of Lemma \ref{lempatchdg} we have the following. 

\begin{prop}\label{propQP}
We have a morphism of dg-operads 
$$ \bar{\phi} \hotimes \phi : (\Quilt,\partial) \longrightarrow (\Patch_s,\partial_{(D,1)})$$
\end{prop}

\begin{cor}\label{correp1quilt}
We have a morphism of dg-operads
$$R:=\LL \circ (\bar{\phi}\hotimes \phi) : \Quilt \longrightarrow \End(s\CGS(\A),d_{0})$$
\end{cor}
\begin{proof}
Immediate from Theorem \ref{thmPGS} and Proposition \ref{propQP}.
\end{proof}

Our action of $\Quilt$ on the GS-complex of a prestack is orthogonal to the action constructed in \cite[Thm. 4.26]{hawkins} in the case of presheaves, and thus also new for the latter case. This can be interpreted in a geometric sense: our action encodes a quilt $Q=(W,T)$ as a vertical patchwork according to $W$ and a horizontal patchwork according to $T$. In Hawkins' action their roles are reversed, where the role of the multiplication is filled in by the identity $1^{u,v} : v\st u\st = (uv)\st$. This does not translate to the case of prestacks due to the occurring twists $c^{u,v}: v \st u \st \longrightarrow (uv)^{\st}$.

\section{Incorporating Twists}\label{parc}

The morphism $R:\Quilt \longrightarrow \End(s\CGS(\A))$ from Corollary \ref{correp1quilt} only involves the multiplication $m$ and the functors $f$ of the data of a prestack $(\A,m,f,c)$. In this section, we will incorporate the twists $c$ by adding a formal element with certain relations, resulting in the bounded powerseries operad $\Quiltb[[c]]$. 
In \S \ref{parquiltbcCGS}, we extend $R$ above to a morphism $R_c: \Quiltb[[c]] \longrightarrow \End(s\CGS(\A))$ (see Theorem \ref{thmRc}). In Hawkins' approach to the presheaf case, the initial action of $\Quilt$ on $\End(s\CGS(\A))$, which only involves the restriction maps $f$, is later extended in order to incorporate the multiplications $m$. As far as the structure of both approaches goes, our operad $\Quiltb[[c]]$ is the counterpart of the operad $\mathrm{mQuilt}$ from \cite[Def. 5.2]{hawkins}.

In \cite[\S 7.1]{hawkins}, Hawkins constructs a morphisms $L_{\infty} \longrightarrow \Quilt$ (see \S \ref{parLinftyquilt}). In \S \ref{parLinftyquilt}, we establish a more involved morphism $L_{\infty} \longrightarrow \Quiltb[[c]]$ (see Theorem \ref{maintheorem}) by extending to an infinite series of higher components incorporating the element $c$. 

Putting Theorems \ref{maintheorem} and \ref{thmRc} together, we have thus endowed $s\CGS(\A)$ with an $L_{\infty}$-structure. In the case of presheaves, this coincides on reduced and normalised cochains with the $\Linf$-structure from \cite[Thm. 7.13]{hawkins}. 

In the final section \ref{pardef} we briefly discuss the relation of this structure with the deformation theory of the prestack $\A$.

\subsection{Powerseries operads}

In order to obtain an $L_{\infty}$-structure incorporating twists, we will make use of operads of formal power series.

\begin{mydef}
Let $\Oo$ be a graded operad, then define $\Oo[x]$ as the graded operad generated by $\Oo$ and an element $x$ of degree $t$ and define the subspaces
$$\Oo[x](n,r) := \{ \ga \in \Oo[x](n) : \ga \text{ has } r \text{ occurrences of } x \} \sub \Oo[x](n)$$
which is well-defined as there are no relations on $x$. Define
$$\Oo[[x]](n) := \prod_{r\geq 0} \Oo[x](n,r)$$
with component-wise $\Ss_{n}$-action and write their elements as $\sum_{r\geq0} Q_{r}$ for $Q_{r} \in \Oo[x](n,r)$. 
For every $1\leq k \leq n$ the composition of $\Oo[x]$ descends to a map
$$\Oo[x](n,r)\otimes \Oo[x](n,s) \longrightarrow \Oo[x](n,r+s)$$
which extends to a composition map
$$(\sum_{r\geq 0} Q_{r}) \circ_{k} (\sum_{s\geq 0} P_{s} ) := \sum_{t\geq 0} ( \sum_{i+j=t} Q_{i} \circ_{k} P_{j})$$
We call an element $\sum_{r\geq 0} Q_{r} \in \Oo[[x]]$ \emph{bounded} if the set $\{ deg(Q_{r}) : r \geq 0\}\sub \Z$ is bounded. Let $\Oo_{b}[[x]]$ be the $\Ss$-submodule of bounded series which is graded by the series with coefficients of constant degree.
\end{mydef}

\begin{lemma}
\begin{enumerate}
\item $\Oo[[x]]$ is an operad.
\item $\Oo_{b}[[x]]$ is a graded suboperad of $\Oo[[x]]$.
\item We have a sequence of injective operad morphisms
$$\Oo \hookrightarrow \Oo[x] \hookrightarrow \Oo_{b}[[x]] \hookrightarrow \Oo[[x]] $$
\end{enumerate}
\end{lemma}
\begin{proof}
These are straight-forward computations.
\end{proof}

We call $\Oo[[x]]$ the \emph{operad of powerseries with coefficients in} $\Oo$ and $\Oo_{b}[[x]]$ the \emph{operad of bounded powerseries with coefficients in} $\Oo$.

\begin{mydef}\label{defc}
Consider $\Quilt[x]$, $\Quilt_b[[x]]$ and $\Quilt[[x]]$ for $x$ a $0$-ary element $x\in \Quilt(0)$ of degree $-1$, then let $\Quilt[c]$, $\Quilt_b[[c]]$ and $\Quilt[[c]]$ be their respective quotients under the following relations
\begin{enumerate}
\item $\partial(c) =0$ \label{nattrans}
\item $(12, \begin{tikzpicture}
	\begin{pgfonlayer}{nodelayer}
		\node [style=none] (0) at (-1, 0) {$1$};
		\node [style=none] (1) at (1, 0) {$2$};
		\node [style=none] (2) at (0.525, 0) {};
		\node [style=none] (3) at (-0.55, 0) {};
	\end{pgfonlayer}
	\begin{pgfonlayer}{edgelayer}
		\draw (2.center) to (3.center);
	\end{pgfonlayer}
\end{tikzpicture}
) \circ_{1} c \circ_{1} c = 0$ \label{cocycle}
\item $(W,T)\circ_{i} c = 0$ if $i$ has more than two children in $T$ or $i$ is repeated in $W$ (that is, it has a child in $W$).\label{form}
\end{enumerate}
\end{mydef}
\begin{opm}
In Definition \ref{defc}, \eqref{nattrans} determines that $c$ encodes a natural transformation, \eqref{cocycle} embodies the cocycle condition and \eqref{form} determines the form of $c$. The letter $c$ will always stand for the twist subject to its relations, and should not be confused with a free variable.
\end{opm}

On inspection of $\mQuilt$ from \cite[Def. 5.2]{hawkins}, we see that our conditions on $c$ are a subset of those imposed on $m$ in $\mQuilt$. Hence, we obtain a morphism $\Quilt[c] \longrightarrow \mQuilt$ sending $c$ to $m$.

\subsection{The morphism $\Linf \longrightarrow \Quilt$}\label{parLinftyquilt}

In \cite[Thm 7.8]{hawkins} a morphism $\Linf\longrightarrow \Quilt: l_n \longmapsto L_n^0$ is defined by setting
$$L_{n}^{0} := \sum_{\substack{Q \in \Quilt(n)\\ deg(Q) =n-2}} \sgn(Q) \; Q $$
In particular, this means that for every $n\geq 2$ the equation
\begin{equation}
0 = \partial L_{n}^{0} + \sum_{\substack{p+q = n+1 \\ p,q\geq 2}}  \sum_{\si \in Sh_{p-1,q}} (-1)^{(p-1)q} (-1)^{\si} (L_{p}^{0} \circ_{p} L_{q}^{0})^{\si\inv}  \label{Linf}
\end{equation}
holds. 

An important feature which we will need, is that we can write $L_{n}^{0}$ as the antisymmetrization of elements $P_{n}^{0}$. Namely, we set
$$P_{n}^{0} := \sum_{\substack{Q \in \Quilt(n)\\ deg(Q) = n-2 \\ Q \text{ labelled in } \downarrow \text{order}}} (-1)^{1+\frac{n(n-1)}{2}} Q$$
then we have
$$L_{n}^{0} = \sum_{\si \in \Ss_{n}} (-1)^{\si} (P_{n}^{0})^{\si}.$$
In fact, the $\Linf$-relations translate to the following
\begin{equation}
\partial P_{n}^{0} + \sum_{\substack{p+q = n+1 \\ p,q\geq 2}}  \sum_{j=1}^{p} (-1)^{(p-1)q+ (p-j)(q-1)}P_{p}^{0} \circ_{j} P_{q}^{0} \in \langle Q -Q^{\si} | \si \in \Ss_{n} \rangle \sub \Quilt(n)\label{partialP}
\end{equation} 
where $\langle - \rangle$ denotes `free $k$-module generated by$-$'.

\subsection{The morphism $\Linf \longrightarrow \Quiltb[[c]]$}\label{parLinftyquiltbc}

Next we will define more involved $\Linf$-operations incorporating $c$.

\begin{mydef}\label{P}
For $n+r\geq 2$ we define
\begin{align*}
 P^{r}_{n} &:= \sum_{\substack{1\leq y_{1} < \ldots < y_{r} \leq n+r }}  (-1)^{\fromto{y}{r}} P_{n+r}^{0} \circ_{y_{1}}  c \circ_{y_{2}-1} \hdots \circ_{y_{r}-r+1} c
 \end{align*}
 where $(-1)^{\fromto{y}{r}}$ denotes the sign of the $(r,n)$-shuffle defined by $(\fromto{y}{r})$. Using these we set
 $$L_{n}^{r} := \sum_{\si \in \Ss_{n}} (-1)^{\si}\left(P^{r}_{n}\right)^{\si}$$
 \end{mydef}
 \begin{opm}
 Note that $L_{n}^{r}$ live in $\Quilt[c]$.
 \end{opm}
 Let us compute $P_n^r$ for some low $n$ and $r$.
 \begin{vb}
 In case no elements $c$ are added, we obtain the original $P_n^0$
 \begin{align*}
P_2^0 = \quad \input{P0_2.tikz}\;, \quad P_3^0= \quad \input{P0_3.tikz}\;, \quad P_4^0 = - \quad \input{P0_4_1.tikz}\quad - \quad \input{P0_4_2.tikz} \quad - \quad \input{P0_4_3.tikz}\;, \quad \ldots
\end{align*}
 Similar to how we drew elements from $\mNSOp$ as trees with vertices plugged by $m$, elements of $\Quilt[c]$ can be drawn as quilts with rectangles plugged by $c$. For example, we have
  $$ P_1^1 = \quad \input{P21_1.tikz}\quad - \quad \input{P21_2.tikz} $$
  or 
  \begin{align*}
   P_2^2 &= \quad \input{P42_1.tikz}\quad  - \quad \input{P42_2.tikz}\quad - \quad \input{P42_3.tikz} \quad - \quad \input{P42_4.tikz}\quad - \quad \input{P42_5.tikz}
  \end{align*}
 Note that depending on where we plug the elements $c$ a sign is added.
 \end{vb}
 
This enables us to define the following.

\begin{mydef}
$$L_{n} := \sum_{r\geq 0} L_{n}^{r} \text{ for } n\geq 2 \text{ and } L_{1} := \sum_{r\geq 1} L_{1}^{r}$$
Set 
$$\partial' := \partial + \partial_{L_{1}}$$
where 
$$\partial_{L_{1}}(A) := L_{1}\circ_{1} A - (-1)^{|A|} \sum_{i} A \circ_{i} L_{1}$$
(a derivation by an element) which will be the new differential.
\end{mydef}
\begin{opm}
Note that $L_{n}$ are bounded because their components have constant degree $n-2$. Hence, $L_{n}$ live in $\Quiltb[[c]]$.
\end{opm}
The main theorem of this section is the following.

\begin{theorem}\label{maintheorem}
The map
$$\Linf  \longrightarrow (\Quiltb[[c]],\partial') : l_{n} \longmapsto L_{n}$$
defines a morphism of dg-operads.
\end{theorem}
First we need some lemmas.
\begin{lemma}\label{Lvanish}
For $r \geq 2$ we have $L^{r}_{0} = 0$.
\end{lemma}
\begin{proof}
Due to relation (\ref{form}) in definition \ref{defc} of $c$, there cannot be a vertical composition of $c$. Hence, $L_{0}^{r} =0$ for $r \geq 3$. $L_{0}^{2}=0$ due to condition \ref{cocycle} of the definition of $c$.
\end{proof}
\begin{opm}
Note that we have used $2$ out of the three conditions on $c$ to prove this lemma.
\end{opm}

The following lemma extends the $\Linf$-equation of $(L^{0}_{n})_{n}$ for higher $L_{n}^{r}$.
\begin{lemma}\label{mainlemma}
For $n\geq 1$ and $r\geq 0$, we have
$$-\partial(L_{n}^{r})= \sum_{\substack{i+j=r \\ k+l=n+1 \\ k+i \geq 2 \\ l+j \geq 2}}\sum_{\substack{\chi \in Sh_{k-1,l}}}(-1)^{\chi}(-1)^{(k-1)l}\left(L_{k}^{i} \circ_{k} L_{l}^{j}\right)^{\chi\inv}.$$

\end{lemma}
\begin{proof} 
By applying $\partial$ and using equation (\ref{partialP}) and $\partial(c)=0$, we deduce
$$-\partial(L^{r}_{n}) = \sum_{\substack{p+q = n+r+1\\ p,q \geq 2}} \sum_{\si \in \Ss_{n}} \sum_{\substack{1\leq y_{1} < \ldots < y_{r} \leq n+r \\ x=1,\ldots,p}} (-1)^{\fromto{y}{r}} (-1)^{\si} (-1)^{(p-1)q}(-1)^{(p-x)(q-1)}\left(P_{p}^{0}\circ_{x} P^{0}_{q} \left(\circ_{y_{e}-e+1}  c \right)_{e} \right)^{\si}$$
Given $1 \leq y_{1} < \ldots < y_{r} \leq n+r$ and $1 \leq x \leq p$, we have a subdivision into two groups where $\circ_{y_{s}-s+1}c$ is inserted into either $P^{0}_{p}$ or $P^{0}_{q}$. Hence, if we are also given a permutation $\si \in \Ss_{n}$, then we show that there exists unique integers $i+j=r, k+l=n+1$, indices $1 \leq z_{1} < \ldots < z_{i} \leq k+i$ and $1 \leq z'_{1} < \ldots < z'_{j} \leq j+l$, permutations $\tau \in \Ss_{k+i},\tau'\in \Ss_{j+l}$ and a shuffle $\chi \in Sh_{k-1,l}$ such that
\begin{multline}\label{signs}
(-1)^{\fromto{y}{r}} (-1)^{\si} (-1)^{(p-1)q}(-1)^{(p-x)(q-1)}\left(P_{p}^{0}\circ_{x} P^{0}_{q} \left(\circ_{y_{e}-e+1}  c \right)_{e} \right)^{\si} = \\
(-1)^{(\fromto{z}{i})+(\fromto{z'}{j})}(-1)^{\chi+\tau+\tau'}(-1)^{(k-1)l}\left(\left(P_{k+i}^{0}\left(\circ_{z_{e}-e+1} c\right)_{e}\right)^{\tau} \circ_{k} \left(P^{0}_{l+j}(\circ_{z'_{f}-f+1}c)_{f}\right)^{\tau'}\right)^{\chi\inv}
\end{multline}
 In this case, we obtain
\begin{multline}
-\partial(L_{n}^{r}) = \sum_{\substack{i+j=r\\ k+l=n+1 \\ k+i \geq 2 \\ l+j \geq 2}}\sum_{\substack{\chi \in Sh_{k-1,l}}} (-1)^{\chi}(-1)^{(k-1)l}\\
 \sum_{\tau \in \Ss_{k},\tau'\in \Ss_{l}}\sum_{ \substack{1 \leq z_{1} < \ldots < z_{i} \leq k+i \\ 1 \leq z'_{1} < \ldots < z'_{j} \leq j+l}}   (-1)^{\tau+\tau'+(\fromto{z}{i})+(\fromto{z'}{j})}\left(\left(P_{k+i}^{0}\left(\circ_{z_{e}-e+1} c\right)_{e}\right)^{\tau} \circ_{k} \left(P^{0}_{l+j}(\circ_{z'_{f}-f+1}c)_{f}\right)^{\tau'}\right)^{\chi\inv}
\end{multline}
By applying the definition of $L_n$, this proves the lemma.

We show that equation (\ref{signs}) holds: the set $y_{1}< \ldots < y_{r}$ splits into three subsets 
\begin{itemize}
\item $z_{1}^{1}< \ldots < z_{i_{1}}^{1} \text{ such that } z^{1}_{t} < x$,
\item $z_{1}^{2}< \ldots < z_{i_{2}}^{2} \text{ such that } z^{2}_{t} > x+q-1$,
\item $z_{1}^{'0}< \ldots < z_{j}^{'0} = \{y_{1}< \ldots < y_{r}\} \setminus \{ z_{1}^{1}< \ldots < z_{i_{1}}^{1}, z_{1}^{2}< \ldots < z_{i_{2}}^{2}\}$.
\end{itemize} 
from which we define
\begin{align*}
&z_{t} := \begin{cases} z^{1}_{t} & t \leq i_{1} \\ z^{2}_{t-i_{1}}-q & t> i_{1} \end{cases} &z'_{t} := z^{'0}_{t} - x+1\\
&i := i_{1} +i_{2} \text{ and } j:= r-i, &k:= p- i \text{ and } l:=q-j
\end{align*} 
We then compute 
\begin{equation}
P_{p}^{0}\circ_{x} P^{0}_{q} \left(\circ_{y_{e}-e+1}  c \right)_{e} = (-1)^{i(l+j)+i_{2}j} \left(P_{k+i}^{0}\left(\circ_{z_{e}-e+1} c\right)_{e}\right)\circ_{x-i_{1}} \left(P^{0}_{l+j}(\circ_{z'_{f}-f+1}c)_{f}\right)
\end{equation}
where the sign appears because we move the $c$'s corresponding to $z_{t}^{'0}$ past $j$ $c$'s corresponding to $z^{'0}_{t}$ and also $i$ $c$'s past $P_{q}^{0}$. Note that if we know $x$ and $i_{1}$, then given $i,j,k,l,\fromto{z}{i},\fromto{z'}{j}$ we can uniquely determine $p,q,r,n,\fromto{y}{r}$.

We also compute the sign of $\fromto{y}{r}$: let $\te$ and $\te'$ be the shuffles such that
\begin{align*}
&\fromto{z^{1}}{i_{1}} \overset{\te}{\leadsto} 1,\ldots,i_{1}\\
&\fromto{z^{'0}}{j} \overset{\fromto{z'}{j}}{\leadsto} x,\ldots,x+j \overset{j(x-i_{1}-1)}{\leadsto} i_{1}+1,\ldots,i_{1}+j\\
&\fromto{z^{2}}{i_{2}} \overset{\te'}{\leadsto} x+q,\ldots,x+q+i_{2}-1 \overset{i_{2}(l+x-i_{1}-1)}{\leadsto} i_{1}+j+1,\ldots,r
\end{align*}
then we obtain that
$$(-1)^{\fromto{y}{r}} = (-1)^{\te + (\fromto{z'}{j}) + j(x-i_{1}-1) + \te' + i_{2}(l+x-i_{1}-1)}$$
On the other hand, 
\begin{align*}
&\fromto{z}{i_{1}},z_{i_{1}+1},\ldots,z_{i} \overset{\te+\te'}{\leadsto} 1,\ldots,i_{1},x+1,\ldots,x+i_{2} \overset{i_{2}(x-i_{1})}{\leadsto} 1,\ldots,i
\end{align*}
and thus we have 
$$(-1)^{\fromto{y}{r}} = (-1)^{(\fromto{z}{i})+ (\fromto{z'}{j}) + j(x-i_{1}-1) +  i_{2}(l+1)}$$

Given $\si \in \Ss_{n}$, let $b_{1}<\ldots<b_{l}\sub \lh k+l-1\rh$ such that $\si(b_{t}) \in \{x-i_{1},\ldots,x-i_{1}+l-1\}$ and let $a_{1} <\ldots < a_{k_1} := \lh k+l-1 \rh \setminus \{b_{1}< \ldots < b_{l} \}$. We then define the $(k-1,l)$-shuffle $\chi = (a_{1}< \ldots < a_{k-1},b_{1} < \ldots < b_{l})$ and 
$$ \tau(t) = \begin{cases} \si \chi (t) & t< k \\ x-i_{1} & t = k\end{cases} \text{ and } \tau'(t) = \si(b_{t}) - (x-i_{1}-1) $$ 
It is then easy to see that 
$$\left(P_{k+i}^{0}\left(\circ_{z_{e}-e+1} c\right)_{e}\circ_{x-i_{1}} \left(P^{0}_{l+j}(\circ_{z'_{f}-f+1}c)_{f}\right)\right)^{\si\chi} = \left(P_{k+i}^{0}\left(\circ_{z_{e}-e+1} c\right)_{e}\right)^{\tau}\circ_{k} \left(P^{0}_{l+j}(\circ_{z'_{f}-f+1}c)_{f}\right)^{\tau'}$$
Note that $\tau$ and $\fromto{z}{i}$ determine $x$ and $i_{1}$ uniquely. Hence, given $\fromto{z}{i},\fromto{z'}{j},k,l$ we can uniquely determine $\fromto{y}{r},p,q,n,r$ in the above manner.
In order to show that equation (\ref{signs}) holds, we only need to verify the corresponding signs: let $\tau_{0}$ be the permutation such that 
\begin{align*}
&\si \chi(1),\ldots,\si\chi(k-1) \overset{\tau_{0}}{\leadsto} 1,\ldots,x-i_{1}-1,x-i_{1}+l,\ldots,k+l-1 \\
&\si\chi(k),\ldots,\si\chi(k+l-1) \overset{\tau'}{\leadsto} x-i_{1},\ldots,x-i_{1}+l
\end{align*}
and thus we have that $(-1)^{\si+\chi} = (-1)^{\tau_{0}+\tau'+l(k-x+i_{1})}$.
On the other hand, we have that $\tau$ corresponds to
\begin{align*}
\tau(1),\ldots,\tau(k) \overset{\tau_{0}}{\leadsto} 1,\ldots,x-i_{1}-1,x-i_{1}+1,\ldots,k-1,\tau(k)\overset{k-x+i_{1}}{\leadsto} 1,\ldots,k-1
\end{align*}
and thus we have
$$(-1)^{\si + \chi } = (-1)^{\tau + \tau' + (l+1)(k-x+i_{1})}$$
As such, we can compute 
\begin{align*}
&(-1)^{(\fromto{z}{i})+(\fromto{z'}{j})}(-1)^{\chi+\tau+\tau'}(-1)^{(k-1)l}(-1)^{i(l+j)+i_{2}j}\\
 &= (-1)^{\fromto{y}{r} + \si } (-1)^{i_{2}(l+1) + j(x-i_{1}-1)+(l+1)(k-x+i_{1}) +(k-1)l+ i(l+j)+i_{2}j}
\end{align*}
and 
\begin{align*}
&i_{2}(l+1) + j(x-i_{1}-1)+(l+1)(k-x+i_{1}) +(k-1)l+ i(l+j)+i_{2}j\\
 &= (l+1)(k-x+i) + j(x-i-1)+ (k-1)l+i(l+j) \\
&= (l+1)(p-x) + j(p-x+(k-1))+ (k-1)l+iq \\
&= (q-1)(p-x)+(k-1)q + iq\\
&= (q-1)(p-x)+(p-1)q
\end{align*}
which completes the proof.
 \end{proof}
 
 \begin{lemma}
$L_{n}$ are skew symmetric and $\partial'$ is a differential making $\Quiltb[[c]]$ into a dg-operad.
 \end{lemma}
 \begin{proof}
It is clear from the definition of $L_{n}^{r}$ that they are skew symmetric and thus also $L_{n}$.

Per definition $\partial_{L_{1}}$ is a derivation by construction and so is $\partial$, and thus so is $\partial'$.

It is clear from the definition of $c$ that $\partial'(c)= 0$. Hence, we only need to show that $\partial ' \partial'(Q) = 0$ for every $Q\in \Quilt$. Using lemma \ref{Lvanish} and \ref{mainlemma} we first prove that $-\partial L_{1} = L_{1} \circ_{1} L_{1}$. Namely, we compute
\begin{align}
-\partial( L_{1} )= \sum_{r \geq 1}\sum_{\substack{i+j=r \\ k+l=2 \\ k+i \geq 2 \\ l+j \geq 2}}\sum_{\substack{\chi \in Sh_{k-1,l}}}(-1)^{\chi}(-1)^{(k-1)l}(L_{k}^{i} \circ_{k} L_{l}^{j})^{\chi\inv} =\sum_{\substack{i+j=r \\ i,j \geq 1 \\ r\geq 1}}L_{1}^{i} \circ_{1} L_{1}^{j} = L_{1} \circ_{1} L_{1} \label{Lone}
\end{align}
Now let us compute
$$\partial' \partial ' Q = \partial \partial Q + \partial \partial_{L_{1}} Q + \partial_{L_{1}} \partial Q + \partial_{L_{1}}\partial_{L_{1}} Q$$
we compute separately 
$$\partial_{L_{1}} \partial Q = L_{1} \circ_{1} \partial Q - (-1)^{|Q|-1} \sum_{i} \partial Q \circ_{i} L_{1}$$
and 
\begin{align*}
\partial \partial_{L_{1}} Q &= \partial (L_{1} \circ_{1} Q) - (-1)^{|Q|} \sum_{i} \partial (Q\circ_{i} L_{1}) \\
&= \partial L_{1} \circ_{1} Q + (-1)^{|Q|} L_{1} \circ_{1} \partial Q - (-1)^{|Q|} \sum_{i} \partial Q \circ_{i} L_{1} - (-1)^{|Q|-1} \sum_{i} \partial L_{1}
\end{align*} 
adding them gives
\begin{align*}
 \partial_{L_{1}} \partial Q + \partial \partial_{L_{1}} Q = \partial L_{1} \circ_{1} Q - (-1)^{|Q|-1}\sum_{i} Q \circ_{i} \partial L_{1}  
 \end{align*}
Next, we compute $\partial_{L_{1}}\partial_{L_{1}} Q$. As $L_{1}$ has only a single input and has degree $-1$, we have for $i\neq j$ that $Q \circ_{i} L_{1} \circ_{j} L_{1} = - Q \circ_{j} L_{1} \circ_{i} L_{1}$. Hence, by also using using equation (\ref{Lone}), we obtain
\begin{align*}
\partial_{L_{1}}\partial_{L_{1}} Q &= L_{1} \circ_{1} (L_{1} \circ_{1} Q) - (-1)^{|Q|} \sum_{i} L_{1} \circ_{1} (Q \circ_{i} L_{1}) + (-1)^{|Q|} \sum_{i} (L_{1} \circ_{1} Q) \circ_{i} L_{1}\\
&  -  \sum_{i,j} Q \circ_{i} L_{1} \circ_{j} L_{1}\\
&= L_{1} \circ L_{1} \circ_{1} Q - \sum_{i} Q\circ_{i} (L_{1} \circ_{1} L_{1}) \\
&= \partial_{L_{1}} \partial Q + \partial \partial_{L_{1}} Q
\end{align*}
As $\partial \partial Q =0$ we thus obtain $\partial'\partial' Q= 0$.
 \end{proof}

\begin{proof}[Proof of Theorem \ref{maintheorem}]
We need to show for every $n\geq 2$ that the equation
$$ 0 = \partial'(L_{n}) +  \sum_{\substack{k+l = n+1 \\ k,l\geq 2}}  \sum_{\si \in Sh_{k-1,l}} (-1)^{(k-1)l} (-1)^{\si} (L_{k} \circ_{k} L_{l})^{\si\inv}$$
holds, which is equivalent to
$$ 0 = \partial(L_{n}) + \sum_{\substack{k+l = n+1 \\ k,l\geq 1}}  \sum_{\si \in Sh_{k-1,l}} (-1)^{(k-1)l} (-1)^{\si} (L_{k} \circ_{k} L_{l})^{\si\inv}$$
This is equivalent to showing for every $r\geq 0$ that the equation 
\begin{equation}
0 = \partial(L_{n}^{r}) + \sum_{\substack{i+j=r\\k+l = n+1 \\ k,l\geq 1 \\ (l,j)\neq (1,0) \neq (k,i) \\ (l,j) \neq (0,1) \neq (k,i)}}  \sum_{\si \in Sh_{k-1,l}} (-1)^{(k-1)l} (-1)^{\si} (L_{k}^{i} \circ_{k} L_{l}^{j})^{\si\inv}\label{eq1}
\end{equation} 
holds, which follows from lemma \ref{mainlemma} and $L^{i}_{0} = 0$ for $i\geq0$ (lemma \ref{Lvanish}).
\end{proof}

\begin{opm}
Under the natural morphism $\Quilt[c] \longrightarrow \mQuilt$ sending $c$ to $m$ our $\Linf$-structure corresponds to the $\Linf$-structure from \cite[Thm. 7.13]{hawkins}, that is, we have the commutative diagram
$$\begin{tikzcd}
                           & {\Quilt_{b}[[c]]} \arrow[rd] &                   \\
\Linf \arrow[ru] \arrow[r] & \mQuilt \arrow[r]            & {\Quilt_{b}[[m]]}
\end{tikzcd}$$
where $\Quilt_{b}[[m]]$ denotes the quotient of the operad of bounded powerseries by the relations on $m$ in $\mQuilt$.
\end{opm}

\subsection{The morphism $\Quiltb[[c]] \longrightarrow \End(s\CGS(\A))$}\label{parquiltbcCGS}

In this section, we make (the suspension of) $\CGS(\A)$ into a $\Quiltb[[c]]$-algebra.

The morphism $R: \Quilt \longrightarrow \End(s\CGS(\A))$ naturally extends to a morphism of graded operads $R_{c}:\Quilt[c] \longrightarrow \End(s\CGS(\A))$ by sending $c$ to $c\in \CGS(\A)$ as the axioms of $\Quilt[c]$ correspond respectively to $c$ being a natural transformation, the cocycle condition of $c$ and $c\in \textbf{C}^{2,0}(\A)$. Next, we will show that it further extends to the operad of bounded power series.
\begin{lemma}\label{finitesupport}
Let $\nth{\te}\in \CGS(\A)$ where $\te_{i} \in C^{p_{i},q_{i}}(\A)$ and $Q \in \Quilt[c](n,r)$ of degree $t$, then $R_{c}(Q)(\nth{\te}) = 0$ if $r > \sum_{i=1}^{n}q_{i} - t$.
\end{lemma}
\begin{proof}
This follows from $\CGS(\A)$ having only non-negative bidegree. Namely, given $Q = Q' \circ_{i_{1}} c \circ_{i_{2}} \ldots \circ_{i_{r}} c$ for a certain $Q'\in \Quilt(n+r)$, the bidegree of $R_{c}(Q')$ is $(\sum_{i=1}^{n}p_{i}+ 2r -(n+r-1), \sum_{i=1}^{n}q_{i} - deg(Q'))$.
As $deg(Q') = deg(Q) + r = t +r$, we have that 
$$\sum_{i=1}^{n}q_{i} - deg(Q') < 0 \iff r >\sum_{i=1}^{n}q_{i} - t$$
proving the lemma.
\end{proof}

\begin{theorem}\label{thmRc}
For $\A = (\A,m,f,c)$ a linear prestack over $\U$, the map 
$$R_{c}:(\Quiltb[[c]],\partial') \longrightarrow \End(s\CGS(\A),d)$$
defined as $R_{c}(Q) = R(Q)$ for $Q\in \Quilt$ and $R_{c}(c) = c$, is a morphism of dg-operads.
\end{theorem}
\begin{proof}
The representation $R_{c}$ extends to $\Quiltb[[c]]$ due to lemma \ref{finitesupport}.

We verify that $R_{c}\partial'Q = \partial_{d}R_{c}Q $ for $Q\in \Quilt$. It suffices to verify that $R_{c}L_{1}^{j} = d_{j}$ as $\partial$ corresponds to $\partial_{d_{0}} $. Let $\te\in \textbf{C}^{p,q}(\A)$ and $\si \in N_{p+j}(\U)$ and note that we write $|\te|$ for the degree $p+q$ of $\te$ in $\CGS(\A)$.

\proofpart{2}{$R_c L_1^1 = d_1$}

We can write $d_1 = \sum_{s=0}^{p+1}(-1)^{p+q+1+s}d_1^s$ where $d_1^s$ names the $s$th component of $d_1$. We write out the left-hand side
\begin{equation}\label{simplicial}
 R_{c}L^{1}_{1}(\te) = R(P^{0}_{2})(c,\te) - (-1)^{|\te|-1}R(P_{2}^{0})(\te,c) 
\end{equation}
where $P_{2}^{0} = (-1)^{\frac{2\cdot 1}{2} +1} (12,\;\begin{tikzpicture}
	\begin{pgfonlayer}{nodelayer}
		\node [style=none] (0) at (-1, 0) {$1$};
		\node [style=none] (1) at (1, 0) {$2$};
		\node [style=none] (2) at (0.525, 0) {};
		\node [style=none] (3) at (-0.55, 0) {};
	\end{pgfonlayer}
	\begin{pgfonlayer}{edgelayer}
		\draw (2.center) to (3.center);
	\end{pgfonlayer}
\end{tikzpicture}
\;)$. There exists two colorings $(X_0,J_0)$ and $(X_{p+1},J_{p+1})$ of the quilt $Q = (W,T)$ in $P_2^0(c,\te)$ given by the patchworks 
$$(X_1,J_1) = \quad \scalebox{0.85}{$\input{d_10.tikz}$} \qquad (X_{p+1},J_{p+1}) = \quad \scalebox{0.85}{$\input{d_1last.tikz}$} $$
which correspond respectively to $d_1^0$ and $d_1^{p+1}$. We verify their signs:  the sign $\sgn_Q(X_0,J_0)= \sgn_W(X_0) \sgn_T(J_0) (-1)^{ q}$ is determined by two shuffles
\begin{align*}
0_1 1_1 0_2 \ldots (q-1)_2 &\leadsto 0_1 1_1 0_2 \\
1_2 \ldots (p-1)_2 1_1 &\leadsto 1_1 1_2 \ldots (p-1)_2  
\end{align*}
and thus $\sgn_Q(X_0,J_0) = (-1)^{p-1 +q}$. In the second case, we have $\sgn(X_{p+1},J_{p+1}) = \sgn_{W}(X_{p+1}) \sgn_{T}(J_{p+1})(-1)^q $ determined by the two shuffles
\begin{align*}
0_1 1_1 0_2 \ldots (q-1)_2 &\leadsto 0_1 1_1 0_2 \\
1_1 1_2 \ldots (p-1)_2  &\leadsto 1_1 1_2 \ldots (p-1)_2  
\end{align*}
Hence, $\sgn_Q(X_{p+1},J_{p+1}) = (-1)^q$. 

There exists $p$ colorings $(X_i,J_i)_{i=1}^p$ of $P_{2}^{0}$ occuring in the second term of (\ref{simplicial}), given by the patchworks $$ (X_i,J_i) = \quad \scalebox{0.80}{$\input{d_1i.tikz}$}  $$
for $i= 1,\ldots,p$. They correspond to $d_{1}^{p+1-i}$ and we verify their signs: we have $\sgn_Q(X_i,J_i) = \sgn_W(X_i)\sgn_T(J_i)$ determined by the two shuffles
\begin{align*}
0_1 \ldots (q-1)_1 0_2 1_2 &\leadsto 0_1 \ldots (q-1)_1 0_2 1_2 \\
1_1 \ldots (i-1)_1 1_2 i_1 \ldots (p-1)_1 &\leadsto 1_1 \ldots (p-1)_1 1_2 
\end{align*}
Hence, we obtain the sign
$$-(-1)^{p+q-1} \sgn_Q(X_i,J_i) = (-1)^{q-i}$$

\proofpart{2}{$ R_c L_1^j = d_j$ }

First, we name the terms of $d_j$ and write 
$$d_j(\te)^{\si} := \sum_{\be \in \Ss_{q+1-j}} \sum_{r \in P(R_p(\si))} (-1)^{r+\be+q-j+1} d_j(\te)^{\si}(\be,r)$$
such that 
$$ d_j(\te)^{\si}(\be,r) \in \prod_{A \in \A(U_{p})^{q-j+2}} \Hom(\bigotimes_{i=1}^{q} \A(U_{p})(A_{i},A_{i-1}), \A(U_{0})(\si\hs A_{q-j+2},\si\st A_{0}))$$

Next, we note that the only non-vanishing term of $L_{1}^{j}$ is given by $\left(-1\right)^{\left(j-1\right)}P_{j+1}^{0} \circ_{1} c \circ_{2} c \circ_{2} \ldots \circ_{2} c$ as $c$ cannot have any children vertically. Moreover, the only non-vanishing quilts are of the form 
$$Q_{T}= \left(W',T'\right) =\left(123242\ldots 2\left(j+1\right)2, \; \input{Rc_tree.tikz}\;\right)$$
for any binary tree $T$ on $j-1$ vertices whose numbering is compatible with the word component $W$ of $Q_T$. Hence, the possible colorings of $Q_{T}$ are 
$$\left(X,J'\right) = \left(\quad \input{Rc_mtree_j.tikz} \; , \;\input{Rc_tree_j.tikz} \quad \right) \quad = \quad \scalebox{0.85}{$\input{Rc_Patchwork_j.tikz}$}$$
 for $t_{1}<\ldots<t_{j-1}$ in $\lh q \rh$ and $J$ an indexed tree coloring $T$.

As a result, we have that 
$$R_c L_1^j\left(\te\right) = \left(-1\right)^{\left(j-1\right) + \left(j-1\right)\left(p+q-1\right)} R P_{j+1}^0\left(c,\te,c,\ldots,c\right)$$
which sums over the terms 
$$\sgn_{Q_T}\left(X,J'\right) \left(-1\right)^{\frac{\left(j+1\right)j}{2} + 1}  \LL\left(X,J'\right)$$
 We finish the proof by showing that a formal path ${r}$ and a shuffle $\be\in \Ss_{q+1-j,j-1}$ correspond uniquely to such a binary tree $T$ and a coloring $\left(X,J'\right)$ of $Q_{T}$ such that 
$$\left(-1\right)^{r+\be+\left(q+1-j\right)} 
d_j\left(\te\right)^{\si}\left(\be,r\right) = \sgn_{Q}\left(X,J'\right)\left(-1\right)^{\frac{j\left(j+1\right)}{2} +1}\left(-1\right)^{\left(j-1\right) + \left(j-1\right)\left(p+q-1\right)}\LL\left(X,J'\right)\left(c,\te,c\ldots,c\right)\left(\te\right)^{\si}$$
Given a formal path $r=\left(\fromto{r}{j-1}\right)= \left(\left(\tau_1,i_1\right),\ldots,\left(\tau_{j-1},i_{j-1}\right)\right)$ we first define $T$ and its coloring $J$ as trees with vertex set $\{3,\ldots,j+1\}$ inductively:
\begin{itemize}
\item In the degenerate case $j=2$, $r$ is uniquely determined and we set $T$ to be the one vertex tree and $J$ the empty function.
\item  For $j>2$, let $(T_{0},J_{0})$ be the indexed tree corresponding to $\left((\tau_1,i_1),\ldots,(\tau_{j-2},i_{j-2})\right)$, in order to add vertex $j+1$ we have three cases 
\begin{enumerate}
\item if $i_{j-1}< i_{j-2}$, then set $j \tre_{J} (j+1)$ and start over with $(i_{j-1},i_{j+1})$
\item if $i_{j-1} = i_{j-2}$ or $i_{j-1} = i_{j-2}+1$, then let $(j,j+1)\in E_{T}$ and set resp. $J(j,j+1)= 2$ or $= 1$. 
\item if $i_{j-1} > i_{j-2}+1$, then set $(j+1) \tre_J j$ and start over with $(i_{j-1},i_{j+1}-1)$
\end{enumerate}
which we can draw as follows
$$\scalebox{0.9}{$\input{ordering_c.tikz}$}$$
\end{itemize}
Clearly, this process is reversible: given $(T,J)$ we obtain a unique sequence $r\in P\left(R_p(\si)\right)$. By identifying the shuffle $\be \in Sh_{q+1-j,j-1}$ and $X$ via $t_{l} = \be(q-j+1+l)$, we clearly obtain that $\LL(X,I')(\te)^{\si} = d_j(\te)^{\si}(\be,r)$.

The remaining work is to verify the signs: $\sgn_{Q_T}(X,J')$ consists of three components
$$\sgn_{W'}(X) \sgn_{T'}(J') (-1)^{\si}$$
where $(-1)^{\si} = (-1)^q$ as it corresponds to the shuffle
$$0,q,0,\ldots,0,1,p-1,1,\ldots,1 \quad \leadsto \quad 0,1,q,p-1,0,1,\ldots,0,1.$$
The sign $\sgn_{W'}(X)$ corresponds to the shuffle transforming the word
$$\underbracket{0_{(j+1)}\ldots 0_{3} 0_1 1_{1}}_{(1)} \underbracket{0_{2}\ldots \left(\be(t+1)-1\right)_{2}1_{3}\be(t+1)_{2} \ldots 1_{(j+1)}\be(t+j-1)_{2}\ldots (q-1)_{2}}_{(2)},$$
for $t= q-j+1$, into the word
$$0_{1}\ldots (q_1-1)_{1}\ldots 0_{(j+1)}\ldots (q_{j+1}-1)_{(j+1)}.$$
We observe that shuffling the second part $(2)$  
$$\underbracket{0_{2}\ldots \left(\be(t+1)-1\right)_{2}} 1_{3} \underbracket{\be(t+1)_{2} \ldots } 1_{(j+1)} \underbracket{\be(t+j-1)_{2}\ldots (q-1)_{2}} \quad \leadsto \quad \underbracket{ 0_{2}\ldots (q_2-1)_2 } 1_3\ldots 1_{(j+1)}$$
almost corresponds to the shuffle $\be$. However, there is in every interval $\be(s)_2,\ldots,(\be(s+1)-1)_2$ exactly one element too many. We remedy this by moving $1_{(j-2)}$ one place to the right, then $1_{(j-3)}$ two place, and so on. As such, its corresponding sign is $(-1)^{\be + \sum_{i=1}^{j-2} i } = (-1)^{\be + \frac{(j-2)(j-1)}{2}}$. Next, we shuffle 
$$0_{(j+1)} \ldots 0_3 0_1 1_1 0_{2}\ldots (q_2-1)_2 1_3\ldots 1_{(j+1)}\quad \leadsto \quad 0_1 1_1 0_2 \ldots (q-1)_2 0_3 1_3 \ldots 0_{(j+1)} 1_{(j+1)}$$
whose sign is $(-1)^{(j-1)(q+2)}$. Hence, we obtain that 
$$\sgn_{W'}(X) = (-1)^{\be + \frac{(j-1)(j-2)}{2} + (j-1)q}.$$
Next, we determined $\sgn_{T'}(J')$ as the sign of the shuffle
$$ [C_J]1_1 1_2 \ldots (p-1)_2 \quad \leadsto \quad 1_1 1_2 \ldots (p-1)_2 1_3 \ldots 1_{(j+1)},$$
where $[C_J]$ denotes the word obtained from the indexed tree $J$. We will show that the sign corresponding to the shuffle $\chi: [C_J] \leadsto 1_3 \ldots 1_{(j+1)}$ is $(-1)^{r + j-1}$. As a consequence, we obtain that 
$$\sgn_{T'}(J') = (-1)^{r + j-1  + p (j-1)},$$
and thus that 
$$\sgn(X,J') =  (-1)^{\be + r +q +j -1} (-1)^{\frac{(j-1)j}{2} + (j-1)(p+q-1)}.$$
Hence, we have
$$ \sgn_{Q}(X,J')(-1)^{\frac{j(j+1)}{2} +1}(-1)^{(j-1) + (j-1)(p+q-1)} = (-1)^{\be + r + q -j +1 }$$
which completes the proof.

We compute $\chi$ inductively: we have that $[C_J] = A 1_{(j+1)} B$ for certain words $A$ and $B$ and let $\chi_0$ denote the shuffle $AB \leadsto 1_3 \ldots 1_j$. By induction we know that $(-1)^{\chi_0}= (-1)^{r_0 + j-2}$ for the formal path $r_0= (r_1,\ldots,r_{j-2})$ and $(-1)^{r} = (-1)^{r_0 + i_{j-1}}$ where $r_{j-1} = (\si, i_{j-1})$. Moreover, we have $(-1)^{\chi} = (-1)^{\chi_0 + |B|}$ where $|B|$ denotes the length of $B$. We determine $|B|$. First, observe that the sequence associated to two indexed trees 
$$\scalebox{1}{$\input{twotrees.tikz}$}$$
is respectively $1_{h4}1_{h3}$ and $1_{h3}1_{h4}$. Thus, let $3 = v_{1}<_{T} \ldots <_{T} v_{t} = j+1$ be the unique chain of vertices from the root of $T$ to $j+1$, then we can define the numbers
\begin{itemize}
\item $l$ as the the number of vertices to the right of vertex $j+1$ in $T$,
\item $k$ as the number of vertices in the above chain such that $J(v_{g},v_{g+1}) = 1$.
\end{itemize} 
We then easily compute that 
\begin{itemize}
\item the height $i_{j-1}$ of $r_{j-1}=(\partial_{i_{j-1}}\si,i_{j-1})$ is exactly $(1+l)+k$,
\item length of $B$ is $l+k$
\end{itemize}
and thus by induction we obtain
$$(-1)^{\chi} = (-1)^{\chi_{0} + |B|} = (-1)^{\chi_{0} + i_{j-1}+1} = (-1)^{{r}+(j-1)}$$
where the last equality follows from induction.
\end{proof}

\begin{opm}
In the case of presheaves, when looking at the subcomplex $\CGSnr(\A)$ of normalised and reduced cochains, the map $R_{c}$ factorises through $\Quilt_{b}[[m]]$ as $m$ is sent to the identity $1^{u,v}:u\st v\st = (vu)\st$. Only in the case of normalised and reduced cochains does the identity satisfy all the relations on $m$. 

Note however that our induced morphism $\mQuilt \longrightarrow \Quilt_{b}[[m]] \longrightarrow \End(s\CGSnr(\A))$ does not correspond to the $\mQuilt$-algebra structure \cite[Thm. 5.6]{hawkins}.

However, the resulting $\Linf$-algebra structure, in the case of presheaves, does correspond to the one obtained from \cite[Thm. 5.6, Thm. 7.13]{hawkins}. This is essentially due to the multiplication of the prestack being unital. A $\mQuilt$-algebra structure also induces a Gerstenhaber-algebra structure on cohomology \cite[Thm. 6.11]{hawkins}. Writing down the relevant quilts, it is also easy to see that both $\mQuilt$-algebra structures (ours and \cite[Thm. 5.6]{hawkins}) induce the same Gerstenhaber-algebra structure on cohomology. 
\end{opm}

\subsection{Deformations of prestacks}\label{pardef}

Let $k$ be a field of characteristic $0$. Let $(\mathbf{C}, d, (L_n)_{n \geq 2})$ be an $L_{\infty}$-algebra. By definition, the Maurer-Cartan equation for $\te \in \mathbf{C}$ is given by
$$MC(\te) = d(\te) + \sum_{n=1}^{\infty}\frac{1}{n!} L_{n}(\te,\ldots,\te).$$
We consider the set of (degree $1$) Maurer-Cartan elements ${MC}(\mathbf{C}) = \{ \te \in \mathbf{C}^1 \,\, | \,\, MC(\te) = 0\}$ and for the appropriate notion of gauge equivalence (see \cite{kontsevich}), we consider
$$\underline{MC}(\mathbf{C}) = \{ \te \in \mathbf{C}^1 \,\, | \,\, MC(\te) = 0\} /\sim.$$
This gives rise to a functor $\underline{MC}_{\mathbf{C}}: \mathrm{Art} \longrightarrow \mathsf{Set}: (R, \mathfrak{m}) \longmapsto \underline{MC}(\mathfrak{m} \otimes \mathbf{C})$ on the category $\mathrm{Art}$ of Artin local $k$-algebras.

Consider the GS-complex $(\mathbf{C}_{GS}(\A), d)$ of a prestack $(\A,(m+f+c))$.
In \cite[Thm. 3.19]{DVL} it is shown that normalised reduced $2$-cocycles in $\mathbf{C}_{GS}(\A)$ correspond to first order deformations of $\A$, that is, deformations in the direction of $R = k[\epsilon]$. More precisely, for $(m', f', c')\in \CGS^{2}(\A)$, we have that $(\A[\epsilon],m+m'\eps,f+f'\eps,c+c'\eps)$ is a first-order deformation of $(\A,m,f,c)$ if and only if $d(m', f', c')=0$ and $(m', f', c')$ is normalised and reduced.
Further, it is shown in loc. cit. that the cohomology of the GS-complex classifies the first-order deformations of $\A$ up to equivalence. 

Putting Theorems \ref{maintheorem} and \ref{thmRc} together, $s\CGS(\A)$ is endowed with an $L_{\infty}$-structure which can be used to obtain higher order versions of these results.

For normalised, reduced cochains, it will be convenient to express the MC-equation in terms of the unsymmetrised components $(P_{n})_{n\geq1}$ from Definition \ref{P}.
The following characteristic-free expression of the MC-equation should be seen as the counterpart of the equation $MC(\te) = d(\te) + \te\{ \te\}$ for the first brace operation (or ``dot product'') on the Hochschild complex of an algebra.
Note that we omit writing $R_{c}$ and consider everything as elements of $\End(s\CGS(\A))$.

\begin{prop}
For a reduced and normalised cochain $\te = (m',f',c') \in s\CGS^{1}(\A)$, we have $$MC(\te) =  d(\te) + P_{1}(\te) + P_{2}(\te,\te) + P_{3}(\te,\te,\te,\te) + P_{4}(\te,\te,\te,\te).$$
 In particular, the MC-equation is quartic.
\end{prop}
\begin{proof}
First note that as $|\te| =1$, that $Q^{\si}(\te,\ldots,\te) = (-1)^{\si}Q(\te,\ldots,\te)$. Moreover, as we can consider $c\in \textbf{C}^{2,0}(\A)$, the MC-equation consists of $d_0( \te')$ and summations of $Q(\nth{\te})$ for quilts $Q$ of degree $n-2$ and $\te_{i} \in \CGS^{2}(\A)$. 
Let $Q=(W,T)$ be such a quilt, part of some $P_{n}^{0}$, then $W= 12\ldots2$. For $n\geq 3$, we know that $3$ is also a child in $T$ of $1$. As $1$ can have at most two children in $T$, for $n \geq 4$, we thus have $4$ and $3$ are children of $2$ in $W$, i.e. $W= 12324\ldots2$.
$$Q = \input{MC_1.tikz}$$
 In this case, only the elements $c'$ or $c$, and $m'$ can be inserted in $Q$ respectively in $1$ and $2$, with $2$ a child of $1$ in $T$. As $c'$ is reduced, this means that $Q(\te,\ldots,\te) =0$ for $n\geq 4$.

As $c$ is not necessarily reduced, more quilts are possible. Note that we write $\te[i]$ to refer to the element $\te$ inserted in vertex $i$. The above reasoning still applies and we can once more apply this reasoning to obtain for $n\geq 6$ that $W= 12324546\ldots42$ and $1$ and $3$ have two children in $T$.
$$Q = \input{MC_2.tikz}$$
This means that $2$ and $4$ need be inserted by $m'$ and hence $3$ and $1$ by either $c'$ or $c$. In case either one is $c'$, we already know it is zero as $c'$ is reduced. Thus, consider vertex $1$ and $c$ inserted by $c$, then, as $m'[4]$ is a child of $c[3]$ in $T$, $c[3]$ is the unit of the corresponding category $\A(U)$ and it is plugged into $m'[2]$ which is normalised, whence we obtain that $Q(\te_{1},\ldots,\te_{n}) =0$ for $n \geq 6$. Hence, $P_n^0(\nth{\te}) =0$ for $n \geq 6$.

Combining the above reasonings, we have that $P_{n}(\te,\ldots,\te) = 0$ for $n> 5$. In the case $n=5$ there has to be at least one $c$ present and thus the non-zero terms are contained in $P_{4}$.
\end{proof}

\begin{prop}\label{deformation}
For a prestack $(\A,m,f,c)$ and $\te = (m', f', c') \in s\CGS^1(\A)$, we have that $(\A,m+m',f+f',c+c')$ is a prestack if and only if $MC(\te) = 0$ and $\te$ is normalised and reduced.\end{prop}

\begin{proof}
Using the fact that $\te= (m', f', c') \in \CGS^{2}(\A)$ is reduced and normalised, we compute $MC(\te)$ and look at each component $MC(\te)_{[p,q]} \in \textbf{C}^{p,q}(\A)$ for $p+q=3$. We will use that $c^{1,-} =1 = c^{-,1}$, $m$ is unital and that $(m,f,c)$ satisfy the axioms of a prestack. Note that for a cochain $\te= (\te^\si(A))_{\si,A}$, we omit writing the set of objects $A$ explicitly where possible in order to lighten the equations below.

For $(p,q) = (0,3)$, let $ U \in \U$ and $A=(A_0,A_1,A_2,A_3)$ objects in $\A(U)$, then we compute
\begin{align*}
MC(\te)_{[0,3]}^{U}(A) &= d\te_{[0,3]}^{U}(A) +P_{2}(\te,\te)_{[0,3]}^{U}(A) + P_{3}(\te,\te,\te)_{[0,3]}^{U}(A) + P_{4}(\te,\te,\te,\te)_{[0,3]}^{U}(A) \\
&= d_0(m') + P_3^0(c,m',m') \\
 &=  m^{U}\circ( 1^{U} ,  m^{'U}) - m^{'U} \circ ( m^{U} , 1^{U}) + m^{'U} \circ ( 1^{U} ,  m^{U} ) - m^{U} \circ ( m^{'U} , 1^{U})  \\
 &- m^U \circ ( c^{1_U,1_U}(A_0) , m^{'U} \circ (m^{'U} , 1^U) ) +  m^U \circ ( c^{1_U,1_U}(A_0) , m^{'U} \circ (1^U, m^{'U} ) ) \\
 &= (m^U+m^{'U})\circ \left(1^U, (m^U+m^{'U})\right) - (m^U+m^{'U}) \circ \left( (m^U+m^{'U}), 1^U\right) 
\end{align*}

For $(p,q) = (1,2)$, let $u:U_0 \rightarrow U_1$ in $\U$ and $A=(A_0,A_1,A_2)$ objects in $\A(U_1)$, then we compute
\begin{align*}
d\te_{[1,2]}^u(A) &= d_0(f')^u(A)+ d_1(m')^u(A)\\
 &= m^{U_0} \circ (f^u , f^{'u}) - f^{'u} \circ m^{U_1} + m^{U_0} \circ (f^{'u},f^u) 
 + m^{U_0} \circ \left( c^{u,1_{U_0}}(A_0) , m^{'U_0} \circ ( f^u, f^u) \right) \\
&- m^{U_0} \circ \left( c^{1_{U_1}, u}(A_0), f^u \circ m^{'U_1} \right) 
\end{align*}
\begin{align*}
P_{2}(\te,\te)_{[1,2]}^u(A) &= P_2^0(f',f')^u(A) + P^0_3(c,m',f')^u(A) + P^0_3(c,f',m')^u(A) \\
&=m^{U_0} \circ ( f^{'u}, f^{'u}) + m^{U_0} \circ \left( c^{u,1_{U_0}}(A_0), m^{'U_0} \circ (f^{'u}, f^u) \right) + m^{U_0} \circ \left( c^{u,1_{U_0}}(A_0), m^{'U_0}(f^u,f^{'u}) \right) \\
&- m^{U_0}\circ \left( c^{1_{U_1},u}(A_0), f^{'u}\circ m^{'U_1} \right)
\end{align*}
\begin{align*}
P_{3}(\te,\te,\te)_{[1,2]}^u(A) &= P^0_4(c,m',f',f')^u(A) =  m^{U_0}\circ \left( c^{u,1_{U_0}}(A_0),m^{'U_0}\circ(f^{'u},f^{'u}) \right)\\
P_{4}(\te,\te,\te,\te)_{[1,2]}^u(A) &= 0
\end{align*}
and thus 
$$MC(\te)_{[1,2]}^u(A) = -(f^u+f^{'u}) \circ (m^{U_0}+m^{'U_0}) + (m^{U_0}+m^{'U_0})\circ \left((f^u+f^{'u})\otimes (f^u+f^{'u})\right)$$

For $(p,q)= (2,1)$, let $\si=(U_0 \overset{u_1}{\rightarrow} U_1 \overset{u_2}{\rightarrow} U_2)$ be a $2$-simplex in $\U$ and $A=(A_0,A_1)$ objects in $\A(U_2)$, then we compute 
\begin{align*}
d\te_{[2,1]}^\si(A) &= d_0(c')^\si(A) + d_1(f')^\si(A) + d_2(m')^\si(A)\\
&= m^{U_0} \circ \left( f^{u_2u_1} , c^{'\si}(A_1) \right) - m^{U_0}\circ \left( c^{'\si}(A_0) , f^{u_2} \circ f^{u_2} \right) - m^{U_0}  \circ \left( c^\si(A_0) , f^{u_1} \circ f^{'u_2} \right) \\
&+ m^{U_0} \circ \left(f^{'u_2u_1} ,c^\si(A_1) \right) - m^{U_0} \circ \left( c^\si(A_0), f^{'u_1}\circ f^{u_2} \right)\\
&- m^{U_0} \circ \left( c^{1_{U_0},u_2u_1}(A_0), m^{'U_0} \circ \left( c^{u_1,u_2}(A_0) , f^{u_2} \circ f^{u_1} \right) \right)\\
&+ m^{U_0} \circ \left( c^{1_{U_0},u_2u_1}(A_0) , m^{'U_0} \circ \left(f^{u_2u_1} , c^{u_1,u_2}(A_1) \right) \right) 
\end{align*}
\begin{align*}
P_{2}(\te,\te)_{[2,1]}^\si(A) &= P_2^0(f',c')^\si(A) + P_2^0(c',f')^\si(A) + P_3^0(c,f',f')^\si(A) + P_3^0(c,m',c')^\si(A)\\
&+  P_4^0(c,m',f',c)^\si(A)  + P_4^0(c,m',c,f')^\si(A)   \\
&= m^{U_0}\circ\left( f^{'u_2u_1}, c^{'u_1,u_2}(A_1)\right) - m^{U_0} \circ \left( c^{'u_1,u_2}(A_0) , f^{'u_1}\circ f^{u_2} \right) - m^{U_0} \circ \left( c^{'u_1,u_2}(A_0), f^{u_1}\circ f^{'u_2} \right)\\
&- m^{U_0} \circ \left( c^{u_1,u_2}(A_0), f^{'u_1} \circ f^{'u_2} \right) + m^{U_0} \circ \left( c^{1_{U_0},u_2u_1}(A_0), m^{'U_0} \circ \left( f^{u_2u_1}, c^{'u_1,u_2}(A_1) \right) \right) \\
&- m^{U_0} \circ \left( c^{1_{U_0}, u_2u_1}(A_0), m^{'U_0} \circ\left( c^{'u_1,u_2}(A_0), f^{u_1} \circ f^{u_2} \right) \right)\\
&+m^{U_0} \circ \left( c^{1_{U_0},u_2u_1}(A_0), m^{'U_0}\circ \left( f^{'u_2u_1}, c^{u_1,u_2}(A_1) \right) \right) \\
&- m^{U_0} \circ \left( c^{1_{U_0}, u_2u_1}(A_0), m^{'U_0}\circ\left( c^{u_1,u_2}(A_0), f^{u_1} \circ f^{'u_2} \right) \right) \\
&- m^{U_0} \circ \left( c^{1_{U_0}, u_2u_1}(A_0), m^{'U_0}\circ \left( c^{u_1,u_2}(A_0), f^{'u_1} \circ f^{u_2} \right)\right)
\end{align*}
\begin{align*}
P_{3}(\te,\te,\te)_{[2,1]}^\si(A) &= P_3^0(c',f',f')^\si(A) + P_4^0(c,m',f',c')^\si(A) +  P_4^0(c,m',c',f')^\si(A) + P_5^0(c,m',c,f',f')^\si(A)  \\
&=- m^{U_0} \circ \left( c^{'u_1,u_2}(A_0) , f^{'u_1} \circ f^{'u_2} \right) +m^{U_0} \circ \left( c^{1_{U_0},u_2u_1}(A_0), m^{'U_0} \circ \left( f^{'u_2u_1}, c^{'u_1,u_2}(A_1) \right) \right)\\
&-m^{U_0} \circ \left( c^{1_{U_0}, u_2u_1}(A_0), m^{'U_0}\circ\left( c^{'u_1,u_2}(A_0), f^{'u_1} \circ f^{u_2} \right) \right)\\
&-m^{U_0} \circ \left( c^{1_{U_0}, u_2u_1}(A_0), m^{'U_0}\circ \left( c^{'u_1,u_2}(A_0), f^{u_1} \circ f^{'u_2} \right)\right)\\
&- m^{U_0} \circ \left( c^{1_{U_0}, u_2u_1}(A_0), m^{'U_0} \circ\left( c^{u_1,u_2}(A_0), f^{'u_1} \circ f^{'u_2} \right) \right)
\end{align*}
\begin{align*}
P_4(\te,\te,\te,\te)_{[2,1]}^\si(A) &= P_5^0(c,m',c',f',f')^\si(A) = -m^{U_0} \circ \left( c^{1_{U_0}, u_2u_1}(A_0), m^{'U_0}\circ \left( c^{'u_1,u_2}(A_0), f^{'u_1} \circ f^{'u_2} \right) \right)
\end{align*}
and thus 
\begin{align*}
 MC(\te)_{[2,1]}^\si(A) &=\left( m^{U_0} + m^{'U_0} \right) \circ \left( \left( f^{u_2u_1} + f^{'u_2u_1} \right) , \left( c^{u_1,u_2}(A_0) + c^{'u_1,u_2}(A_0) \right) \right) \\
 &- \left( m^{U_0} + m^{'U_0} \right) \circ \left( \left( c^{u_1,u_2}(A_0) + c^{'u_1,u_2}(A_0) \right) , \left( f^{u_1} + f^{'u_1} \right) \circ \left( f^{u_2} + f^{'u_2} \right) \right)
\end{align*}

Lastly, for $(p,q) = (3,0)$, let $\si=(U_0 \overset{u_1}{\rightarrow} U_1 \overset{u_2}{\rightarrow} U_2 \overset{u_3}{\rightarrow} U_3)$ be a $3$-simplex in $\U$ and $A_0$ an object in $\A(U_2)$, then we compute
\begin{align*}
d\te_{[3,0]}^\si(A_0) &= d_3(m')^\si(A_0) + d_2(f')^\si(A_0) + d_1(c')^\si(A_0)\\
&= m^{U_0} \circ \left( c^{1_{U_0},u_3u_2u_1}(A_0), m^{'U_0} \circ \left( c^{u_2u_1,u_3}(A_0), c^{u_1,u_2}(u_3\st A_0) \right) \right) - m^{U_0} \circ \left( c^{u_1,u_3u_2}(A_0), f^{'u_1} \circ c^{u_2,u_3}(A_0) \right) \\
&- m^{U_0} \circ \left( c^{u_1,u_3u_2}(A_0), f^{u_1} \circ c^{'u_2,u_3}(A_0) \right) + m^{U_0} \circ \left( c^{u_1,u_3u_2}(A_0), c^{'u_1,u_2}(A_0) \right)  \\
&+ m^{U_0} \circ \left( c^{'u_2 u_1,u_3}(A_0), c^{u_1,u_2}(u_3\st A_0) \right) - m^{U_0} \circ \left( c^{u_1,u_3u_2}(A_0), f^{u_1} \circ c^{'u_2,u_3}(A_0) \right)
\end{align*}
\begin{align*}
P_{2}(\te,\te)_{[3,0]}^\si(A_0) &= P_2^0(c',c')^\si(A_0) + P_3^0(c,f',c')^\si(A_0) + P_3^0(c',f',c)^\si(A_0) +P_4^0(c,m',c',c)^\si(A_0) +P_4^0(c,m',c,c')^\si(A_0)\\
&+ P_5^0(c,m',c,f',c)^\si(A_0)\\
&=m^{U_0} \circ \left( c^{'u_2u_1,u_3}(A_0) , c^{'u_1,u_2}( u_3\st A_0) \right) - m^{U_0} \circ \left( c^{'u_1,u_2}(A_0), f^{u_1} \circ c^{'u_2,u_3}(A_0) \right) \\ 
& - m^{U_0} \circ \left( c^{u_1,u_2}(A_0), f^{'u_1} \circ c^{'u_2,u_3}(A_0) \right)  - m^{U_0} \circ \left( c^{'u_1,u_2}(A_0), f^{'u_1} \circ c^{u_2,u_3}(A_0) \right) \\ 
&+m^{U_0} \circ \left( c^{1_{U_0},u_3u_2u_1}(A_0), m^{'U_0} \circ \left( c^{'u_2u_1,u_3}(A_0), c^{u_1,u_2}( u_3\st A_0) \right) \right)\\ 
&-m^{U_0} \circ \left( c^{1_{U_0},u_3u_2u_1}(A_0), m^{'U_0} \circ \left( c^{'u_1,u_3u_2}(A_0), f^{u_1} \circ c^{u_2,u_3}( A_0) \right) \right)\\ 
&+m^{U_0} \circ \left( c^{1_{U_0},u_3u_2u_1}(A_0), m^{'U_0} \circ \left( c^{u_2u_1,u_3}(A_0), c^{'u_1,u_2}( u_3\st A_0) \right) \right)\\ 
&- m^{U_0} \circ \left( c^{1_{U_0},u_3u_2u_1}(A_0), m^{'U_0} \circ \left( c^{u_1,u_3u_2}(A_0), f^{u_1} \circ c^{'u_2,u_3}( A_0) \right) \right)\\
&-m^{U_0} \circ \left( c^{1_{U_0},u_3u_2u_1}(A_0), m^{'U_0} \circ \left( c^{u_1,u_3u_2}(A_0), f^{u_1} \circ c^{'u_2,u_3}( A_0) \right) \right)\\
&- m^{U_0} \circ \left( c^{1_{U_0},u_3u_2u_1}(A_0), m^{'U_0} \circ \left( c^{u_1,u_3u_2}(A_0), f^{'u_1} \circ c^{u_2,u_3}( A_0) \right) \right)
\end{align*}
\begin{align*}
P_{3}(\te,\te,\te)_{[3,0]} &= P_3^0(c',f',c')^\si(A_0) +  P_4^0(c,m',c',c')^\si(A_0)+ P_5^0(c,m',c',f',c)^\si(A_0) \\
&+P_5^0(c,m',c,f',c')^\si(A_0) \\
&= - m^{U_0} \circ \left( c^{'u_1,u_3u_2}(A_0) , f^{'u_1} \circ c^{'u_2,u_2}(A_0) \right)\\
&+m^{U_0} \circ \left( c^{1_{U_0},u_3u_2u_1}(A_0), m^{'U_0} \circ \left( c^{'u_2u_1,u_3}(A_0), c^{'u_1,u_2}( u_3\st A_0) \right) \right) \\
&- m^{U_0} \circ \left( c^{1_{U_0},u_3u_2u_1}(A_0), m^{'U_0} \circ \left( c^{'u_1,u_3u_2}(A_0), f^{u_1} \circ c^{'u_2,u_3}( A_0) \right) \right)\\
&- m^{U_0} \circ \left( c^{1_{U_0},u_3u_2u_1}(A_0), m^{'U_0} \circ \left( c^{'u_1,u_3u_2}(A_0), f^{'u_1} \circ c^{u_2,u_3}( A_0) \right) \right)\\
&- m^{U_0} \circ \left( c^{1_{U_0},u_3u_2u_1}(A_0), m^{'U_0} \circ \left( c^{u_1,u_3u_2}(A_0), f^{'u_1} \circ c^{'u_2,u_3}( A_0) \right) \right)
\end{align*}
\begin{align*}
P_{4}(\te,\te,\te,\te)_{[3,0]}^\si(A_0) &= P_5^0(c,m',c',f',c')^\si(A_0) =  - m^{U_0} \circ \left( c^{1_{U_0},u_3u_2u_1}(A_0), m^{'U_0} \circ \left( c^{u_1,u_3u_2}(A_0), f^{'u_1} \circ c^{'u_2,u_3}( A_0) \right) \right)
\end{align*}
and thus 
\begin{align*}
MC(\te)_{[3,0]}^\si(A_0) &= \left( m^{U_0} + m^{'U_0} \right) \circ \left( \left( c^{u_2u1_,u_3}(A_0) + c^{u_2u_1,u_3}(A_0) \right) , \left( c^{u_1,u_2}(u_3\st A_0) + c^{'u_1,u_2}(u_3\st A_0) \right) \right) \\
&-\left( m^{U_0} + m^{'U_0} \right) \circ \left( \left( c^{u1_,u_3u_2}(A_0) + c^{u_1,u_3u_2}(A_0) \right) , \left( f^{u_1} + f^{'u_2} \right) \circ \left( c^{u_2,u_3}( A_0) + c^{'u_2,u_3}(A_0) \right) \right)
\end{align*}
These computations show that $MC(\te) = 0$ if and only if $(\A,m+m',f+f',c+c')$ is a prestack. 
\end{proof}
\begin{opm}
Note that for the functor condition we only need the cubic part and for the twists the full quartic part of the equation.
\end{opm}

Based upon Proposition \ref{deformation}, with some more work taking gauge equivalence into account, one may extend \cite[Thm. 3.19]{DVL} to higher order deformations. Recall that for $(R, \mathfrak{m}) \in \mathrm{Art}$ an $R$-deformation of a $k$-linear prestack $(\A,m,f,c)$ is an $R$-linear prestack $(R \otimes_k \A, \bar{m}, \bar{f}, \bar{c})$ of which the algebraic structure reduces to that of $\A$ modulo $\mathfrak{m}$, and an equivalence of deformations is an isomorphism between the deformed prestacks which reduces to the identity morphism. Let $\mathrm{Def}_{\A}: \mathrm{Art} \longrightarrow \mathsf{Set}$ be the deformation functor of $\A$ with 
$\mathrm{Def}_{\A}(R, \mathfrak{m})$ the set of $R$-deformations of $\A$ up to equivalence of deformations.
The following theorem, of which the proof will appear elsewhere, expresses that the deformation theory of $\A$ is controlled by the $L_{\infty}$-algebra $s\CGS(\A)$.

\begin{theorem}
Let $(\A,m,f,c)$ be a prestack. There is a natural isomorphism of functors $\mathrm{Art} \longrightarrow \mathsf{Set}$:
 $$\mathrm{Def}_{\A} \cong \underline{MC}_{s\CGS(\A)}.$$
\end{theorem}

\appendix

\section{Generator-Free Description of the morphism $\NSOp \longrightarrow \Multitr$}\label{appgenfree}

In this appendix, we provide a generator-free description of the morphism $\NSOp \longrightarrow \Multitr$ from lemma \ref{NSOpDelta}. Although per construction we have a morphism of operads induced from the generators $E_i$ of $\NSOp$, we consider it valuable in practice to have an explicit definition.
 
For $I$ an indexed tree in $\NSOp(\nth{q};\sum_{i=1}^n q_i - n+1)$, we call $l_I =\sum_{i=1}^n q_i - n+1$ the \emph{number of leaves} of $I$. Given a number $m\in \N$, we call the interval
$$[m,m+l_I] = \{ m , m+1, \ldots, m+l_I \}$$
the \emph{numbering set of leaves} of $I$.

\begin{constr}\label{NSOpMultiD_alt}
For $I$ an indexed tree in $\NSOp(\nth{q};\sum_{i=1}^n q_i -n +1)$, let $I_j$ be the maximal subtree of $I$ with root $j$. Let $u$ be the root of $I$ with children $u_1 \tre_I \ldots \tre_I u_k$ which have index $i_j:= I(u,u_j)$, then $I$ decomposes as follows
$$I = \input{Decomposition_I.tikz}$$
Given a number $m\in \N$, we define a non-decreasing map $\zeta(I,m): [q_u] \longrightarrow [m,m+l_I]$ as follows
$$ \zeta(I,m)(t) = \begin{cases} m + t & \text{ for } 0 \leq t < i_1  \\
m + \sum_{j=1}^s l_{I_j} + t & \text{ for } i_s\leq t < i_{s+1} \\
m + \sum_{j=1}^k l_{I_j} + t & \text{ for } i_s \leq t \leq q_u \end{cases} $$
which determines where the leaves of the root of $I$ are placed in its numbering set of leaves $[m,m+l_I]$.

In order to define the tuple $\zeta_I = (\zeta_{I,1},\ldots,\zeta_{I,n}) \in \Multitr(\nth{q};\sum_{i=1}^nq_i - n + 1)$, we run inductively through the tree $I$ from root $u$ to leaves setting 
$$\zeta_{I,u} := \zeta(I,0):[q_u] \longrightarrow [l_I]$$
and for vertex $a$ with child $b$ and index $i:=I(a,b)$ we set
$$\zeta_{I,b} := \zeta(I_b, \zeta_{I,a}(i-1)): [q_b] \longrightarrow [ \zeta_{I,a}(i-1), \zeta_{I,a}(i-1) + l_{I_b}] \hookrightarrow [l_I]$$
with $l_I = \sum_{i=1}^nq_i - n + 1$.
\end{constr}
 
Note that construction \ref{NSOpDelta} corresponds to the above construction applied to the generators $E_i$. As such, if the generator-free description defines a morphism of operads, they coincide.
 
\begin{prop}
Construction \ref{NSOpMultiD_alt} defines a morphism of operads $\NSOp \longrightarrow \Multitr$.
\end{prop}
\begin{proof}
Let $I\in \NSOp(\nth{q}; q)$ and $I' \in \NSOp(\fromto{q'}{m}; q')$ and consider their composition $I'':= I \circ_i I'$. We will show that 
$$\zeta_{I''} = \zeta_I \circ_i \zeta_{I'}$$
As in construction \ref{NSOpMultiD_alt}, let $I_j, I'_j$ and $I''_j$ be the maximal subtrees of respectively $I,I'$ and $I''$ with root $j$. Due to equivariance, we can assume that a vertex $j$ belongs to the subtree $I_i$ iff $j\geq i$. We then have three cases to consider which we depict diagrammatically as follows 
$$\scalebox{0.9}{$\input{Triangle_i.tikz}$}
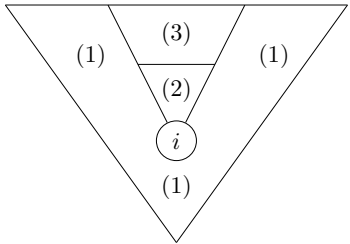
$$
\begin{enumerate}
\item For $j<i$, if we contract the subtree $I_i$ and $I''_i$ to a single vertex with number of leaves $l_{I_i} = l_{I''_i}$, we obtain the same indexed tree $I\setminus I_i = I'' \setminus I''_i$ in which $I_j$ and $I''_j$ coincide. Thus, it is easy to see that 
$$\zeta_{I'',j} = \zeta_{I,j}$$
\item If $j\geq i$ and $j$ lies in the image $\Image(I')$ of $I'$ in $I''$, then $I''_j$ consists of the subtree $I'_{j-i+1}$ with a sequence of subtrees $I_{t},\ldots,I_{t'}$ placed on top.
$$\scalebox{0.9}{$\input{triangle_case2.tikz}$}$$
 In this case, $\zeta_{I',j-i+1}$ determines where the leaves of $j$ are placed in $I'$ and $\zeta_{I,i}$ determines where the leaves of $i$ in $I$ are placed. As a result, we see that
$$\zeta_{I'',j} = \zeta_{I,i}\circ\zeta_{I',j-i+1}$$
which determines where the leaves of $j$ are put in $I''$.
\item If $j\geq i$ and $j$ does not lie in the image of $I'$ in $I''$, then $j-m+1$ lies in $I$ above $i$. In this case, the subtrees $I_{j-m+1}$ and $I''_j$ are equal. If the parent $a$ of $j$ in $I''$ does not lie in $\Image(I')$, then we clearly have 
$$\zeta_{I'',j} = \zeta_{I,j-m+1}.$$
 If $a$ lies in $\Image(I')$, then $\zeta_{I'',a}=\zeta_{I,i}\circ \zeta_{I',a-i+1}$ due to the previous case.
 $$\scalebox{0.9}{$\input{triangle_case3.tikz}$}$$ Moreover, the index $I''(a,j)$ then equals $I(i,j) - \zeta_{I',a-i+1}(0)$. Hence, we have that 
 $$\zeta_{I'',j} = \zeta_{I,j-m+1}$$
\end{enumerate}
\end{proof}

\def\cprime{$'$}
\providecommand{\bysame}{\leavevmode\hbox to3em{\hrulefill}\thinspace}
\bibliography{Bibfile}
\bibliographystyle{amsplain}
\end{document}